\documentclass{article}

\usepackage{geometry}
\usepackage[T1]{fontenc}
\usepackage[utf8]{inputenc}
\usepackage{lmodern}

\usepackage{amsmath,amssymb,amsthm,mathtools,bm,bbm,mathrsfs,needspace,mleftright,accents,url}
\usepackage[inline]{enumitem}
\usepackage[breaklinks]{hyperref}

\ifdefined\mleft
\else
 \let\mleft\left
 \let\mright\right
\fi

\ifdefined\mathbbm
\else
 \newcommand{\mathbbm}[1]{\mathbf{#1}}
\fi

\newtheoremstyle{mystyle}
{}{}{}{}{\bfseries}{}{ }
{\thmname{#1}\thmnumber{ #2}\thmnote{ (#3)}}

\theoremstyle{mystyle}
\newtheorem{theorem}{Theorem}[section]
\newtheorem{definition}[theorem]{Definition}
\newtheorem{lemma}[theorem]{Lemma}
\newtheorem{claim}[theorem]{Claim}
\newtheorem{remark}[theorem]{Remark}
\newtheorem{corollary}[theorem]{Corollary}

\newtheorem{notation}[theorem]{Notation}
\newtheorem{proposition}[theorem]{Proposition}

\allowdisplaybreaks

\newcommand{\pr}{\mathbb{P}}
\newcommand{\ev}{\mathbb{E}}

\newcommand{\indic}[1]{\mathbbm{1}{\mleft[\,{#1}\,\mright]}}

\newcommand{\quand}{\quad\text{and}\quad}
\newcommand{\ffrac}[2]{{\textstyle\frac{{#1}}{{#2}}}}

\newcommand{\Alower}{\underline{A}^*}
\newcommand{\Aupper}{\overline{A}^*}

\title{The Window Process of Slightly Subcritical Frozen Percolation}
\author{Bal\'azs R\'ath \thanks{Department of Stochastics, Institute of Mathematics, Budapest University of Technology and Economics,
M\H{u}egyetem rkp.\ 3., H-1111 Budapest, Hungary; HUN-REN--BME Stochastics Research Group, Budapest University of Technology and Economics, M\H{u}egyetem rkp.\ 3., H-1111 Budapest, Hungary;
HUN-REN Alfr\'ed R\'enyi Institute of Mathematics, Re\'altanoda utca 13-15, H-1053 Budapest,
Hungary. Email:~\href{mailto:rathb@math.bme.hu}{\tt rathb@math.bme.hu}, \href{mailto:szokemp@math.bme.hu}{\tt szokemp@math.bme.hu}.}
\and
M\'arton Sz\H{o}ke \footnotemark[1]
\and
Dominic Yeo \thanks{Department of Mathematics, King's College London, Strand, London WC2R 2LS, United Kingdom. E-mail: \href{mailto:dominic.yeo@kcl.ac.uk}{\tt dominic.yeo@kcl.ac.uk}.}
}

\begin{document}

\maketitle

\begin{abstract}
We study the frozen percolation model on the complete graph $K_n$, where each edge appears at rate~$1/n$ and connected components are destroyed by lightning strikes hitting vertices at rate $\lambda_n$ (thus the rate at which a connected component is destroyed is linear in its size). We focus on the slightly subcritical sub-regime of the self-organized critical (SOC) regime, assuming $(\ln n)^{4/3}n^{-1/3} \ll \lambda_n \ll 1$. Our main object of study is the so-called window process, which allows us to study the fine details of the mechanism that drives the SOC behavior of this dynamic random graph. We prove two scaling limit theorems. First, we show that the fluctuations of the window process around its deterministic fluid limit converge to a stationary Ornstein--Uhlenbeck process, whose mean-reverting property captures the SOC nature of the model. Second, by zooming in on a finer space-time scale, we prove that the scaled increments of the window process converge to a tempered stable L\'evy process.

\vspace{0.2cm}

\end{abstract}

\tableofcontents

\section{Introduction}

\textbf{Frozen percolation.}
The concept of frozen percolation was introduced by Aldous~\cite{A00}, motivated by \cite{S43}. In this model, i.i.d.\ \textit{activation times} that are uniformly distributed on~$[0,1]$ are assigned to the edges of an infinite, unoriented graph. Initially, all edges are \textit{closed}. At its activation time, an edge \textit{opens}, provided it is not \textit{frozen}.
Here, by definition, an edge freezes as soon as one of its end-vertices becomes part of an infinite open cluster. Aldous defined and proved the existence of this process on the \textit{infinite 3-regular tree}.

For general graphs, the existence of a process satisfying the above naive description is not obvious. Indeed, Benjamini and Schramm (see \cite[Section~5.1]{A00}) observed that on the square lattice~$\mathbb{Z}^2$, frozen percolation does not exist (see also an observation of J\'arai in~\cite[Section~3]{BT01} for an account of the argument) and it is currently open whether frozen percolation can be defined on $\mathbb{Z}^d$, $d \ge 3$.

\medskip
\textbf{Frozen percolation and forest fire model on $K_n$.}
In this paper, we study the \textit{finite volume mean-field version} of the frozen percolation model where the underlying graph is the complete graph $K_n$ on~$n$ vertices. This version of frozen percolation and the closely related forest fire model were defined by R\'ath and T\'oth~\cite{R09, RT09}. This is a modification of the dynamical Erd\H{o}s--R\'enyi graph on~$n$ vertices: edges appear between any pair of vertices with rate~$1/n$, and vertices are `struck by lightning' with rate~$\lambda_n$, where ${1/n \ll \lambda_n \ll 1}$. A lightning strike destroys the whole connected component of the vertex: in the forest fire model, we only remove edges, while in the frozen percolation model we remove vertices as well as edges of the destroyed component. These models exhibit \textit{self-organized criticality} (SOC): from the moment when a giant component tries to appear, the incipient giants are constantly destroyed and the graph permanently stays in a state similar to that of the critical Erd\H{o}s–-R\'enyi graph. The mean-field forest fire model has been further studied in~\cite{ CFT15, CRY21}.

\medskip
\textbf{Further variants.}
Besides these, numerous other variants of frozen percolation have been studied in the literature. Merle and Normand~\cite{MN14} investigated a variant in which components freeze when their size reaches a threshold that depends on~$n$. Another version of frozen percolation on the complete graph (the $p$-frozen Erd\H{o}s--R\'enyi model) was introduced by Contat and Curien~\cite{CC23}, where the growth of components containing a cycle is slowed down and multi-cyclic components are forbidden. This model was also investigated in~\cite{HV26, K24, V25}. For more on the literature of frozen percolation, see Section~\ref{subsec:literature}.

\medskip
\textbf{Window process and scaling limits.}
We focus on the SOC regime of the mean-field frozen percolation. Specifically, we consider the slightly subcritical sub-regime by assuming ${(\ln n)^{4/3}n^{-1/3} \ll \lambda_n \ll 1}$. Our main object of study is the \textit{window process}~$W_n(t):=\frac{t A_n(t)}{n}$, where $A_n(t)$ denotes set of vertices that are still alive at time $t$
(see Definition~\ref{def:window} below for further details). The window process keeps track of how close to criticality is the state of the dynamic random graph as it evolves. It was shown in~\cite{R09} that the window process~$W_n(t)$ converges in probability to a deterministic hockey-stick-shaped curve $t \mapsto t \wedge 1$, which is a universal sign of SOC behavior (see Remark~\ref{rem:hockey_stick} below). In this paper, we describe the microscopic behavior of the window process~$W_n(\cdot)$ by establishing two \textit{scaling limit} theorems, see Theorems~\ref{thm:OU_scaling} and~\ref{thm:tempered_stable_scaling} below. The study of scaling limits for dynamic random graph models also has a rich literature. We mention here the paper of Aldous~\cite{A97} in which the multiplicative coalescent process, i.e., the scaling limit of the dynamics of component structure of the near-critical Erd\H{o}s--R\'enyi graph is defined and studied, and the paper of Martin and Yeo~\cite{MY18} in which scaling limits for the largest components in cycle-free modifications are proved. Recently, the scaling limits of the $p$-frozen Erd\H{o}s--R\'enyi model have also been investigated in~\cite{V25}. For more on the literature of scaling limits for dynamic random graphs, see Section~\ref{sss:scaling_limits_literature}.

\medskip
\textbf{Fluctuations around the window curve.}
First, we study the fluctuations of the window process~$W_n(\cdot)$ around a finer deterministic curve~$w_n(t)$ that we call the \textit{window curve}, which is defined by the solution of an explicit ordinary differential equation (see Definition~\ref{def:qn_window_curve}).
We show that the fluctuations converge to a stationary \textit{Ornstein--Uhlenbeck process} (O.-U.~process), see Theorem~\ref{thm:OU_scaling} below. The mean-reverting property of the O.-U.~process captures the fine details of the mechanism behind the self-organized criticality of the model (see Remark \ref{remark_OU_SOC_explain} for further details).

\medskip
\textbf{Jump structure of the window process.}
Second, note that the window process~$W_n(\cdot)$ increases continuously and decreases with jumps (due to lightning strikes destroying entire connected components). In order to see these jumps in the scaling limit, we need to zoom in further and study the window process~$W_n(\cdot)$ on a finer space-time scale. On this finer scale, we prove that the scaled increments of the process~$W_n(\cdot)$ converge to a \textit{tempered stable L\'evy process} (see Theorem~\ref{thm:tempered_stable_scaling} below).

\section{Statements of main results}

\begin{notation}
We introduce the following notation for sequences $f_n,\, g_n\colon \mathbb{N} \to \mathbb{R}$:
\begin{alignat*}{2}
f_n &\ll g_n
&&\text{ if } \quad \frac{f_n}{g_n} \to 0 \text{ as } n \to \infty,\\
f_n &= \mathcal{O}(g_n) \quad
&&\text{ if } \quad \exists \, C \in \mathbb{R}_+ \, : \, \forall \, n \in \mathbb{N} \, : \, 0 \le |f_n| \le C \cdot g_n,\\
f_n &\asymp g_n
&&\text{ if } \quad \exists \, c,\, C \in \mathbb{R}_+ \, : \forall \, n \in \mathbb{N} \, : \, c g_n \le f_n \le C g_n, \\
f_n &\approx g_n
&&\text{ if } \quad \frac{f_n}{g_n} \to 1 \text{ as } n \to \infty.
\end{alignat*}

We also use the notation $f_n = o(g_n)$ equivalently to $f_n \ll g_n$, and $f_n = \Theta(g_n)$ equivalently to $f_n \asymp g_n$.
\end{notation}

We begin by recalling the definition of the frozen percolation model with linear deletion on~$K_n$ from~\cite{R09}. 

\begin{definition}[Frozen percolation on $K_n$ with lightning rate $\lambda_n$]\label{def:frozen_perc}
Given some $n \in \mathbb{N}$ and $\lambda_n \in \mathbb{R}_+$, we denote by $G_n(t)$ the state of the frozen percolation process at time $t$. $G_n(t)$ is a subgraph of the complete graph $K_n$ with $n$ vertices. The vertex set of $G_n(t)$ is $\mathcal{A}_n(t)$, the set of alive vertices at time $t$. We start with $\mathcal{A}_n(0)= \{\,1,\,2,\,\ldots,\,n\,\}$ and $G_n(0)$ is the empty graph on $n$ vertices. The dynamics of the model is as follows. Edges appear at rate $\frac 1n$ between any $v,w \in \mathcal{A}_n(t)$. 
On the other hand, lightnings hit vertices of~$\mathcal{A}_n(t)$ at rate $\lambda_n$ per vertex. If a lightning hits $v$ at time $t$, then it destroys $\mathcal{C}_n(v, t)$ the vertex set of the connected component of vertex $v$ in $G_n(t)$, i.e., $\mathcal{A}_n(t) = \mathcal{A}_n(t-) \setminus \mathcal{C}_n(v, t)$ and $G_n(t)$ is the restriction of the graph~$G_n(t-)$ to~$\mathcal{A}_n(t)$. 
\end{definition}

Note that $\mathcal{A}_n(t)$ is monotone decreasing, i.e., for any $s \le t$, we have $\mathcal{A}_n(t) \subseteq \mathcal{A}_n(s)$.

\begin{definition}[Alive vertices]\label{def:alive_vert}
Let us denote by $A_n(t)$ the \textbf{number of alive vertices} at time $t$, i.e., $A_n(t) = |\mathcal{A}_n(t)|$. Note that $A_n(0)=n$. Denote by $\mathcal{F}_n(t)$ the $\sigma$-algebra generated by $\mathcal{A}_n(s),\, 0 \leq s \leq t$.

Let us also introduce the notation~$Q_n(t)$ for the fraction of alive vertices at time~$t$:
\begin{equation}\label{eq:Qn_def}
Q_n(t) \coloneqq \frac{A_n(t)}{n}.
\end{equation}
\end{definition}

\begin{definition}[Window process]\label{def:window}
We define the \textbf{window process} as
\begin{equation}\label{eq:Wn_def}
W_n(t)\coloneqq tQ_n(t) = \frac{t A_n(t)}{n}.
\end{equation}
\end{definition}

Before we explain why we call the window process as we do in Remark~\ref{rem:wp_name}, let us introduce some notation.

\begin{notation}\label{not:ER_graph}
Let us denote by $\mathcal{G}(n,p)$ the binomial Erd\H{o}s--R\'enyi random graph with $n$ vertices and edge probability~$p$. We also use the notation $\mathcal{G}(V,p)$, where $V \subseteq [n]$, i.e., $\mathcal{G}(V,p)$ denotes a random graph on the vertex set $V$ with the same distribution as $\mathcal{G}(|V|,\,p)$.
\end{notation}

\begin{lemma}[Connection of frozen percolation with Erd\H{o}s--R\'enyi graph]\label{lem:ER}
Given $\mathcal{F}_n(t)$, we have $G_n(t) \sim \mathcal{G}(\mathcal{A}_n(t),\,1-\exp(-t/n))$, i.e., the conditional distribution of $G_n(t)$ given $\mathcal{F}_n(t)$ is Erd\H{o}s--R\'enyi on the vertex set $\mathcal{A}_n(t)$ with edge probability
$p=1-\exp(-t/n)$.
\end{lemma}

We will prove Lemma~\ref{lem:ER} in Section~\ref{sec:ER}. 

\begin{remark}[Window process nomenclature]
\label{rem:wp_name}
Lemma~\ref{lem:ER} implies that the conditional expectation of the number of neighbors of a fixed vertex in $G_n(t)$ given $\mathcal{F}_n(t)$ is $(A_n(t)-1)(1-\exp(-t/n))$, and this quantity is approximately equal to $W_n(t)$, cf.~\eqref{eq:Wn_def}. Thus, loosely speaking, if $W_n(t) >1$, then $G_n(t)$ is supercritical, if $W_n(t) <1$, then $G_n(t)$ is subcritical and if $W_n(t) \approx 1$, then~$G_n(t)$ is near-critical. This explains why we call $W_n(t)$ the window process: it tells us where we are in the critical window. However, let us note that in the literature of Erd\H{o}s--R\'enyi graphs (e.g.\ in~\cite{ABG12}), the critical window refers to the case when $A_n(t)=n$ and $|W_n(t)-1| \asymp n^{-1/3}$, while in this paper, for any fixed $t>1$ the value of $1-W_n(t)$ will turn out to be close to~${\lambda_n t}$ (cf.~Remark~\ref{rem:window_curve} below) and our assumptions on~$\lambda_n$ (cf.~\eqref{eq:subcrit_lambda_assumption} below) imply $n^{-1/3} \ll 1-W_n(t) \ll 1$, i.e., in this paper the graph at time $t>1$ will be near-critical, but leaning towards the subcritical. We conjecture that if $\lambda_n \asymp n^{-1/3}$ then $|W_n(t)-1| \asymp n^{-1/3}$ holds (i.e., the graph is in the critical window) if $t>1$ is fixed and $n \gg 1$.
\end{remark}

\begin{definition}[Time horizon]
We fix a time-horizon $T \in (1,+\infty)$. We will run the frozen percolation processes from time $0$ until time~$T$.
\end{definition}

\begin{definition}[SOC parameter range of $\lambda_n$, {\cite[Definition 1.2]{R09}}]
\label{def:soc_lambda_regime}
A sequence of frozen percolation models on~$K_n$ as $n \to \infty$ is called \emph{self-organized critical} if $1/n \ll \lambda_n \ll 1$.
\end{definition}
Loosely speaking, the assumption $\frac{1}{n} \ll \lambda(n)$ guarantees that giant connected components are instantly destroyed, while the condition $ \lambda(n) \ll 1$ guarantees that the lightnings do not harm small components.

\begin{theorem}[Fluid limit of the window process, {\cite[Theorem 1.2, Claim 1]{R09}}]\label{thm:window_critical}
If $1/n \ll \lambda_n \ll 1$ and $t \ge 0$, then $W_n(t)$ defined in \eqref{eq:Wn_def} converges in probability to $t \wedge 1$ as $n \to \infty$.
\end{theorem}

\begin{remark}[Hockey sticks and activated random walks]\label{rem:hockey_stick}
The hockey-stick-shaped curve $t \mapsto t \wedge 1$ is a universal sign of SOC behavior. Analogously to Theorem~\ref{thm:window_critical}, the driven-dissipative activated random walk on a long interval drives itself directly to and then sustains a critical density, producing a hockey-stick-shaped curve, see~\cite{HJJ26}.
\end{remark}

Our goal is to study the fluctuations of the window process around the flat part of the curve $t \mapsto t \wedge 1$. For that purpose, we assume that the lightning rate~$\lambda_n$ is in the barely subcritical sub-regime of the SOC regime:
\begin{equation}\label{eq:subcrit_lambda_assumption}
(\ln n)^{4/3} n^{-1/3} \ll \lambda_n \ll 1. 
\end{equation}

\begin{remark}\label{rem:log_correction} We note that the factor $(\ln n)^{4/3}$ is a technical assumption. We believe that our results also hold for $n^{-1/3} \ll \lambda_n \ll 1$. 
\end{remark}

\begin{remark}[Drift of $Q_n(t)$]
The ODE \eqref{eq:alive_ODE} below serves as a deterministic proxy for the time evolution of the fraction $Q_n(t)$ of alive vertices at time $t$ (cf.~\eqref{eq:Qn_def}). The rate at which a lightning strikes the set $\mathcal{A}_n(t)$ of alive vertices is $\lambda_n \cdot A_n(t)$, moreover, by Lemma~\ref{lem:ER} and the susceptibility formula~\eqref{eq:first_moment_of_component_size} the expected number of vertices that a lightning strike kills is close to $\frac{1}{1-W_n(t)}$. Recalling that $W_n(t)= tQ_n(t)$ from~\eqref{eq:Wn_def}, the drift of $A_n(t)$ is close to $-\lambda_n A_n(t) \cdot \frac{1}{1- t \cdot Q_n(t) }$ and thus the drift of $Q_n(t)$ is close to 
$-\lambda_n \cdot \frac{ Q_n(t)}{1- t \cdot Q_n(t) }$.
\end{remark}

\begin{definition}[ODE approximation of $Q_n(t)$]\label{def:qn_window_curve}
Denote by $q_n(t)$ the solution of the ODE
\begin{equation}\label{eq:alive_ODE}
\frac{\mathrm{d}}{\mathrm{d}t} q_n(t)= -\lambda_n \cdot \frac{q_n(t)}{1- t \cdot q_n(t) }, \qquad q_n(0)=1,
\end{equation}
and let us define the \emph{window curve}
\begin{equation}\label{eq:wn_def}
w_n(t)\coloneqq q_n(t) \cdot t, \qquad t\geq 0.
\end{equation}
\end{definition}

\needspace{3\baselineskip}
\begin{remark}[About the window curve]\label{rem:window_curve}
\mbox{}
\begin{enumerate}[label=(\roman*)]
\item 
We will see that the window process $W_n(t)$ and the window curve $w_n(t)$ are quite close to each other.
In fact, we will show that $w_n(t)$ attracts $W_n(t)$ in Theorem~\ref{thm:OU_scaling} if $t>1$. We will see that our assumption $\lambda_n \ll 1$ implies $w_n(t) \approx t \wedge 1$ for all $t \geq 0$, so our results elaborate on Theorem~\ref{thm:window_critical}. 

\item
We will state some properties of the window curve $w_n(t)$ in Section~\ref{sec:ODE}, but let us already mention that for any $t>1$ we have $w_n(t)=1-\lambda_n t +o(\lambda_n)$ (cf.\ Lemma~\ref{lem:gn_properties}). Note that by looking at \eqref{eq:alive_ODE} one sees that $\frac{\mathrm{d}}{\mathrm{d}t} q_n(t) \approx 0$ for any $t<1$ and that $q_n(t)$ only starts to move when $|1-w_n(t)| \asymp \lambda_n$. 
In order to quantify when this happens, we introduce~$t_n^c$ below.
\end{enumerate}
\end{remark}
We will refer to the regime when $|1-w_n(t)| \asymp \lambda_n $ holds as $\lambda_n$-\emph{criticality} of the window curve. We will see that this also implies $|1-W_n(t)| \asymp \lambda_n $, too. We call this regime $\lambda_n$-criticality because this is when the SOC of the model starts to manifest itself.

\begin{definition}[Arrival at $\lambda_n$-criticality]
We define the time $t_n^c$ of the arrival of the deterministic curve~$w_n(t)$ to $\lambda_n$-criticality as follows:
\begin{equation}\label{eq:tnc_def}
t_n^c\coloneqq \min\{\,t\ge 0\,:\, w_n(t)= 1-4\lambda_n\,\}
\end{equation}
\end{definition}

\begin{claim}[Asymptotic behavior of $t_n^c$]\label{cl:tc_asymp}
The time $t_n^c$ defined in \eqref{eq:tnc_def} satisfies
\begin{equation*}
 t_n^c = 1 + \lambda_n \cdot \ln \mleft( \frac{1}{\lambda_n} \mright) + \mathcal{O}(\lambda_n).
\end{equation*}
\end{claim}

We prove Claim~\ref{cl:tc_asymp} in Appendix~\ref{ap:window_curve_proofs}.

\begin{remark}[Slightly delayed arrival]
Note that $t_n^c \approx 1$, but let us also note without proof that $1-w_n(1) \asymp \lambda_n \cdot \ln \mleft( \frac{1}{\lambda_n} \mright)$, thus the window curve is still quite far from being $\lambda_n$-critical at $t=1$ (noting that~${t=1}$ is the critical time of the dynamic Erd\H{o}s--R\'enyi graph process without lightnings).
\end{remark}

\begin{definition}[Window fluctuation process]
We define the shifted and scaled \textbf{window fluctuation process} as
\begin{equation}\label{eq:Z_n_t_def}
Z_n(t)=(n \lambda_n)^{1/2} \cdot \left( W_n(t)-w_n(t) \right),
\end{equation}
and its time-scaled version
\begin{equation}\label{eq:tilde_Z_def}
\widetilde{Z}_n^t(s)\coloneqq Z_n(t +\lambda_n s)
\stackrel{\eqref{eq:Z_n_t_def} }{=}
(n \lambda_n)^{1/2} \cdot \left[ W_n(t +\lambda_n s)-w_n(t +\lambda_n s) \right], \quad s\in \mathbb{R}.
\end{equation}
\end{definition}

\begin{lemma}[Window process fluctuation bound]\label{lem:Z_bound}
For any $\varepsilon>0$, there exists $K<+\infty$ such that
\begin{equation}\label{eq:window_fluctuation_bound}
\sup_{n \in \mathbb{N}} \sup_{ 0 \leq t \leq T} \pr( \,|Z_n(t)| \geq K ) \leq \varepsilon.
\end{equation}
\end{lemma}

We prove Lemma~\ref{lem:Z_bound} in Section~\ref{sec:arrive-crit} for $0 \le t \le t_n^c$ using a Gr\"onwall-type argument and in Section~\ref{sec:SOC} for $t_n^c \le t \le T$ using a Lyapunov-type argument. Actually we will show in Theorem~\ref{thm:OU_scaling} that if $t >1$, then $ |Z_n(t)| \asymp 1$, i.e., we have $|W_n(t)-w_n(t)| \asymp \frac{1}{\sqrt{n \lambda_n}}$.

\needspace{3\baselineskip}
\begin{remark}[On the order of fluctuations]
\mbox{}
\begin{enumerate}[label=(\roman*)]
\item Note that if $t \geq t_n^c$, then $1-w_n(t) \asymp \lambda_n$ (cf.\ Remark~\ref{rem:window_curve}), while $|W_n(t)-w_n(t)| = \mathcal{O}\mleft( \frac{1}{\sqrt{n \lambda_n}} \mright)$ by~\eqref{eq:window_fluctuation_bound}. 
Thus by~\eqref{eq:subcrit_lambda_assumption} we have $ |W_n(t)-w_n(t)| \ll 1-w_n(t)$, i.e., the fluctuation is much smaller than the distance of the window curve from $1$, which implies that $W_n(t)$ is itself $\lambda_n$-critical.

\item The differential equation~\eqref{eq:alive_ODE} for $q_n(t)$ does not have an explicit solution, so there is no explicit formula for $w_n(t)$ either. One is thus tempted to approximate $w_n(t)$ with something simpler like $1-\lambda_n t$ when $t >1$, cf.\ Remark~\ref{rem:window_curve}. However, one can show that $|w_n(t)-(1-\lambda_n t)| \asymp \lambda_n^2$, and if $n^{-1/5} \ll \lambda_n$, then $|W_n(t)-w_n(t)| \asymp \frac{1}{\sqrt{n \lambda_n}} \ll \lambda_n^2$, much smaller than $|w_n(t)-(1-\lambda_n t)|$. One might try to remedy this by finding an expansion of form $w_n(t)=1-\lambda_n t +\sum_{\ell=2}^{L-1} \lambda^{\ell}_n a_\ell(t) +\Theta(\lambda^{L}_n) $, but the form of $a_\ell(t)$ becomes more and more complex as $\ell$ increases, while $|W_n(t)-w_n(t)| \asymp \frac{1}{\sqrt{n \lambda_n}} \ll \lambda^L_n$ as soon as $n^{-\frac{1}{2L+1}} \ll \lambda_n$. 
In other words, we cannot generally replace~$w_n(t)$ with something more explicit in the definition~\eqref{eq:Z_n_t_def} of~$Z_n(t)$ if we want~\eqref{eq:window_fluctuation_bound} to hold.
\end{enumerate}
\end{remark}

We denote by $D[a,b]$ the space of right-continuous functions $z\colon [a,b] \to \mathbb{R}$ with left limits.
The Skorokhod topology on this space was introduced in~\cite[Definition~2.2.2]{S56}, where it is called the topology $J_1$. A more detailed description of the space $D[a,b]$ and the Skorokhod topology can be found in~\cite[Section~14]{B68} or \mbox{\cite[Section~8.6]{D96}}.

Our first main result identifies the scaling limit of the fluctuations of $W_n(t)$ around $w_n(t)$.

\begin{theorem}[Ornstein--Uhlenbeck scaling limit]\label{thm:OU_scaling}
Recall~\eqref{eq:tnc_def}. Let $t \ge 1$ and $t_n$ be such that ${t_n \to t}$ and ${(t_n - t_n^c)/\lambda_n \to \infty}$ as $n \to \infty$. Consider the process $\widetilde{Z}_n^{t_n}$ defined in \eqref{eq:tilde_Z_def}. Then $\widetilde{Z}_n^{t_n}(\cdot) \Rightarrow \widetilde{Z}^t(\cdot)$ in $D[-S,S]$ for all $S<\infty$, where $\widetilde{Z}^t$ is the stationary Ornstein--Uhlenbeck process that satisfies the SDE
\begin{equation}\label{eq:OU_SDE}
\mathrm{d} \widetilde{Z}^t(s)= -\frac{1}{t^2} \widetilde{Z}^t(s)\, \mathrm{d}s + \frac{1}{t} \, \mathrm{d} B(s).
\end{equation}
\end{theorem}

We prove Theorem~\ref{thm:OU_scaling} in Section~\ref{sec:SOC}.

\needspace{3\baselineskip}
\begin{remark}[On the Ornstein--Uhlenbeck scaling limit]\label{remark_OU_SOC_explain}
\mbox{}
\begin{enumerate}[label=(\roman*)]
\item Note that in~\eqref{eq:OU_SDE} $s$ is the time variable, $t$ is a parameter. Observe that for any fixed~$t$, the process~$\widetilde{Z}^t(s)$ is a stationary Ornstein--Uhlenbeck process, but for different values of $t$, we obtain different stationary processes (i.e., if we `zoom in' on different parts of~$Z_n(\cdot)$, we see different scaling limits).
 
\item One may interpret the mean-reverting property of the O.-U.~process as the driving force behind the SOC of the frozen percolation model: if $W_n(t)$ is either more subcritical (i.e., lower) or more supercritical (i.e., higher) than $w_n(t)$, then the drift pushes it closer to the curve~$w_n(t)$, but~$w_n(t)$ stays in the $\lambda_n$-critical zone, thus $W_n(t)$ also stays $\lambda_n$-critical.
However, the mental image that the process $W_n(\cdot)$ fluctuates around the \emph{relatively flat} curve~$w_n(\cdot)$ can be misleading:
if $n^{-1/5} \ll \lambda_n$, then for any $0 \neq s \in \mathbb{R} $ we have $(n \lambda_n)^{1/2} \cdot | w_n(t+ \lambda_n s) -w_n(t)| \gg 1 $, i.e., using the same spatial and temporal scaling as in~\eqref{eq:tilde_Z_def}, the rescaled window curve is not at all flat, i.e., it looks `almost vertical' and not `almost horizontal', as one might naively expect. 
\end{enumerate}
\end{remark}

The window process $W_n(\cdot)$ increases continuously and decreases with jumps (cf.~\eqref{eq:Wn_def}). In order to see jumps in the scaling limit, we need to zoom in further than in Theorem~\ref{thm:OU_scaling}.

\begin{definition}[Scaled window increment process]
Let $t \in \mathbb{R}_+$ and let
\begin{equation}\label{eq:defnU_n(s)}
U_n^t(s)\coloneqq n\lambda_n^2 \cdot \left( W_n\mleft(t+\frac{s}{n\lambda_n^2}\mright) - W_n(t)\right), \qquad s \in \mathbb{R}_+.
\end{equation}
\end{definition}

Note that $n \lambda_n^2 \gg (n \lambda_n)^{1/2}$ and $ \frac{1}{n \lambda_n^2} \ll \lambda_n $ both follow from~\eqref{eq:subcrit_lambda_assumption}, thus we indeed stretch both space and time more in \eqref{eq:defnU_n(s)} than in \eqref{eq:tilde_Z_def} (i.e., we indeed zoom in further in \eqref{eq:defnU_n(s)} than in~\eqref{eq:tilde_Z_def}).

\begin{theorem}[Tempered stable L\'evy process scaling limit]\label{thm:tempered_stable_scaling}
Let $t$ and $t_n$ as in Theorem~\ref{thm:OU_scaling}, and consider the process $U_n^{t_n}$ defined in \eqref{eq:defnU_n(s)}. Then $U_n^{t_n}(\cdot) \Rightarrow U^t(\cdot) $ in $D[0,S]$ for each $S<\infty$, where $U^t(\cdot)$ is the L\'evy process defined by
\begin{equation}\label{eq:defnU(s)}
U^t(s)\coloneqq \frac{1}{t} s - \frac{1}{t} \widetilde{D}(s), \qquad 
\widetilde{D}(s)\coloneqq
\int \limits_0^s \int \limits_0^\infty y N(\mathrm{d} y,\mathrm{d} u),
\end{equation}
where $N(\cdot,\cdot)$ is a Poisson point process (PPP) on $\mathbb{R}^2_+$ with intensity measure $\frac{1}{\sqrt{2\pi}}y^{-3/2}\mathrm{e}^{-y/2}\, \mathrm{d} y \, \mathrm{d} u$.
\end{theorem}

We prove Theorem~\ref{thm:tempered_stable_scaling} in Section~\ref{sec:levy_proof}.

\begin{remark}[Tempered stable processes]
The process~$\widetilde{D}(\cdot) $ is called a tempered $1/2$-stable L\'evy process. In~\cite{R07} Rosi\'nski gives the standard general framework for tempered stable processes and explains their short-time stable and long-time Gaussian behavior (cf.~Remark~\ref{rem:further_scaling}). This is particularly relevant because Remark~\ref{rem:further_scaling} describes exactly these two crossovers: a stable limit at a finer scale (cf.~Remark~\ref{rem:further_scaling}~\ref{remark_ii_levy}) and a Brownian limit after zooming out (cf.\ Remark~\ref{rem:further_scaling}~\ref{remark_i_br}).
\end{remark}

\begin{remark}[Scaling limit of the empirical distribution of frozen components] Note that the result of
Theorem~\ref{thm:tempered_stable_scaling} agrees with the slightly subcritical statement of \cite[Conjecture~1.1]{R09} (i.e., that the size of the connected component of a frozen vertex is comparable to $\lambda_n^{-2}$) and the result of \cite[Theorem~1.8]{R09} about the scaling limit of the size-biased empirical distribution of frozen components.
\end{remark}

Now we formulate some conjectures about other scaling limits associated to our setup. We believe that these statements can be proved by straightforward modifications of the arguments of this paper. We also believe that these results are less interesting than Theorems~\ref{thm:OU_scaling} and~\ref{thm:tempered_stable_scaling} because the processes that arise in the limit are simpler (and therefore capture less information than our main theorems). However, for the sake of completeness, we wanted to at least mention all possible non-trivial scaling limits that can arise from the processes $W_n(\cdot)$ and $A_n(\cdot)$.

\needspace{3\baselineskip}
\begin{remark}[Further (conjectured) scaling limits in the $\lambda_n$-critical regime]\label{rem:further_scaling}
\mbox{}
\begin{enumerate}[label=(\roman*)]
\item \textbf{ Properties of $U^t(\cdot)$.} The process $U^t(\cdot)$ is a martingale, which grows at constant speed, but decreases with jumps. The set of its jump times is dense in $\mathbb{R}_+$. The variance of $U^t(s)$ is $s/t^2$.
\item\label{remark_i_br} \textbf{ Intermediate Brownian scaling limit.} If we either zoom in on $\widetilde{Z}^{t}(\cdot)$ (cf.~\eqref{eq:OU_SDE}) or zoom out on~$U^{t}(\cdot)$ (cf.~\eqref{eq:defnU(s)}), then we see the same Brownian motion in the limit. Indeed, let us note without proof that if $t$ and $t_n$ satisfy the same conditions as in Theorem~\ref{thm:OU_scaling}, if we introduce the intermediate time scale $1/{(n \lambda_n^2)} \ll \chi_n \ll \lambda_n$ and if we define
\begin{equation}\label{eq:intermediate_scaling}
B_n^{t_n}(s)\coloneqq\frac{\lambda_n \sqrt{n} }{\sqrt{\chi_n}} \cdot \left[ (W_n(t_n+\chi_n s )-W_n(t_n))-(w_n(t_n+\chi_n s )-w_n(t_n)) \right],
\end{equation}
then $B_n^{t_n}(\cdot) \Rightarrow \frac{1}{t} B(\cdot) $ in $D[0,S]$ for each $S<\infty$, where $B(\cdot)$ is a standard Brownian motion. 
Note that the term $w_n(t_n+\chi_n s )-w_n(t_n)$ can be omitted from~\eqref{eq:intermediate_scaling} if $\chi_n \ll 1/{(n \lambda^4_n)}$.
\item\label{remark_ii_levy} \textbf{ Scaling limits of the increments of $A_n(t)$.} 
Let $J_n^t(s) \coloneqq A_n(t) - A_n(t+s) $. The result of Theorem~\ref{thm:tempered_stable_scaling} can be rephrased as $\left( \lambda_n^2 \cdot J_n^{t_n}\mleft( \frac{s}{n \lambda_n^2} \mright)\right)_{s \in [0,S]} \Rightarrow \left(\frac{1}{t^2} \widetilde{D}(s)\right)_{s \in [0,S]}$, where $\widetilde{D}(\cdot)$ was defined in \eqref{eq:defnU(s)}, cf.~Proposition~\ref{prop:Yn_convergence}. Let us mention other possible scaling limits of $J_n^t(\cdot)$ without proof.
If~${t \geq 1}$, then $\left(J^{t}_n(\frac{s}{n \lambda_n})\right)_{s \in [0,S]} \Rightarrow \left(J^t(\cdot)\right)_{s \in [0,S]}$, where $J^t(\cdot)$ jumps at the arrival times of a
homogeneous PPP on~$\mathbb{R}_+$ with intensity~$1/t$ and the jump sizes are i.i.d.\ with critical Borel distribution (i.e., the distribution of the number of vertices of a critical Poisson Galton--Watson tree). If we either zoom in on $\frac{1}{t^2} \widetilde{D}(\cdot)$ or if we zoom out on $J^t(\cdot)$, then we obtain the same stable subordinator process in the limit. Indeed, 
if we introduce the intermediate time scale $\frac{1}{n \lambda_n} \ll \chi_n \ll \frac{1}{n \lambda^2_n}$, then
\begin{equation*}
\left( ( n \chi_n \lambda_n )^{-2} J_n^t(\chi_n s) \right)_{s \in [0,S]} \Rightarrow \left( \frac{1}{t^2} L(s) \right)_{s \in [0,S]},
\end{equation*}
where 
$L(s)\coloneqq
\int_0^s \int_0^\infty y M(\mathrm{d} y,\mathrm{d} u) $
where $M(\cdot,\cdot)$ is a PPP on $\mathbb{R}^2_+$ with intensity measure $\frac{1}{\sqrt{2\pi}}y^{-3/2}\, \mathrm{d} y \, \mathrm{d} u$.
\end{enumerate}
\end{remark}

The SOC regime of lightning rates (i.e., $1/n \ll \lambda_n \ll 1$, cf.~Definition~\ref{def:soc_lambda_regime}) can be subdivided into three sub-regimes: the slightly subcritical regime $n^{-1/3} \ll \lambda_n \ll 1$, the truly critical regime $\lambda_n \asymp n^{-1/3}$, and the slightly supercritical regime $1/n \ll \lambda_n \ll n^{-1/3}$.
This paper focuses on the slightly subcritical regime (assuming an extra $(\ln n)^{4/3}$ correction for technical reasons, cf.~Remark~\ref{rem:log_correction}), but let us briefly mention some conjectures pertaining to the other two regimes and their connections to the literature.

\begin{remark}[Conjectured scaling limit in the $\lambda_n \asymp n^{-1/3}$ regime]\label{rem:crit_mcld}
Note that if we plug $\lambda_n \asymp n^{-1/3}$ into the spatial scaling $(n \lambda_n)^{1/2}$ of Theorem~\ref{thm:OU_scaling} and the spatial scaling $ n\lambda_n^2$ of Theorem~\ref{thm:tempered_stable_scaling}, in both cases we obtain $n^{1/3}$. Similarly, if we plug $\lambda_n \asymp n^{-1/3}$ into the time scaling $\lambda_n$ of Theorem~\ref{thm:OU_scaling} and the time scaling $\frac{1}{n \lambda_n^2}$ of Theorem~\ref{thm:tempered_stable_scaling}, in both cases we obtain $n^{-1/3}$. 
Thus, it is reasonable to conjecture that if $\lambda_n=\lambda \cdot n^{-1/3}$ (where $\lambda \in \mathbb{R}_+$ is fixed), $t \ge 1$ and $t_n$ satisfies ${t_n \to t}$ and ${(t_n - t_n^c)/\lambda_n \to \infty}$ as $n \to \infty$, then $\left(n^{1/3} \cdot \left[ W_n(t_n +n^{-1/3} s)-w_n(t_n +n^{-1/3} s) \right]\right)_{s \in \mathbb{R}}$ converges in distribution as $n \to \infty$, where the limit process is a stationary process (just like the O.-U.~process $\widetilde{Z}^t(\cdot)$), but it increases with constant speed and decreases with jumps that are dense in time (just like the L\'evy process $U^t(\cdot)$). Also note that the law of the limit process depends on $t$ and $\lambda$. The conjectured limit process is intimately related to the \emph{multiplicative coalescent with linear deletion} (introduced in~\cite{MR17, R19}).
\end{remark}

\needspace{3\baselineskip}
\begin{remark}[Conjectured scaling limits in the $\frac{1}{n} \ll \lambda_n \ll n^{-1/3}$ regime] We conjecture that in this regime
baby giants of size comparable to $\sqrt{\frac{n}{\lambda_n}}$ are born and get destroyed again and again, giving rise to alternating time intervals of slight sub- and supercriticality, 
analogously to \cite[Theorem~1.9]{R09}. \end{remark}

\subsection{Related literature}\label{subsec:literature}

\subsubsection{Related modifications of dynamic Bernoulli percolation on infinite graphs}

The term `frozen percolation' was coined by Aldous in~\cite{A00} where he constructed frozen percolation on the infinite 3-regular tree and showed that, after the ordinary Bernoulli bond percolation threshold, a typical finite cluster has the law of a critical percolation cluster, while clusters that freeze have the law of the incipient infinite cluster.
These results are the infinite-volume counterpart of the SOC behavior studied here.
Van den Berg, Kiss and Nolin~\cite{BKN12} showed that freezing clusters at a large finite size threshold on the binary tree recovers Aldous' finite-cluster probabilities and freezing probabilities as the threshold tends to infinity.

We have already mentioned it in the Introduction that a stochastic process satisfying the naive definition of frozen percolation on~$\mathbb{Z}^2$ does not exist (see Section~\ref{sss:two_dim} for further details). It is currently open whether frozen percolation (or the closely related forest fire model) on $\mathbb{Z}^d,\, d \geq 3$ can be constructed. The first step of this construction is to show that \emph{self-destructive percolation} (introduced by van den Berg and Brouwer in~\cite{BB04}) behaves regularly: this is proved on $\mathbb{Z}^d$ for $d \geq 11$ by Ahlberg, Duminil-Copin, Kozma and Sidoravicius in~\cite{ADKS15} and for a large class of non-amenable graphs by Ahlberg, Sidoravicius and Tykesson in~\cite{AST14}. Also note that a spatially inhomogeneous forest fire model on~$\mathbb{Z}_+$ is constructed by van den Berg and T\'oth in~\cite{BT01}.

\subsubsection{Frozen percolation in two dimensions}\label{sss:two_dim}

The non-existence of frozen percolation on~$\mathbb{Z}^2$ is already surprising, but one can define various approximations of it which are well-defined, but behave in a highly counter-intuitive way. One example of this is the irregular behavior of self-destructive percolation on~$\mathbb{Z}^2$ proved by Kiss, Manolescu and Sidoravicius in~\cite{KMS15}.

Van den Berg, de Lima and Nolin~\cite{BLN12} introduced diameter-frozen percolation on $\mathbb{Z}^2$ and proved that the final cluster of a fixed vertex has, with probability bounded away from zero, diameter of order the freezing threshold while remaining unfrozen.
Kiss~\cite{K15} subsequently showed, for the triangular-lattice model, that freezing in a box of comparable diameter occurs in the near-critical window, and that the final cluster of a typical vertex is unfrozen and has diameter of order the threshold with high probability.
This localization of freezing near the ordinary percolation threshold is different from the behavior of frozen percolation on~$K_n$ because on~$K_n$ the freezing starts with a delay compared to the critical window of the Erd\H{o}s--R\'enyi graph (cf.~Claim~\ref{cl:tc_asymp}) and continues to happen at a slowly decaying intensity for a long time after that. 

Replacing diameter by volume leads to a different picture.
Van den Berg and Nolin~\cite{BN17b} identified a sequence of exceptional box sizes for volume-frozen percolation: the probability that the origin eventually freezes stays bounded away from zero at these scales, but vanishes between successive exceptional scales.
Van den Berg, Kiss and Nolin~\cite{BKN18} established a deconcentration property for the holes surrounded by successively formed frozen clusters and used it to prove that, in the full-plane triangular-lattice model, a typical vertex remains unfrozen with high probability as the volume threshold diverges.

The closest planar analogue of our lightning mechanism is the forest fire process without recovery studied by van den Berg and Nolin in~\cite{BN21}: vertices are struck independently, and an entire occupied component is permanently destroyed when one of its vertices is struck: the authors obtain exceptional scales for this variant of frozen percolation, too.

\subsubsection{Related modifications of mean field random graph process}

The direct predecessor of this paper is~\cite{R09}, which introduced the frozen percolation model with linear deletion on~$K_n$ that we study here. 
Merle and Normand~\cite{MN14} considered a slightly different modification of the dynamical Erd\H{o}s--R\'enyi graph (already suggested in \cite[Section~5.5]{A00}) in which components freeze once their size reaches a prescribed threshold $\omega_n$ that satisfies $1 \ll \omega_n \ll n$. They prove that the statement of Theorem~\ref{thm:window_critical} holds in this model and also show that after the critical time the local limit of the graph is a critical Poisson--Galton--Watson branching process tree, giving another mean-field realization of SOC.

A close relative is the mean-field forest fire model of R\'ath and T\'oth~\cite{RT09}, where a lightning strike removes all edges of the affected component but retains its vertices, allowing them to participate in subsequent growth.
They proved convergence of the empirical cluster-size distribution to the solution of a modification of Smoluchowski's coagulation equations in the regime when the lightning rate satisfies $1/n\ll\lambda_n\ll1$, also proving that after the critical time the solution exhibits power-law behavior, a sign of SOC.
Crane, Freeman and T\'oth~\cite{CFT15} identified the limiting size process of the cluster containing a tagged vertex as an explosive, time-inhomogeneous branching process that returns instantaneously to size one after each explosion.
Crane, R\'ath and Yeo~\cite{CRY21} showed that the empirical distribution of vertex ages has a deterministic measure-valued limit, and proved that the local weak limit of the graph is a multi-type Galton--Watson tree that is critical after the critical time.
Yeo~\cite{Y18} studied the same linear-deletion dynamics as ours, started from finite-type inhomogeneous random graphs, and derived a hydrodynamic limit for the surviving type proportions.
The multi-type branching process that one obtains in the local weak limit remains critical after gelation.

Bertoin~\cite{B12} considered fires on the uniform spanning tree of~$K_n$ (i.e., the uniform Cayley tree), where fire spreads through flammable edges and is stopped by edges that have become fireproof, and identified a transition in the fraction of vertices that survive as the ignition rate varies with the tree size.
A different connection with Cayley trees was established by Contat and Curien~\cite{CC23}, who coupled parking on a uniform rooted Cayley tree to a frozen Erd\H{o}s--R\'enyi process and identified its critical component size scaling limit.
In the associated $p$-frozen family, trees freeze when they acquire a cycle, attempted mergers between a tree and a frozen component are accepted with probability $p$, and multicyclic components are forbidden, with the parking coupling corresponding to $p=1/2$; thus freezing is triggered by cycle formation or absorption rather than by external lightning.
Krapivsky~\cite{K24} used kinetic equations to study this random graph model, deriving predictions for the evolving tree and unicycle populations.
Haas and Viau~\cite{HV26} proved functional fluid limits for the main observables of $p$-frozen percolation and obtained precise asymptotics for the time at which the entire graph becomes frozen.

\subsubsection{Scaling limits for dynamic random graph models}\label{sss:scaling_limits_literature}

Aldous~\cite{A97} identified the scaling limit of the evolution of the collection of Erd\H{o}s--R\'enyi component sizes throughout the critical window as the standard multiplicative coalescent~(SMC) process, a process with infinite total mass in which components of masses~$x$ and~$y$ merge at rate~$xy$. He also showed that the joint law of the collection of component sizes at a fixed time in the SMC is the same as the collection of excursion lengths above past minima of Brownian motion with parabolic drift. 
Aldous and Limic~\cite{AL98} classified the entrance boundary of the multiplicative coalescent, describing its extreme eternal versions and excursion representations of their one-time distributions.
Limic~\cite{L19} strengthened this excursion description to a simultaneous encoding of the entire eternal coalescent, with time evolution corresponding to linear tilting of a single L\'evy-type path.

To incorporate our destruction mechanism, Martin and R\'ath~\cite{MR17} developed so-called rigid representations of the multiplicative coalescent with linear deletion~($\mathrm{MCLD}(\lambda)$), where a component of mass $x$ is deleted at rate~$\lambda x$ in addition to the usual multiplicative coalescent dynamics.
In their excursion representation, coalescence is encoded by tilting, while deletion is encoded by a shift of the underlying path.
R\'ath~\cite{R19} proved the Feller property of $\mathrm{MCLD}(\lambda)$ for square-summable component-mass sequences.
These results provide a natural framework for the conjectured component and window-process limits in the truly critical regime $\lambda_n\asymp n^{-1/3}$ of frozen percolation on $K_n$, but do not by themselves establish the stationary window-process limit mentioned in Remark~\ref{rem:crit_mcld}.
The present paper instead considers the slightly subcritical regime, where the fluctuation and individual-jump scales separate.

Martin and Yeo~\cite{MY18} obtained scaling limits for the largest components of uniform random forests and Erd\H{o}s--R\'enyi graphs conditioned to be acyclic, expressed through excursions of a certain reflected diffusion. Such graphs are conjectured to appear in a modification of frozen percolation where components with cycles are frozen.

Viau~\cite{V25} studied the long-time behavior of the total frozen mass in the critical scaling limit of the $p$\nobreakdash-frozen Erd\H{o}s--R\'enyi model, proving convergence of the centered frozen mass to an invariant distribution when~${p=0}$, in contrast to concentration around a deterministic linear trend when $p>0$.
For $p=0$, the unfrozen components evolve as a multiplicative coalescent with quadratic deletion, since a component of mass~$x$ freezes at rate $x^2/2$. The convergence to an invariant distribution in the $p=0$ case is a clear sign of SOC and offers a close parallel to the mean-reverting property (and stationarity) of the O.-U.~process in our Theorem~\ref{thm:OU_scaling}.

\section{Basic properties of Erd\H{o}s--R\'enyi graphs and the window process }\label{sec:ER}

In this section we prove Lemma~\ref{lem:ER}, i.e., we show that, given $\mathcal{F}_n(t)$, the conditional distribution of $G_n(t)$ is an Erd\H{o}s--R\'enyi graph. Then in Section~\ref{subsec:ER_comp} we prove some useful formulas for the component size of a fixed vertex of the Erd\H{o}s--R\'enyi graph and in Section~\ref{subsec:martingales} we construct martingales using the number of alive vertices.

\begin{definition}[PPP of lightnings]
The lightning strikes of $(G_n(\cdot))$ can be described as i.i.d.\ exponential clocks on each vertex with parameter $\lambda_n$ (see Definition~\ref{def:frozen_perc}). We can augment the underlying probability space and describe the lightning strikes as a Poisson point process $\Lambda$ on $[n] \times \mathbb{R}_+$ with rate $1 \otimes \lambda_n \text{Leb}$.
\end{definition}

Lemma~\ref{lem:ER} is similar to \cite[Theorem 2.6]{CRY21}, which is proved in~\cite[Section 3.1]{CRY21}. We prove Lemma~\ref{lem:ER} using an analogous argument.

\begin{proof}[Proof of Lemma~\ref{lem:ER}]
We prove a stronger statement. In particular, we condition also on the whole PPP $\Lambda$ of the lightnings and show that for any $t \ge 0$,
\begin{equation}\label{eq:ER-stronger}
\text{given } \Lambda \text{ and } \mathcal{A}_n(s),\, 0 \leq s \leq t, \text{ the conditional distribution of } G_n(t) \text{ is } \mathcal{G}(\mathcal{A}_n(t),\,1-\mathrm{e}^{-t/n}).
\end{equation}

Assume that $\Lambda$ consists of the points $(v_i, t_i), \, i =1,\, 2,\, \ldots$, i.e., $v_i \in [n]$, $0 < t_1 < t_2 < \ldots$ and at time $t_i$ a lightning hits vertex $v_i$. We prove that~\eqref{eq:ER-stronger} holds for any $t \le t_i$ by induction on $i$.

First, note that~\eqref{eq:ER-stronger} is fulfilled for $t \le t_1$ by Definition~\ref{def:frozen_perc}: $G_n(t)$ is a graph where each edge between any $v,\, w \in [n]$ with $v \neq w$ is present with probability $1-\mathrm{e}^{-t/n}$, independently of all other edges, i.e., ${G_n(t) \sim \mathcal{G}([n],1-\mathrm{e}^{-t/n})}$.

Now assume that~\eqref{eq:ER-stronger} holds for $t \le t_{i-1}$. Then for any $t \in [t_{i-1}, t_i)$, we know that there is no lightning in the interval $(t_{i-1},t]$. Hence, an edge $e$ is present in $G_n(t)$ if and only if $e$ is present in $G_n(t_{i-1})$ or it appears during the interval $(t_{i-1},t]$. By the induction hypothesis, we know that an edge between $v \neq w \in \mathcal{A}_n(t_{i-1})$ is present in $G_n(t_{i-1})$ with probability $1-\exp(-t_{i-1}/n)$ and if not, it appears in the interval $(t_{i-1},t]$ independently from the past with probability $1-\exp(-(t-t_{i-1})/n)$. Therefore, for any $v \neq w \in \mathcal{A}_n(t)$ the edge $\{v, w\}$ is present in $G_n(t)$ 
with probability
\begin{align*}
1-\mathrm{e}^{-\frac{t_{i-1}}{n}} \cdot \mathrm{e}^{-\frac{t-t_{i-1}}{n}} = 1-\mathrm{e}^{-\frac tn},
\end{align*}
independently of other edges.
This proves that~\eqref{eq:ER-stronger} holds for any $t < t_i$, thus it remains to show \eqref{eq:ER-stronger} for $t=t_i$. If $v_i \notin \mathcal{A}_n(t_{i-1})$ (i.e., $v_i$ was deleted earlier), then nothing happens at time $t_i$ and~\eqref{eq:ER-stronger} holds. Otherwise, the connected component $\mathcal{C}_n(v_i, t_i)$ of vertex~$v_i$ in $G_n(t_i)$ is deleted at time $t_i$ and $\mathcal{A}_n(t_i) = \mathcal{A}_n(t_{i-1}) \setminus \mathcal{C}_n(v_i, t_i)$.
Given $\Lambda$ and~$\mathcal{A}_n(s),\, 0 \leq s \leq t_i$, the conditional distribution of the subgraph of $G_n(t_i-)$ spanned by $\mathcal{A}_n(t_{i})$ is $\mathcal{G}(\mathcal{A}_n(t_i),\, 1-\mathrm{e}^{-t_i/n})$, because the identity of the deleted component reveals information about edges within that component and edges from that component to surviving vertices, but it does not reveal any information about edges among the surviving vertices. Since no edge appears exactly at time $t_i$, the subgraphs of $G_n(t_i-)$ and~$G_n(t_i)$ spanned by the vertex set $\mathcal{A}_n(t_i)$ are the same. From these observations \eqref{eq:ER-stronger} follows for $t=t_i$ and this completes the proof of the induction step. The proof of Lemma~\ref{lem:ER} is complete.
\end{proof}

\subsection{Properties of the Erd\H{o}s--R\'enyi graph model}\label{subsec:ER_comp}

Now we recall/deduce some facts about the distribution of the size of the connected component of a fixed vertex of an Erd\H{o}s--R\'enyi graph.

\begin{notation}[Erd\H{o}s--R\'enyi graph with tunable vertex set]
\label{not:ER_notation_A_n_t}
Let $\pr_{A,n,t}(\cdot)$ be a shorthand for the distribution of the Erd\H{o}s--R\'enyi random graph $\mathcal{G}(A,1-\exp(-t/n))$, where $A,\, n \in \mathbb{N}$, $A \leq n$. Let $\ev_{A,n,t}(\cdot)$ denote the expectation of random variables that are defined in terms of $\mathcal{G}(A,1-\exp(-t/n))$.

\noindent Denote by $\mathscr{C}$ the number of vertices in the connected component of a fixed vertex of the Erd\H{o}s--R\'enyi graph.

Given $A$, $n$ and $t$, let us denote
\begin{equation}\label{eq:W_t_A_n_def}
W_n^A(t)=\frac{tA}{n},
\end{equation}
in accordance with~\eqref{eq:Wn_def}.
\end{notation}

\begin{claim}[Moments of component size]\label{cl:component_size}
Let $A,\, n \in \mathbb{N}$, $A \leq n$, $t \ge 0$. There exists a universal constant~$\widetilde{C}$ such that for any $A$, $n$, $t$ if $W_n^A(t)\le 1-A^{-1/3}$, then we have 
\begin{align}
&\left| \ev_{A,n,t}[\mathscr{C}] - \frac{1}{1-W_n^A(t)} \right| \le \frac{\widetilde{C}}{(1-W_n^A(t))^4 A}, \label{eq:first_moment_of_component_size}\\
&\left| \ev_{A,n,t} \mleft[\mathscr{C}^2 \mright] - \frac{1}{(1-W_n^A(t))^3} \right| \le \frac{\widetilde{C}}{(1-W_n^A(t))^6 A}.\label{eq:second_moment_of_component_size}
\end{align}

For any $A$, $n$, $t$, we have
\begin{equation}
\ev_{A,n,t}\mleft[ \prod_{i=0}^{\mathscr{C}-1} \mathrm{e}^{-tj/n}\left( 1+ \frac{j}{A-i} \right) \mright]=1+ \frac{j}{A}, \quad j \in \{-A, \ldots,-1, 0\},\label{eq:magic_A_n_t}
\end{equation}

and if $W_n^A(t) \in (0,1)$, then for any $k \in \mathbb{N}_+$, there exists a universal constant $\widetilde{C}_k$ such that
\begin{equation}
\ev_{A,n,t} \mleft[\mathscr{C}^k \mright] \le \frac{\widetilde{C}_k}{(1-W_n^A(t))^{2k-1}}, \qquad k = 1,2,3,\ldots. \label{eq:general_moment_of_component_size_crude}
\end{equation}

Moreover, if $W_n^A(t) \in (\frac{1}{3},1)$, then
\begin{equation}
\ev_{A,n,t}\mleft[ \mathrm{e}^{\frac{1}{4} (1-W_n^A(t))^2 \mathscr{C}} \mright] \leq 3.\label{eq:magic_branching_proc_estimate_2}
\end{equation}
\end{claim}

\begin{proof} Some of these results are classical (see~\cite{R18} for an overview), but we will use~\cite{R18} as a source of reference. Observe that by~\eqref{eq:W_t_A_n_def}, $1-\exp(-t/n) = 1-\exp(-W_n^A(t)/A)$. Therefore, if we apply \cite[Equations~(3.9) and~(1.8)]{R18} for the Erd\H{o}s--R\'enyi random graph $\mathcal{G}(A,1-\exp(-t/n))$, we obtain~\eqref{eq:first_moment_of_component_size} and~\eqref{eq:second_moment_of_component_size}. We also note that the condition $t \le 1-n^{-1/3}$ of the referred equations of~\cite{R18} is equivalent to~${W_n^A(t)\le 1-A^{-1/3}}$ in the case of $\mathcal{G}(A,1-\exp(-t/n))$.

Equation~\eqref{eq:magic_A_n_t} and inequality~\eqref{eq:general_moment_of_component_size_crude} follow from \cite[Equations (1.2), (3.1)]{R18}, respectively. Finally, note that if ${2}/{(1-W_n^A(t))} \ge A$, then inequality~\eqref{eq:magic_branching_proc_estimate_2} is satisfied trivially. On the other hand, if ${2}/{(1-W_n^A(t))} \le A$, then using similar estimations as in~\cite[Inequality~(3.6)]{R18}, we can deduce that
\begin{equation*}
\ev_{A,n,t}\mleft[ \exp \mleft(\frac{1}{4} (1-W_n^A(t))^2 \mathscr{C} \mright) \mright] \le \frac{1}{W_n^A(t)}.
\end{equation*}
Hence, \eqref{eq:magic_branching_proc_estimate_2} holds if $W_n^A(t) \ge 1/3$.
\end{proof}

Note that Janson and Luczak earlier proved a statement (see \cite[Theorem~3.4]{JL08}) that is equivalent to \cite[Equation~(1.8)]{R18} that we used in the proof of Claim~\ref{cl:component_size}. For more details, see \cite[Remark~1.5]{R18}

\subsection{Martingales associated to the window process}\label{subsec:martingales}

\begin{definition}[Drift of a stochastic process]
If $(X_n(t))$ is a stochastic process adapted to $(\mathcal{F}_n(t))$, we define the \textbf{drift} of $X_n(t)$ at time $t$ as
\begin{equation}\label{eq:drift_def}
\mathcal{L} X_n(t)\coloneqq \lim_{h \to 0_+} \frac{1}{h}\ev \mleft[ \, X_n(t+h)-X_n(t)\, | \, \mathcal{F}_n(t) \, \mright]
\end{equation}
\end{definition}

\begin{claim}[Infinitesimal generator formula for time-dependent functions of $(A_n(t))$]\label{cl:time_dep_inf_gen}
$(A_n(t))_{t \geq 0 }$ is a continuous-time Markov process (albeit not time-homogeneous). Moreover, if $g\colon \{0,1,\dots,n\} \times \mathbb{R} \to \mathbb{R}$ denotes a smooth test function and $g_2$ denotes the partial derivative of $g$ with respect to its second variable, then
\begin{equation}\label{eq:A_n_t_infgen}
\mathcal{L} g(A_n(t),t) = \lambda_n A_n(t) \ev_{A_n(t),n,t} \mleft[\, g(A_n(t)-\mathscr{C},t)-g(A_n(t),t) \, \mright] + g_2(A_n(t),t).
\end{equation}
\end{claim}

\begin{proof}
The Markov property of $A_n(t)$ follows by Lemma~\ref{lem:ER}. We have
\begin{align*}
\mathcal{L} g(A_n(t),t) \stackrel{\eqref{eq:drift_def}}{=} &\lim_{h \to 0_+} \frac{1}{h}\ev\mleft[ \, g(A_n(t+h),t+h)-g(A_n(t),t+h)\, | \, \mathcal{F}_n(t) \, \mright] +\\
&\lim_{h \to 0_+} \frac{1}{h}\ev\mleft[ \, g(A_n(t),t+h)-g(A_n(t),t)\, | \, \mathcal{F}_n(t) \, \mright].
\end{align*}

Here the first term is equal to $\lambda_n A_n(t) \ev_{A_n(t),n,t}[g(A_n(t)-\mathscr{C},t)-g(A_n(t),t)]$ and the second term is equal to $g_2(A_n(t),t)$, since $A_n(t) \in \mathcal{F}_n(t)$.
\end{proof}

Recall that $Q_n(t)$ denotes the ratio of alive vertices at time $t$ (see~\eqref{eq:Qn_def}).

\begin{definition}[Martingales of $Q_n(t)$]
Let us introduce the martingales
\begin{alignat}{2}
M_n(t) &\coloneqq Q_n(t) -1 - \int \limits_0^t \mathcal{L} Q_n(s)\, \mathrm{d}s, \qquad & M_n(0)=0, \label{eq:alive_corrector_martingale} \\
N_n(t) &\coloneqq M_n^2(t) - \int \limits_0^t \mathcal{L} M_n^2(s) \, \mathrm{d}s, \qquad & N_n(0)=0. \label{eq:quadratic_corrector_martingale}
\end{alignat}
\end{definition}

\begin{claim}[Drift of $Q_n(t)$ and $M_n^2(t)$]
The drift of $Q_n(t)$ and $M_n^2(t)$ are the following
\begin{align}
\mathcal{L} Q_n(t) &= - \lambda_n Q_n(t) \ev_{A_n(t),n,t}[\mathscr{C}], \label{eq:neighbor_drift}\\
\mathcal{L} M_n^2(t) &= \frac{1}{n} \lambda_n Q_n(t) \ev_{A_n(t),n,t}[\mathscr{C}^2]. \label{eq:quadratic_variation}
\end{align}
\end{claim}

\begin{proof}
Equation \eqref{eq:neighbor_drift} follows from~\eqref{eq:A_n_t_infgen}. Using the martingale property in~$(*)$ below, we obtain~\eqref{eq:quadratic_variation}:
\begin{align*}
\mathcal{L} M_n^2(t)
 \coloneqq& \lim_{h \to 0_+} \frac{1}{h} \ev \mleft[ M_n^2(t+h)-M_n^2(t) \, | \, \mathcal{F}_n(t) \mright]
\stackrel{(*)}{=} \lim_{h \to 0_+} \frac{1}{h} \ev \mleft[ (M_n(t+h)-M_n(t))^2 \, | \, \mathcal{F}_n(t) \mright]=\\
&\lim_{h \to 0_+} \frac{1}{h} \ev \mleft[ (Q_n(t+h)-Q_n(t))^2 \, | \, \mathcal{F}_n(t) \mright]=
\frac{1}{n} \lambda_n Q_n(t) \ev_{A_n(t),n,t}\mleft[\mathscr{C}^2\mright].\qedhere
\end{align*}
\end{proof}

\section{Properties of the window curve}\label{sec:ODE}

In this section we study the window curve $w_n(t)$ defined in \eqref{eq:wn_def}. The proofs of the statements of this section are detailed in Appendix~\ref{ap:window_curve_proofs}.

\begin{claim}[ODE of the window curve]
The function $w_n(t)$ defined in \eqref{eq:wn_def} solves the ODE
\begin{equation}\label{eq:w_ODE}
\dot{w}_n(t) = -\lambda_n \cdot \frac{w_n(t)}{1-w_n(t)} +\frac{w_n(t)}{t}, \quad w_n(0)=0, \quad \dot{w}_n(0)=1.
\end{equation}
\end{claim}

\begin{proof}
It follows from~\eqref{eq:alive_ODE} and~\eqref{eq:wn_def}.
\end{proof}
The ODE~\eqref{eq:w_ODE} is neither linear nor time-homogeneous.
We define an auxiliary function~$v_n(\cdot)$ that arises as the solution of an autonomous ODE, which can be viewed as an approximation of~\eqref{eq:w_ODE}.

\begin{definition}[Time-homogenized window curve]
We define $v_n(t)$ as the solution of the ODE
\begin{equation}\label{eq:v_ODE}
\dot{v}_n(t) = -\lambda_n \cdot \frac{v_n(t)}{1-v_n(t)} + 1, \quad v_n(0)=0.
\end{equation}
\end{definition}

\begin{claim}[Properties of $v_n$]\label{cl:vn_properties}
The function $v_n(t)$ is increasing and $\lim v_n(t) = 1/(1+\lambda_n)$ as $t \to \infty$. Therefore, we can define its inverse.
\end{claim}

\begin{proof}
Follows from the fact that the stationary point of~\eqref{eq:v_ODE} is $v_n=1/(1+\lambda_n)$.
\end{proof}

\begin{definition}[Inverse of $v_n$]\label{def:t_v_inv}
Denote by $\mathrm{t}_n(\cdot)$ the inverse function of $v_n(\cdot)$,
i.e., 
\begin{equation}\label{eq:t_def_as_v_inverse}
\mathrm{t}_n\colon \left[0, \frac{1}{1+\lambda_n} \right) \to \mathbb{R}_+, \qquad
\begin{array}{ll}
\text{$\mathrm{t}_n(v_n(t))=t$ for any $t \in \mathbb{R}_+$}, \\
\text{$v_n(\mathrm{t}_n(v))=v$ for any $0 \leq v < \frac{1}{1+\lambda_n}$}.
\end{array} 
\end{equation}

Let us also introduce the notation
\begin{equation}\label{eq:t_star_def}
t_n^* \coloneqq \mathrm{t}_n \mleft(\frac{1}{1+2 \lambda_n}\mright).
\end{equation}
\end{definition}

\begin{claim}[Arriving at $\lambda_n$-criticality before $t_n^*$]\label{cl:tc_less_than_t*}
Recall the definition of $t_n^c$ from~\eqref{eq:tnc_def}. We have
\begin{equation}\label{eq:tc_smaller_than_t*}
t_n^c < t_n^*.
\end{equation}
\end{claim}

We prove Claim~\ref{cl:tc_less_than_t*} in Appendix~\ref{ap:window_curve_proofs}.

\begin{lemma}[Bounds on the window curve]\label{lem:w_bounds}
For all $t\ge 0$,
\begin{equation}\label{eq:wnvnbound}
0\le w_n(t)\le v_n(t)\le \tfrac{1}{1+\lambda_n}.
\end{equation}
\end{lemma}

We prove Lemma~\ref{lem:w_bounds} in Appendix~\ref{ap:window_curve_proofs}.

\begin{claim}[Bounds on $\dot{v}_n$]
For any $t \in [0, t_n^*]$, we have
\begin{equation}\label{eq:v_dot_lower_bound}
\ffrac 12 \leq \dot{v}_n(t) \leq 1.
\end{equation}
\end{claim}

\begin{proof}
By~\eqref{eq:v_ODE}, $\dot{v}_n(t)$ is decreasing and $\dot{v}_n(0) = 1$. Moreover, by~\eqref{eq:v_ODE}, \eqref{eq:t_def_as_v_inverse} and~\eqref{eq:t_star_def}, we have ${\dot{v}_n(t_n^*) = 1/2}$.
\end{proof}

\begin{claim}[Bound on $t_n^*$] We have
\begin{equation}\label{eq:tstar_n_less_than_2}
t_n^*<2.
\end{equation}
\end{claim}
\begin{proof} Definition~\ref{def:t_v_inv} and \eqref{eq:v_dot_lower_bound} give
$\frac{1}{1+2\lambda_n}=v_n(t_n^*)=\int_0^{t_n^*} \dot{v}_n(s)\, \mathrm{d}s \geq \frac{t_n^*}{2}$, so $t_n^* \leq \frac{2}{1+2\lambda_n}<2$.
\end{proof}

\begin{claim}[Upper bounds by substitution] If $\lambda_n \leq 1/2$ and $t \le t_n^*$, then we have
\begin{equation}\label{eq:substition-integral}
\int \limits_0^t \frac{1}{1-v_n(u)} \, \mathrm{d}u \le 2 \ln \mleft( \frac{1}{\lambda_n} \mright), \qquad
\int \limits_0^t \frac{1}{(1-v_n(u))^k} \, \mathrm{d}u \le \frac{2}{k-1} \cdot \lambda_n^{1-k}, \quad k=2,\,3,\,\ldots
\end{equation}
\end{claim}

\begin{proof} Let $k \in \mathbb{N}_+$.
Using the substitution $v=v_n(u)$ we obtain
\begin{equation}\label{int_bound_subbb}
\int \limits_0^{t} \frac{1}{(1-v_n(u))^k}\, \mathrm{d}u \stackrel{\eqref{eq:v_dot_lower_bound} }{\leq} 2 \int \limits_0^{v_n(t)} \frac{1}{(1-v)^k}\, \mathrm{d}v \stackrel{(*)}{\leq} 
2 \int \limits_0^{\frac{1}{1+2 \lambda_n}} \frac{1}{(1-v)^k}\, \mathrm{d}v
 \stackrel{\lambda_n \leq \frac{1}{2}}{\leq} 2 \int \limits_0^{1-\lambda_n} \frac{1}{(1-v)^k}\, \mathrm{d}v,
\end{equation}
where $(*)$ follows from our assumption $t \le t_n^*$, the monotonicity of $v_n(t)$ (see Claim~\ref{cl:vn_properties}) and~\eqref{eq:t_star_def}.
Now both of the inequalities of~\eqref{eq:substition-integral} follow by evaluating the integral on the r.h.s.\ of~\eqref{int_bound_subbb}.
\end{proof}

\begin{claim}[$q_n(t)$ is close to $1$ before $t_n^*$]\label{claim_q_close_to_one}
If $\lambda_n \leq 1/2$ and $t \le t_n^*$, then we have $0 \leq 1 -q_n(t) \leq 2 \lambda_n \ln \mleft( \frac{1}{\lambda_n} \mright) $.
\end{claim}
\begin{proof}
Recalling \eqref{eq:alive_ODE} we see that $q_n(t)$ is a decreasing function of $t$, thus $0 \leq 1 -q_n(t)$ follows from $q_n(0)=1$.
As for the other bound, we have
\begin{equation*}
1-q_n(t) \stackrel{\eqref{eq:alive_ODE}}{=} \lambda_n \int \limits_0^t \frac{ q_n(u)}{1-u \cdot q_n(u)}\, \mathrm{d}u \leq \lambda_n \int \limits_0^t \frac{1}{1-u \cdot q_n(u)}\, \mathrm{d}u \stackrel{\eqref{eq:wn_def}, \eqref{eq:wnvnbound}}{\leq} \lambda_n \int \limits_0^t \frac{1}{1-v_n(u)} \, \mathrm{d}u \stackrel{\eqref{eq:substition-integral}}{\leq} 2 \lambda_n \ln \mleft( \frac{1}{\lambda_n} \mright).\qedhere
\end{equation*}
\end{proof}

\begin{notation}[Scaled distance of window curve from criticality]
Let us define $g_n(t)$ by
\begin{equation}\label{eq:def_g_n_t}
g_n(t)\coloneqq\frac{1-w_n(t)}{\lambda_n}.
\end{equation}
\end{notation}

\begin{lemma}[Properties of $g_n(t)$]\label{lem:gn_properties}
Recall $t_n^c$ from~\eqref{eq:tnc_def} and consider $g_n(t)$ defined in \eqref{eq:def_g_n_t}. We have
\begin{equation}\label{eq:gnt_bounds}
1 \leq \liminf_{n \to \infty} \inf_{t_n^c \leq t \leq T} g_n(t) \leq \limsup_{n \to \infty} \sup_{t_n^c \leq t \leq T} g_n(t) \leq 4 \vee T.
\end{equation}

We have
\begin{equation}\label{eq:gnt_asymp_t}
\lim \limits_{s \to \infty} \limsup_{n \to \infty} \sup_{t_n^c + s \cdot \lambda_n \leq t \leq T} |g_n(t)-t|=0.
\end{equation}
\end{lemma}

We prove Lemma~\ref{lem:gn_properties} in Appendix~\ref{ap:window_curve_proofs}.
We have already mentioned the relevance of~\eqref{eq:gnt_asymp_t} in Remark~\ref{rem:window_curve}. Let us also mention that \eqref{eq:gnt_asymp_t} is the reason why the parameter~$t$ appears on the r.h.s.\ of~\eqref{eq:OU_SDE} and~\eqref{eq:defnU(s)}.

\section{Arrival of the window process at \texorpdfstring{$\lambda_n$}{lambda-n}-criticality}\label{sec:arrive-crit}

In this section we study $W_n(t)$ before it arrives to the $\lambda_n$-critical state and prove Lemma~\ref{lem:Z_bound} for $t \le t_n^c$.

First, we state a stronger statement and we show that it indeed implies Lemma~\ref{lem:Z_bound} for $t \le t_n^c$.

\begin{proposition}[Concentration before the arrival at $\lambda_n$-criticality]\label{prop:arrival_at_OU}
There exists a constant $C \in \mathbb{R}_+$ such that for any $K>0$ there exists $n_0=n_0(K)$ such that for any $n \geq n_0$, we have
\begin{equation}\label{eq:arrive_at_crit}
\pr \mleft( \max_{0 \leq t \leq t_n^c } |W_n(t)-w_n(t)| \geq K \cdot (n \cdot \lambda_n )^{-1/2} \mright) \leq \frac{C}{ K^2}.
\end{equation}
\end{proposition}

We prove Proposition~\ref{prop:arrival_at_OU} later in this section. Before that, we deduce the main result of this section.

\begin{proof}[Proof of Lemma~\ref{lem:Z_bound} for $t \le t_n^c$]
Our goal is to show~\eqref{eq:window_fluctuation_bound} for $t \le t_n^c$, i.e., that for any $\varepsilon>0$ there exists $K<+\infty$ such that $
\sup_{n \in \mathbb{N}} \sup_{ 0 \leq t \leq t_n^c} \pr( \,|Z_n(t)| \geq K ) \leq \varepsilon$.
Observe that by~\eqref{eq:Z_n_t_def}, equation~\eqref{eq:arrive_at_crit} is equivalent to 
$\pr( \max_{0 \leq t \leq t_n^c } |Z_n(t)| \geq K) \leq \frac{C}{ K^2}$ 
and the bound \eqref{eq:window_fluctuation_bound} for $t \le t_n^c$
immediately follows from this.
\end{proof}

Now we turn to the proof of Proposition~\ref{prop:arrival_at_OU}. We first recall a classical result in analysis.

\begin{claim}[Grönwall's inequality]\label{cl:gronwall}
Let $0<T$, let $\beta \colon [0, T) \to \mathbb{R}_+$ be a continuous function, let $\alpha >0$ and let $D\colon [0, T) \to \mathbb{R}_+$ be a piecewise continuous function. Moreover let us assume that
\begin{equation*}
D(t) \leq \alpha + \int \limits_0^t \beta(u) D(u) \, \mathrm{d}u, \qquad \forall\; t \in [0, T).
\end{equation*}
Then we have
\begin{equation*}
D(t) \leq \alpha \cdot \exp \mleft( \int \limits_0^t \beta(u) \, \mathrm{d}u \mright), 
\qquad \forall\; t \in [0, T).
\end{equation*}
\end{claim}

We need some further definitions before we can state Lemma~\ref{lem:tau_not_happens_early}, which immediately implies Proposition~\ref{prop:arrival_at_OU}. Recall that $Q_n(t)$ denotes the ratio of alive vertices (see~\eqref{eq:Qn_def}) and that $q_n(t)$ is the solution of the differential equation~\eqref{eq:alive_ODE}.

\begin{definition}[Difference to the ODE approximation]
We define
\begin{equation}\label{eq:delta_as_difference}
\delta_n(t) \coloneqq |Q_n(t)-q_n(t)|.
\end{equation}
\end{definition}

Let us also recall the martingale $M_n(t)$ from~\eqref{eq:alive_corrector_martingale} and $t_n^*$ from~\eqref{eq:t_star_def}. Now we define some stopping times, which depend on~$n$ and the parameter~$K$ (but we omit this dependence from our notation).

\begin{definition}[Stopping times]
Let us fix $n$ and $K$. Then we define the following stopping times
\begin{align} 
\label{eq:tau_delta_def}
\tau_\delta &\coloneqq \min \mleft\{ t \; : \; \delta_n(t) \geq K \cdot \left( n \lambda_n \right)^{-1/2} \mright\}, \\
\label{eq:tau_M_def}
\tau_M &\coloneqq \min \mleft\{ t \; : \; |M_n(t)| \geq \frac{K}{4\mathrm{e}^{4}} \cdot \left( n \lambda_n \right)^{-1/2} \mright\}.
\end{align}

\begin{equation}\label{eq:tau_def}
\tau \coloneqq \tau_\delta \wedge \tau_M \wedge t_n^*.
\end{equation}
\end{definition}

Note that $\tau_{\delta}$, $\tau_M$ and $\tau$ depend on $n$ and $K$, but we omit this dependence for ease of notation.

\begin{lemma}[$\tau$ does not happen early]\label{lem:tau_not_happens_early}
There exists $C \in \mathbb{R}_+$ such that for any $K>0$, there exists $n_0=n_0(K)$ such that for any $n \geq n_0$, we have 
\begin{equation}\label{eq:tau_smaller_than_t*}
\pr \mleft( \tau < t_n^* \mright) \leq \frac{C}{ K^2}.
\end{equation}
\end{lemma}
As soon as we prove Lemma~\ref{lem:tau_not_happens_early}, the proof of Proposition~\ref{prop:arrival_at_OU} will follow.

\begin{proof}[Proof of Proposition~\ref{prop:arrival_at_OU}]
If $\tau = t_n^*$, then for any $0 \le t \le t_n^c$, we have
\begin{equation}\label{eq:W_w_difference_small}
|W_n(t)-w_n(t)| \stackrel{\eqref{eq:Wn_def},\, \eqref{eq:wn_def},\, \eqref{eq:delta_as_difference}}{=} t\delta_n(t) \stackrel{\eqref{eq:tc_smaller_than_t*},\, \eqref{eq:tau_delta_def},\, \eqref{eq:tau_def}}{<} t \cdot K \cdot (n \lambda_n )^{-1/2} \stackrel{\eqref{eq:tc_smaller_than_t*}, \eqref{eq:tstar_n_less_than_2}}{<} 2 \cdot K \cdot (n \lambda_n )^{-1/2}.
\end{equation}
Therefore, we obtain the desired \eqref{eq:arrive_at_crit} (with $2K$ in place of $K$):
\begin{equation*}
\pr \mleft( \max_{0 \leq t \leq t_n^c } |W_n(t)-w_n(t)| \geq 2 \cdot K \cdot (n \cdot \lambda_n )^{-1/2} \mright) \stackrel{\eqref{eq:W_w_difference_small}}{\leq} \pr \mleft( \tau < t_n^* \mright) \stackrel{\eqref{eq:tau_smaller_than_t*}}{\le} \frac{C}{K^2}.\qedhere
\end{equation*}
\end{proof}

In order to prove Lemma~\ref{lem:tau_not_happens_early}, we first prove the following statement.
\begin{lemma}[Connection between stopping times]\label{lem:martingale_jumps_first}
For any $K>0$, there exists $n_0=n_0(K)$ such that for any $n \geq n_0$, we almost surely have 
\begin{equation}\label{martingale_jumps_first}
\tau_M \wedge t_n^* \leq \tau_\delta \wedge t_n^*.
\end{equation} 
\end{lemma}

\begin{remark}
Loosely speaking, Lemma~\ref{lem:martingale_jumps_first} states that if $|M_n(t)|$ is well-behaved, then $\delta_n(t)$ is also well-behaved. We will use Gr\"onwall's inequality to show this implication. Given Lemma~\ref{lem:martingale_jumps_first}, we will only need to control the submartingale $M_n^2(t)$ in order to prove Lemma~\ref{lem:tau_not_happens_early}.
\end{remark}

To prove Lemma~\ref{lem:martingale_jumps_first}, we first derive an upper bound on $\delta_n(t)$ in Lemma~\ref{lem:alive_delta_upper_bound_integral_1}. 

\begin{lemma}[Upper bound on $\delta_n(t)$]\label{lem:alive_delta_upper_bound_integral_1}
For any $K>0$, there exists $n_0=n_0(K)$ such that for any $n \geq n_0$ and any $t < \tau$, we almost surely have 
\begin{equation}\label{eq:alive_delta_upper_bound_integral}
\delta_n(t) \leq \frac{1}{2\mathrm{e}^{4}} K \cdot (n \lambda_n)^{-1/2}+ 8 \lambda_n \int \limits_0^{v_n(t)} \frac{ \delta_n(\mathrm{t}_n(v)) }{( 1-v )^2} \, \mathrm{d} v, 
\end{equation}
where $\mathrm{t}_n(v)$ is the inverse function of $v_n(t)$, see~\eqref{eq:v_ODE} and~\eqref{eq:t_def_as_v_inverse}.
\end{lemma}

In order to show the upper bound \eqref{eq:alive_delta_upper_bound_integral} on $\delta_n(t)$, we will use the integral equation~\eqref{eq:diff_of_random_and_deterministic_alive} and some technical estimates detailed in Lemma~\ref{lem:if_t_less_than_tau_estimates}.

\begin{claim}[Integral equation for $\delta_n(t)$]
\begin{equation}\label{eq:diff_of_random_and_deterministic_alive}
Q_n(t)-q_n(t) = \int \limits_0^t \left( \mathcal{L} Q_n(u)- \dot{q}_n(u) \right)\, \mathrm{d}u + M_n(t)
\end{equation}
\end{claim}
\begin{proof}
Follows from~\eqref{eq:alive_corrector_martingale} and the Newton--Leibniz formula for $q_n(t)$.
\end{proof}

\begin{lemma}[Bounds on the number and fraction of alive vertices]\label{lem:if_t_less_than_tau_estimates}
Given $K>0$, there exists $n_0= n_0(K)$ such that for $n \geq n_0$ and for any $t<\tau$, we have
\begin{align}
\label{eq:alive_simple_lower_bound}
A_n(t) &\geq \ffrac12 n,\\
\label{eq:alive_blow_up_upper_bound}
\frac{1}{1- t \cdot Q_n(t)} &\leq \frac{2}{1-v_n(t)},\\
\label{eq:both_subcrit}
1-t \cdot (q_n(t) \vee Q_n(t) ) &> \ffrac{1}{2}\left(1-v_n(t)\right)>0,\\
\label{eq:window_sufficient_condition}
W_n(t) &<1-(A_n(t))^{-1/3},
\end{align}
where $v(\cdot)$ is the solution of the ODE~\eqref{eq:v_ODE}.
\end{lemma}

Note that the usefulness of~\eqref{eq:window_sufficient_condition} is that we can use formulas \eqref{eq:first_moment_of_component_size} and~\eqref{eq:general_moment_of_component_size_crude} of Claim~\ref{cl:component_size} later in this section.

\begin{proof}
Let $t < \tau$. It follows from our assumption $\lambda_n \to 0$ (cf.~\eqref{eq:subcrit_lambda_assumption}), \eqref{eq:tau_def} and Claim~\ref{claim_q_close_to_one} that we have $q_n(t) \to 1$ as $n \to \infty$. On the other hand, $\delta_n(t) \to 0$ as $n \to \infty$ (see~\eqref{eq:subcrit_lambda_assumption}, \eqref{eq:tau_delta_def}, \eqref{eq:tau_def}). Hence, we obtain~\eqref{eq:alive_simple_lower_bound}:
\begin{equation*}
A_n(t) \stackrel{\eqref{eq:Qn_def}}{=} n \cdot Q_n(t) \stackrel{\eqref{eq:delta_as_difference} }{\geq} n \cdot \left(q_n(t)-\delta_n(t) \right) \geq \ffrac12 n,
\end{equation*}
where the last inequality follows from the fact that $q_n(t) \approx 1$ and $\delta_n(t) \approx 0$.

Now we turn to the proof of~\eqref{eq:alive_blow_up_upper_bound}. By~\eqref{eq:subcrit_lambda_assumption}, there exists $n_0= n_0(K)$ such that for $n \geq n_0$, we have
\begin{equation}\label{eq:lambda_Delta_ineq}
2 K \cdot (n \lambda_n)^{-1/2} \leq
\frac{\lambda_n}{1+2\lambda_n}=
\frac{1}{2}\left(1-\frac{1}{1+2\lambda_n} \right) \stackrel{\eqref{eq:t_def_as_v_inverse}, \eqref{eq:t_star_def} }{=} \ffrac{1}{2}\left(1-v_n(t_n^*)\right).
\end{equation}
Also note that we have
\begin{equation}\label{eq:v_t_a_t_ineq}
1-v_n(t) \stackrel{\eqref{eq:wnvnbound} }{\leq} 1-w_n(t) \stackrel{\eqref{eq:wn_def}}{=}1-t \cdot q_n(t).
\end{equation}
Next we use $t < \tau=\tau_\delta \wedge \tau_M \wedge t_n^*$ (see~\eqref{eq:tau_def}) to bound
\begin{equation}\label{eq:t_delta_n_bound}
t \delta_n(t) \stackrel{\eqref{eq:tstar_n_less_than_2} }{\leq} 2 \delta_n(t) \stackrel{ \eqref{eq:tau_delta_def} }{\leq} 2 K \cdot (n \lambda_n)^{-1/2}
\stackrel{ \eqref{eq:lambda_Delta_ineq} }{\leq} \ffrac{1}{2}\left(1-v_n(t_n^*)\right) \stackrel{ \text{Claim }\ref{cl:vn_properties} }{\leq} 
\ffrac{1}{2}\left(1-v_n(t)\right).
\end{equation}
Using this we obtain
\begin{equation}\label{eq:db_bound}
1- t Q_n(t) \geq 1-t q_n(t)-t |Q_n(t)-q_n(t) | \stackrel{ \eqref{eq:delta_as_difference} }{=} 1-t q_n(t) - t \delta_n(t) \stackrel{ \eqref{eq:v_t_a_t_ineq}, \eqref{eq:t_delta_n_bound} }{\geq} \ffrac{1}{2}\left(1-v_n(t)\right),
\end{equation}
from which~\eqref{eq:alive_blow_up_upper_bound} immediately follows. 

The proof of~\eqref{eq:both_subcrit} follows from~\eqref{eq:v_t_a_t_ineq}, \eqref{eq:db_bound} and the inequality $v_n(t)\leq v_n(t_n^*)=\frac{1}{1+2\lambda_n}<1$.

Finally, we prove~\eqref{eq:window_sufficient_condition}:
\begin{align*}
1-W_n(t) \stackrel{\eqref{eq:Wn_def},\, \eqref{eq:db_bound}}{\ge} \ffrac 12 (1-v_n(t)) \stackrel{\eqref{eq:wnvnbound}}{\ge} \frac{\lambda_n}{2(1+\lambda_n)} \stackrel{\eqref{eq:subcrit_lambda_assumption},\, \eqref{eq:alive_simple_lower_bound}}{\gg} (A_n(t))^{-1/3}.\tag*{\qedhere}
\end{align*}
\end{proof}

\begin{proof}[Proof of Lemma~\ref{lem:alive_delta_upper_bound_integral_1}]
Let us first bound the term inside the integral in \eqref{eq:diff_of_random_and_deterministic_alive}:
\begin{align}
\begin{split}\label{eq:derivative_difference_bound_alive_1}
| \mathcal{L} Q_n(t)- \dot{q}_n(t) | 
\stackrel{\eqref{eq:alive_ODE},\, \eqref{eq:neighbor_drift}}{ \leq}&
\lambda_n \left| Q_n(t) \ev_{A_n(t),n,t}[\mathscr{C}]- \frac{Q_n(t)}{1-t \cdot Q_n(t)} \right| + \lambda_n \left| \frac{Q_n(t)}{1-t \cdot Q_n(t)}- \frac{q_n(t)}{1-t \cdot q_n(t)} \right| \stackrel{\eqref{eq:first_moment_of_component_size},\, \eqref{eq:both_subcrit} }{\leq} \\
&\lambda_n \frac{\widetilde{C}}{A_n(t)} \frac{1}{(1-t \cdot Q_n(t))^4}+ \lambda_n \frac{|Q_n(t)-q_n(t)| }{( 1-t \cdot [q_n(t) \vee Q_n(t) ] )^2}.
\end{split}
\end{align}

We obtain for any $u < \tau$ 
\begin{equation}\label{eq:derivative_difference_bound_alive_2}
| \mathcal{L} Q_n(u)- \dot{q}_n(u) | 
 \stackrel{ \eqref{eq:alive_simple_lower_bound},\,
\eqref{eq:alive_blow_up_upper_bound},\,
\eqref{eq:both_subcrit}, \, \eqref{eq:derivative_difference_bound_alive_1} }{\leq} 
 \lambda_n \frac{2\widetilde{C}}{n} \frac{2^4}{(1-v_n(u))^4}+
 \lambda_n \frac{4 \delta_n(u) }{( 1-v_n(u) )^2}.
\end{equation}
Recall the notion of $\mathrm{t}_n(v)$ from~\eqref{eq:t_def_as_v_inverse}.
For any $t < \tau$, we have
\begin{align*}
\delta_n(t) \stackrel{\eqref{eq:delta_as_difference}}{=}& 
|Q_n(t)-q_n(t)| \stackrel{\eqref{eq:diff_of_random_and_deterministic_alive}}{\leq}
\int \limits_0^t \left| \mathcal{L} Q_n(u)- \dot{q}_n(u) \right| \, \mathrm{d}u + |M_n(t)| 
\stackrel{\eqref{eq:derivative_difference_bound_alive_2} }{\leq}\\
&C' \frac{\lambda_n}{n} \int \limits_0^t \frac{1}{(1-v_n(u))^4} \, \mathrm{d} u +
\lambda_n \int \limits_0^t \frac{ 4\delta_n(u) }{( 1-v_n(u) )^2} \, \mathrm{d} u + |M_n(t)|
\stackrel{\eqref{eq:v_dot_lower_bound},\, \eqref{eq:substition-integral}, \, \eqref{eq:tau_M_def},\, \eqref{eq:tau_def}}{\leq}\\
& \frac23 C' \frac{\lambda_n^{-2}}{n} +
\lambda_n \int \limits_0^{v_n(t)} \frac{ 8 \delta(\mathrm{t}_n(v)) }{( 1-v )^2} \, \mathrm{d} v 
+ \frac{K}{4\mathrm{e}^{4}} \cdot (n \lambda_n)^{-1/2} \stackrel{(*)}{\leq} 8 \lambda_n \int \limits_0^{v_n(t)} \frac{ \delta(\mathrm{t}_n(v)) }{( 1-v )^2} \, \mathrm{d} v + \frac{K}{2\mathrm{e}^{4}} \cdot (n \lambda_n)^{-1/2},
\end{align*}
where $(*)$ holds for $n \geq n_0(K)$ by~\eqref{eq:subcrit_lambda_assumption}.
The proof of Lemma~\ref{lem:alive_delta_upper_bound_integral_1} is complete.
\end{proof}

\begin{proof}[Proof of Lemma~\ref{lem:martingale_jumps_first}]
If we define $\delta_n^*(v)= \delta_n(\mathrm{t}_n(v))$, then $\delta^*(v_n(t))= \delta_n(t)$ (see~\eqref{eq:t_def_as_v_inverse}), and we can rewrite~\eqref{eq:alive_delta_upper_bound_integral} as
\begin{equation}\label{eq:integral_ineq_for_diff_alive}
\delta_n^*(v) \leq
\frac{1}{2\mathrm{e}^{4}} K \cdot (n \lambda_n)^{-1/2} +
8 \lambda_n \int \limits_0^{v} \frac{ \delta_n^*(u) }{( 1-u )^2} \, \mathrm{d}u, \quad v < v(\tau).
\end{equation}
Now we can use Grönwall's inequality (i.e., Claim~\ref{cl:gronwall}) to infer from~\eqref{eq:integral_ineq_for_diff_alive} that for any $v < v(\tau)$, we have
\begin{equation*}
\delta_n^*(v) \leq \frac{1}{2\mathrm{e}^{4}} K \cdot (n \lambda_n)^{-1/2} \cdot \exp \mleft( \int \limits_0^{v} \frac{ 8 \lambda_n }{( 1-u )^2} \, \mathrm{d} u \mright) \stackrel{(*)}{\leq}
\frac{K}{2} \cdot (n \lambda_n)^{-1/2},
\end{equation*}
where $(*)$ follows from $v<v(\tau)<\frac{1}{1+2 \lambda_n}$ by~\eqref{eq:t_star_def}. Thus we have
\begin{equation}\label{eq:gronwall_control_on_delta}
\delta_n(t)=|Q_n(t)-q_n(t)| \leq \frac{K}{2} \cdot (n \lambda_n)^{-1/2}, \qquad t < \tau.
\end{equation}

Let us denote by $\delta_n(\tau_+)$, $\delta_n(\tau_-)$, $M_n(\tau_+)$ and $M_n(\tau_-)$ the one-sided limits of $\delta_n$ and $M_n$ at $\tau$. If we let~${t \to \tau}$ in~\eqref{eq:diff_of_random_and_deterministic_alive}, we see that we have
\begin{equation}\label{eq:jumps_of_delta_and_M}
\delta_n(\tau_+)-\delta_n(\tau_-) \leq | M_n(\tau_+)-M_n(\tau_-)|.
\end{equation}
We will prove~\eqref{martingale_jumps_first} by contradiction. If \eqref{martingale_jumps_first} does not hold, then we have
\begin{equation}\label{eq:contradiction_stopping_times}
\tau= \tau_\delta<\tau_M,
\end{equation}
therefore, 
\begin{equation}\label{eq:tau_delta_before_after}
\delta_n(\tau_-) \stackrel{\eqref{eq:gronwall_control_on_delta} }{\leq} \frac{K}{2} \cdot (n \lambda_n)^{-1/2},
\qquad \delta_n(\tau_+) \stackrel{\eqref{eq:tau_delta_def},\, \eqref{eq:contradiction_stopping_times} }{\geq} K \cdot (n \lambda_n)^{-1/2}.
\end{equation}
Also
\begin{equation*}
|M_n(\tau_-)| \stackrel{\eqref{eq:tau_M_def},\,\eqref{eq:contradiction_stopping_times}}{\leq} \frac{K}{4\mathrm{e}^{4}} \cdot \left( n \lambda_n \right)^{-1/2}, 
 \qquad 
|M_n(\tau_+)| \stackrel{\eqref{eq:tau_M_def},\,\eqref{eq:contradiction_stopping_times}}{\leq} \frac{K}{4\mathrm{e}^{4}} \cdot \left( n \lambda_n \right)^{-1/2},
\end{equation*}
therefore, $| M_n(\tau_+)-M_n(\tau_-)| < \frac{K}{2} \cdot (n \lambda_n)^{-1/2}$, but \eqref{eq:tau_delta_before_after} implies $\delta_n(\tau_+)-\delta_n(\tau_-) \geq \frac{K}{2} \cdot (n \lambda_n)^{-1/2}$, which is a contradiction with~\eqref{eq:jumps_of_delta_and_M}. This completes the proof of Lemma~\ref{lem:martingale_jumps_first}.
\end{proof}

\begin{proof}[Proof of Lemma~\ref{lem:tau_not_happens_early}]
The definition of $\tau$ (see~\eqref{eq:tau_def}) and Lemma~\ref{lem:martingale_jumps_first} imply that for any $n \ge n_0(K)$ (determined in Lemma~\ref{lem:martingale_jumps_first}), we have
\begin{equation}\label{eq:def_tau_2}
\tau = \tau_M \wedge t_n^*.
\end{equation}

Recall the notion of the martingale $N_n(t)$ from~\eqref{eq:quadratic_corrector_martingale}. Now it follows from the optional stopping theorem that $\ev[N_n(\tau)]=0$, therefore,

\begin{align}
\begin{split}\label{eq:M_second_moment_upper_bound}
\ev \mleft( M_n^2(\tau) \mright) \stackrel{\eqref{eq:quadratic_corrector_martingale}}{=}& \ev \mleft( \int \limits_0^\tau \mathcal{L} M_n^2(u) \, \mathrm{d}u \mright)
\stackrel{\eqref{eq:quadratic_variation}}{=} \ev \mleft( \int \limits_0^\tau \frac 1n \lambda_n Q_n(u) \ev_{A_n(u),n,u}[\mathscr{C}^2] \, \mathrm{d}u \mright) 
\stackrel{\eqref{eq:general_moment_of_component_size_crude}}{\leq}\\
&\widetilde{C}_2 \frac{\lambda_n}{n} \ev \mleft( \int \limits_0^\tau \frac{1}{(1-u \cdot Q_n(u) )^3 } \, \mathrm{d}u \mright)
\stackrel{\eqref{eq:tau_def},\, \eqref{eq:alive_blow_up_upper_bound} }{\leq}
8\widetilde{C}_2 \frac{\lambda_n}{n} \int \limits_0^{t_n^*} \frac{1}{(1-v_n(u))^3}\, \mathrm{d}u \stackrel{\eqref{eq:substition-integral}}{\leq} 8\widetilde{C}_2 (n\lambda_n)^{-1}.
\end{split}
\end{align}
Now we can apply Chebyshev's inequality in $(*)$ below to deduce
\begin{equation*}
\pr \mleft( \tau < t_n^* \mright) 
\stackrel{\eqref{eq:tau_M_def},\, \eqref{eq:def_tau_2} }{\leq} 
\pr \mleft( |M_n(\tau)| \geq \frac{K}{4\mathrm{e}^{4}} \cdot \left( n \lambda_n \right)^{-1/2} \mright)
\stackrel{(*)}{\leq}
16\mathrm{e}^8 \frac{ \ev \mleft( M_n^2(\tau) \mright) }{ K^2 \cdot \left( n \lambda_n \right)^{-1} } 
\stackrel{\eqref{eq:M_second_moment_upper_bound} }{\leq} \frac{C}{K^2}.
\end{equation*}
This proves Lemma~\ref{lem:tau_not_happens_early}, which completes the proof of Proposition~\ref{prop:arrival_at_OU}.
\end{proof}

\section{Proof of the Ornstein--Uhlenbeck scaling limit result }\label{sec:SOC}

In this section our goal is to prove Lemma~\ref{lem:Z_bound} for $t \ge t_n^c$ and Theorem~\ref{thm:OU_scaling}.

\subsection{Infinitesimal generator asymptotics}\label{subsec:infgen_asymp}

The main result of Section~\ref{subsec:infgen_asymp} is Proposition~\ref{prop:exp_drift_precise_ou_lyap_momgen}, which compares the infinitesimal generator of the centered scaled window process to that of an Ornstein--Uhlenbeck (O.-U.) process.
We use it for two purposes: in Section~\ref{subsec:stays_crit}, using a Lyapunov argument, we show that the window process stays $\lambda_n$-critical on the long run, and we prove the O.-U.~scaling limit process convergence result in Section~\ref{subsec:ou_proof}.

Recall the definition of $Z_n(t)$ from~\eqref{eq:Z_n_t_def}, the definition of $t_n^c$ from~\eqref{eq:tnc_def} and that we have fixed $T \in (1,+\infty)$.

We now define an auxiliary sequence $(\kappa_n)$ that goes to infinity slow enough so that the infinitesimal generator of $Z_n(t)$ is close enough to that of an O.-U.~process (the drift and diffusion coefficients of which do depend on $t$) as long as $t \in [t_n^c,T]$ and $Z_n(t) \in [-\kappa_n,\kappa_n]$ (cf.~Proposition~\ref{prop:exp_drift_precise_statement}), but on the other hand~$\kappa_n$ goes to infinity fast enough so that the process $Z_n(t)$ never leaves the interval $[-\kappa_n,\kappa_n]$ on $[t_n^c,T]$ with high probability (we will show this in Section~\ref{subsec:stays_crit}).

\begin{definition}[Safety barrier]\label{def:kappa}
Under the assumption \eqref{eq:subcrit_lambda_assumption}, we can define $(\kappa_n)_{n=1}^{\infty}$ in such a way that
\begin{equation}\label{eq:kappa_threshold_assumption}
\ln \mleft( \frac{1}{\lambda_n}\mright) \ll \kappa_n, \qquad \lambda_n \ll \frac{1}{\kappa_n}, \qquad \kappa_n^{4/3} n^{-1/3} \ll \lambda_n
\end{equation}
hold. Let us fix such a sequence $(\kappa_n)_{n=1}^{\infty}$. 
\end{definition}

\begin{remark}[Construction of $\kappa_n$]\label{rem:kappa_construction}
A possible choice for the sequence $(\kappa_n)_{n=1}^\infty$ satisfying the conditions in Definition~\ref{def:kappa} is to take the geometric mean of the lower bound and the more restrictive upper bound of~\eqref{eq:kappa_threshold_assumption}. In particular, one may define:
\begin{equation}\label{eq:kappa_const}
\kappa_n \coloneqq \sqrt{ \ln\mleft(\frac{1}{\lambda_n}\mright) \cdot \min\mleft( \frac{1}{\lambda_n},\, (n\lambda_n^3)^{1/4} \mright) }.
\end{equation}
Under the assumption~\eqref{eq:subcrit_lambda_assumption}, the lower bound $\ln(1/\lambda_n)$ is asymptotically smaller than the minimum of the upper bounds $\lambda_n^{-1}$ and $(n\lambda_n^3)^{1/4}$, therefore, the construction in~\eqref{eq:kappa_const} satisfies~\eqref{eq:kappa_threshold_assumption}.
\end{remark}

Our next claim follows from~\eqref{eq:Wn_def}, \eqref{eq:Z_n_t_def} and Claim~\ref{cl:time_dep_inf_gen}.

\begin{claim}[$Z_n$ is Markov]\label{cl:Z_n_markov} 
The process~$(Z_n(t))_{t \geq 0 }$ is a continuous-time Markov process (albeit not time-homogeneous).
\end{claim}
\begin{definition}[Auxiliary O.-U.~process]
Given some parameters $n \in \mathbb{N}$ and $t \in \lbrack t_n^c, T \rbrack$, we consider the Ornstein--Uhlenbeck process $(\widehat{Z}_n^t(s))$, which solves the SDE
\begin{equation}\label{eq:ou_SDE_n_t_g}
\mathrm{d} \widehat{Z}_n^t(s)= - \frac{1}{(g_n(t))^2} \widehat{Z}_n^t(s)\, \mathrm{d}s + \sqrt{\frac{t}{(g_n(t))^3}} \, \mathrm{d} B(s),
\end{equation}
where $g_n(t)=(1-w_n(t))/\lambda_n$ was defined in \eqref{eq:def_g_n_t}.
\end{definition}

\begin{definition}[Infinitesimal generator of $\widehat{Z}_n^t$]
The infinitesimal generator of $(\widehat{Z}_n^t(s))$ is defined for any smooth test function $f\colon \mathbb{R} \to \mathbb{R}$ by
\begin{equation}\label{eq:OU_infgen_def}
\widetilde{L}(n,z,t,f)\coloneqq- \frac{1}{(g_n(t))^2} z f'(z) +
\frac 12 \frac{t}{(g_n(t))^3} f''(z) = \lim_{h \to 0_+} \frac 1h \ev \mleft( f(\widehat{Z}_n^t(h))-f(\widehat{Z}_n^t(0)) \, \middle| \, \widehat{Z}_n^t(0)= z \mright).
\end{equation}
\end{definition}
Thus if $f(z)=\mathrm{e}^{\beta z}$ then $\frac{-\widetilde{L}(n,z,t,f)}{\mathrm{e}^{\beta z}}= \frac{ \beta z}{(g_n(t))^2} - \frac{ \beta^2 t}{2(g_n(t))^3} $.

Recall that we denote by $\mathcal{L}$ the drift of a stochastic process (see~\eqref{eq:drift_def}). 
Now we state the result relating the infinitesimal generator of $Z_n(t)$ to that of the Ornstein--Uhlenbeck process $\widehat{Z}_n^t$.

\begin{proposition}[Exponential test function]\label{prop:exp_drift_precise_statement}
For any $\beta \in \mathbb{C}$, we have
\begin{equation}\label{eq:exp_drift_precise_statement}
\lim_{n \to \infty} \, \sup_{t_n^c \leq t \leq T} \indic{ |Z_n(t)| \leq \kappa_n }
\left| \lambda_n \frac{\mathcal{L} \mathrm{e}^{\beta Z_n(t)} }{\mathrm{e}^{\beta Z_n(t)} } + \frac{ \beta Z_n(t)}{(g_n(t))^2} - \frac{ \beta^2 t}{2(g_n(t))^3} \right| =0.
\end{equation}
\end{proposition}

Next we state and prove Proposition~\ref{prop:exp_drift_precise_ou_lyap_momgen}, which generalizes Proposition~\ref{prop:exp_drift_precise_statement}. In order to do so, we need some further notation.

\begin{notation}
Let $f\colon \mathbb{R} \to \mathbb{C}$ denote a smooth test function. We will now derive an infinitesimal generator formula for $Z_n(t)$. First, let us introduce the parameter $A \in \left\{ 1, \ldots, n \right\}$ and let us denote
\begin{equation}\label{eq:z_n_t_A_def}
z_n^A(t)= \sqrt{ n \lambda_n } \left( \frac{t A}{n} - w_n(t) \right).
\end{equation}
Note that if $A=A_n(t)$, then $z_n^A(t)=Z_n(t)$ by~\eqref{eq:Z_n_t_def}. Recall the notation $\ev_{A,n,t}$ and $\mathscr{C}$ from Notation~\ref{not:ER_notation_A_n_t} and the function $w_n(t)$ from~\eqref{eq:wn_def}. We define
\begin{align}
E(n,A,t,f) & \coloneqq\ev_{A,n,t}\mleft[ f\mleft(z_n^A(t)- t \sqrt{\frac{\lambda_n}{n}} \mathscr{C}\mright)-f\mleft(z_n^A(t)\mright) \mright], \label{eq:E_n_A_t_f_def}\\
H(n,A,t,f) & \coloneqq f'(z_n^A(t)) \sqrt{n \lambda_n} \left( \frac{A}{n} - \dot{w}_n(t) \right)+ \lambda_n A E(n,A,t,f). \label{eq:H_def}
\end{align}
\end{notation}

\begin{claim}[Infinitesimal generator of $Z_n(t)$]
The infinitesimal generator formula for $Z_n(t)$ is the following
\begin{equation}\label{eq:Z_n_t_infgen}
\mathcal{L} f(Z_n(t))= H(n,A_n(t),t,f).
\end{equation}
\end{claim}

\begin{proof}
Follows directly from~\eqref{eq:A_n_t_infgen}.
\end{proof}

\begin{notation}
Given $z , \beta \in \mathbb{R}$, we define the class of functions
\begin{equation}\label{eq:D_z_def}
\mathcal{D}(z, \beta)\coloneqq \Bigg\{ \, f: \mathbb{R} \to \mathbb{R} \, : \, f \in C^3(\mathbb{R}), \; |f'(z)|\leq |\beta|, \; |f''(z)|\leq |\beta|^2, \; \sup_{s \in \mathbb{R} } \frac{|f'''(s)|}{ \mathrm{e}^{|\beta (s-z)|}} \leq |\beta|^3 \, \Bigg\}.
\end{equation}
For $n \in \mathbb{N}$, $\beta \in \mathbb{R}$ and $T \in \mathbb{R}_+$, we define the set of good parameters
$\mathcal{R}(n,\beta,T) \subseteq \{ 1, \ldots, n \} \times \mathbb{R}_+ \times C^3(\mathbb{R})$~by
\begin{equation}\label{eq:R_n_def}
\mathcal{R}(n,\beta,T) \coloneqq \Big\{ \, (A,t,f) \, : \, t_n^c \leq t \leq T, \, |z_n^A(t)| \leq \kappa_n, \, f \in \mathcal{D}(z_n^A(t), \beta) \, \Big\}.
\end{equation}
\end{notation}

\begin{proposition}[Test function with bounded derivatives]\label{prop:exp_drift_precise_ou_lyap_momgen}
For any $\beta \in \mathbb{R}$, we have
\begin{equation}\label{eq:f_test_fn_drift_precise_statement}
\lim_{n \to \infty} \sup_{(A,t,f) \in \mathcal{R}(n,\beta,T) } \left| \lambda_n H(n,A,t,f) - \widetilde{L}(n,z_n^A(t),t,f) \right| =0.
\end{equation}
\end{proposition}

We prove Proposition~\ref{prop:exp_drift_precise_ou_lyap_momgen} in Section~\ref{subsec:inf_gen_proof}. Now we deduce Proposition~\ref{prop:exp_drift_precise_statement} from it.

\begin{proof}[Proof of Proposition~\ref{prop:exp_drift_precise_statement}]
Let $\beta \in \mathbb{C}$. Let us denote $f_z(s)=\mathrm{e}^{\beta(s-z)}$. Note that
\begin{equation*}
\mathcal{L} \mathrm{e}^{\beta Z_n(t)}=
\mathcal{L} f_0(Z_n(t)) \stackrel{\eqref{eq:Z_n_t_infgen}}{=}
H(n,A_n(t),t, f_0 ) \stackrel{\eqref{eq:H_def}}{=}
\mathrm{e}^{\beta Z_n(t)}
H(n,A_n(t),t, f_{Z_n(t)} ).
\end{equation*}
Also note that
$\widetilde{L}(n,z,t,f_z)= - \frac{ \beta z}{(g_n(t))^2} + \frac{ \beta^2 t}{2(g_n(t))^3}$. Thus we want to show
\begin{equation}\label{eq:exp_drift_precise_statement2}
\lim_{n \to \infty} \, \sup_{t_n^c \leq t \leq T} \indic{ |Z_n(t)| \leq \kappa_n }
\left| \lambda_n H(n,A_n(t),t, f_{Z_n(t)} ) - 
\widetilde{L}(n,Z_n(t),t,f_{Z_n(t)})\right| =0.
\end{equation}
Since we only stated Proposition~\ref{prop:exp_drift_precise_ou_lyap_momgen} for real-valued~$f$, we will prove~\eqref{eq:exp_drift_precise_statement2} by treating the real and imaginary parts of~$f_{Z_n(t)}$ separately. First we note that both~$H$ and~$\widetilde{L}$ are linear in~$f$. 
Next we observe that for any $z \in \mathbb{R}$, we have $\mathrm{Re}(f_z), \mathrm{Im}(f_z) \in \mathcal{D}(z, |\beta|)$ by~\eqref{eq:D_z_def} and Euler's formula. Thus Proposition~\ref{prop:exp_drift_precise_ou_lyap_momgen} indeed implies~\eqref{eq:exp_drift_precise_statement2}.
The proof of Proposition~\ref{prop:exp_drift_precise_statement} is complete.
\end{proof}

\subsection{Window process stays~\texorpdfstring{$\lambda_n$}{lambda-n}-critical}\label{subsec:stays_crit}

In this section, we prove Lemma~\ref{lem:Z_bound} for $t_n^c \le t \le T$ using Proposition~\ref{prop:arrival_at_OU} and Proposition~\ref{prop:exp_drift_precise_statement}.

Recall that we fixed a sequence $(\kappa_n)_{n=1}^\infty$ that satisfies \eqref{eq:kappa_threshold_assumption}.

\begin{definition}[Exit time from $(-\kappa_n, \kappa_n)$]
We define a stopping time
\begin{equation}\label{eq:tauZ_def}
\tau_Z \coloneqq \min \mleft\{\, t \ge t_n^c \; : \; |Z_n(t)| \geq \kappa_n \, \mright\}.
\end{equation}
\end{definition}

Note that $\tau_Z$ depends on $n$, but we omit this dependence for ease of notation. Note that by Proposition~\ref{prop:arrival_at_OU}, $\tau_Z > t_n^c$ with high probability.

\begin{lemma}[Upper bound for the infinitesimal generator]\label{lem:chZ_drift_upper_bound}
There exists $n_0 \in \mathbb{N}$ such that for any $n \ge n_0$ and for any $t_n^c \le t \le T$, we have
\begin{equation}\label{eq:chZ_drift_upper_bound}
\indic{t < \tau_Z} \cdot \mathcal{L}( \cosh(Z_n(t))) \le \indic{t < \tau_Z} \cdot \frac{1}{\lambda_n} \cdot \left( 27 - \frac{1}{32T^2} \cdot \cosh(Z_n(t)) \right).
\end{equation}
\end{lemma}

\begin{proof}
By Proposition~\ref{prop:exp_drift_precise_statement} and the definition \eqref{eq:tauZ_def} of $\tau_Z$, we know that for every $T>1$, there exists $n_0 \in \mathbb{N}$ such that for any $n \ge n_0$ and $t \in [t_n^c ,\, T]$, we have
\begin{equation}
\begin{alignedat}{2}\label{eq:exp_inf_gen_upper_bound}
&\indic{t < \tau_Z} \cdot \frac{\lambda_n \cdot \mathcal{L} \mleft( \mathrm{e}^{Z_n(t))} \mright)}{\mathrm{e}^{Z_n(t)}} &&\le
\indic{t < \tau_Z} \cdot \left[ \frac{1}{\left(g_n(t)\right)^2} \cdot \left( \frac{t}{2g_n(t)} - Z_n(t) \right) + \frac{1}{(4 \vee T)^2} \right], \\
&\indic{t < \tau_Z} \cdot \frac{\lambda_n \cdot \mathcal{L} \mleft( \mathrm{e}^{-Z_n(t))} \mright)}{\mathrm{e}^{-Z_n(t)}} &&\le
\indic{t < \tau_Z} \cdot \left[ \frac{1}{\left(g_n(t)\right)^2} \cdot \left( \frac{t}{2g_n(t)} + Z_n(t) \right) + \frac{1}{(4 \vee T)^2} \right].
\end{alignedat}
\end{equation}

Using that $\frac{t}{2g_n(t)} \le 1$ (this inequality follows from~\eqref{eq:gnt_bounds} for $t_n^c \le t \le 2$ and by~\eqref{eq:gnt_asymp_t} for $2 \le t \le T$) and
$\frac{1}{4 \vee T} \le \frac{1}{g_n(t)}$ (see~\eqref{eq:gnt_bounds}), rearranging \eqref{eq:exp_inf_gen_upper_bound}, we obtain
\begin{equation}\label{eq:chZ_drift_first_bound}
\indic{t < \tau_Z} \cdot \mathcal{L}( \cosh(Z_n(t))) \le \frac{\indic{t < \tau_Z}}{2\lambda_n \cdot \left(g_n(t)\right)^2 } \cdot \left( \mathrm{e}^{Z_n(t)} \cdot (2-Z_n(t)) + \mathrm{e}^{-Z_n(t)} \cdot (2+Z_n(t)) \right).
\end{equation}
We will show that the r.h.s.\ of~\eqref{eq:chZ_drift_upper_bound} is larger than the r.h.s.\ of~\eqref{eq:chZ_drift_first_bound}, i.e.,
\begin{equation}\label{eq:chZ_drift_enough}
\frac{1}{2\left(g_n(t)\right)^2 } \cdot \left( \mathrm{e}^{Z_n(t)} \cdot (2-Z_n(t)) + \mathrm{e}^{-Z_n(t)} \cdot (2+Z_n(t)) \right) \le 27 - \frac{1}{32T^2} \cdot \cosh(Z_n(t)).
\end{equation}
Since both sides of~\eqref{eq:chZ_drift_enough} are even functions of $Z_n(t)$, it is enough to show \eqref{eq:chZ_drift_enough} if~${Z_n(t)\ge 0}$.

If $0 \le Z_n(t) < 3$, then
\begin{equation}\label{eq:Z_less_than_3}
\frac{1}{2\left(g_n(t)\right)^2 } \cdot \left( \mathrm{e}^{Z_n(t)} \cdot (2-Z_n(t)) + \mathrm{e}^{-Z_n(t)} \cdot (2+Z_n(t)) \right) <\frac{5 \cosh(3)}{2\left(g_n(t)\right)^2 } \stackrel{\eqref{eq:gnt_bounds}}{\le} 26 \stackrel{(*)}{<} 27 - \frac{1}{32T^2} \cdot \cosh(Z_n(t)),
\end{equation}
where at $(*)$ we used that $\frac{1}{32T^2} \cdot \cosh(Z_n(t)) < 1$ if $0 \le Z_n(t) < 3$.

On the other hand, if $Z_n(t) \ge 3$, then we have
\begin{align}
\begin{split}\label{eq:Z_greater_than_3}
\frac{1}{2\left(g_n(t)\right)^2 } \cdot \left( \mathrm{e}^{Z_n(t)} \cdot (2-Z_n(t)) + \mathrm{e}^{-Z_n(t)} \cdot (2+Z_n(t)) \right) \stackrel{(**)}{\le}& \frac{1}{2\left(g_n(t)\right)^2 } \left( -\mathrm{e}^{Z_n(t)} + 2 \right) \stackrel{\eqref{eq:gnt_bounds}}{\le} \frac{-\mathrm{e}^{Z_n(t)}}{2(4\vee T)^2} + 1 \le\\
& \frac{-\cosh(Z_n(t))}{2(4\vee T)^2} + 2 \le 27 - \frac{1}{32T^2} \cdot \cosh(Z_n(t)),
\end{split}
\end{align}
where at $(**)$ we used that $(2+x)\mathrm{e}^{-x} \le 2$ for any $x \ge 0$.
Putting together \eqref{eq:Z_less_than_3} and~\eqref{eq:Z_greater_than_3}, we obtain~\eqref{eq:chZ_drift_enough}. This concludes the proof of~\eqref{eq:chZ_drift_upper_bound}.
\end{proof}

\begin{remark}
Note that a statement similar to~\eqref{eq:chZ_drift_upper_bound} also holds if the Lyapunov function is an even function that does not grow exponentially (unlike $\cosh(\cdot)$), but only has a slower (e.g.\ linear) growth rate. We will use the fact that the Lyapunov function has an exponential growth rate in the proof of Lemma~\ref{lem:opt_stop}.
\end{remark}

\begin{lemma}[$Z_n$ does not exit the $\kappa_n$-window]\label{lem:opt_stop}
We have $\tau_Z \geq T$ with high probability.
\end{lemma}

\begin{proof} Recall the definition~\eqref{eq:Z_n_t_def} of $Z_n(t)$.
Proposition~\ref{prop:arrival_at_OU} implies that for any $\varepsilon>0$, there exists $K'$ such that if $n$ is large enough, then we have
\begin{equation}\label{eq:Z(tc)_bound_for_tauZ}
\pr(|Z_n(t_n^c)| \ge K') \le \frac{\varepsilon}{2}.
\end{equation}
Let us now condition on the event $\{Z_n(t_n^c)=z\}$ for some $z \in (-K',K')$ for which $\mathbb{P}\left( Z_n(t_n^c)=z \right)>0$. By the Markov property of $(Z_n(\cdot))$ (cf.\ Claim~\ref{cl:Z_n_markov}), we will work with a process $Z_n(\cdot)$ that starts at time $t_n^c$ from the initial condition $ Z_n(t_n^c)= z \in (-K',K')$.

Let $\tau_Z^T \coloneqq \tau_Z \wedge T$. By~\eqref{eq:chZ_drift_upper_bound}, we know that for any $n \ge n_0$ (where $n_0$ is determined in Lemma~\ref{lem:chZ_drift_upper_bound}), the process
\begin{equation*}
\frac{1}{27} \cosh (Z_n(t \wedge \tau_Z^T)) - \frac{(t \wedge \tau_Z^T) - t_n^c}{\lambda_n} , \qquad t \ge t_n^c
\end{equation*}
is a supermartingale. Thus by the optional stopping theorem, we obtain
\begin{equation}\label{eq:supermartingale}
\ev \mleft[ \frac{\cosh( Z_n(\tau_Z^T))}{27}-\frac{\tau_Z^T-t_n^c}{\lambda_n} \mright] \leq
\ev\left[ \frac{\cosh (Z_n(t_n^c)) }{27} \right]
\leq
\frac{\cosh (K')}{27}.
\end{equation}
Thus we obtain 
\begin{equation*}
\frac{T}{\lambda_n} \geq \ev \mleft[ \frac{\tau_Z^T}{\lambda_n} \mright] \stackrel{\eqref{eq:supermartingale}}{\geq} \frac{1}{27} \left( \ev \mleft[ \cosh( Z_n(\tau_Z^T) )\mright]-\cosh (K') \right)\stackrel{\eqref{eq:tauZ_def}}{\geq}
\frac{1}{27} \left[ \pr \mleft( \tau_Z < T \mright) \cdot \cosh( \kappa_n ) -\cosh(K') \right],
\end{equation*}
therefore,
\begin{equation}\label{eq:tauZ_prob_upper_bound}
\pr \mleft(\tau_Z < T \mright) \leq \frac{27T + \lambda_n \cdot \cosh (K')}{\lambda_n \cosh( \kappa_n ) }.
\end{equation}
By~\eqref{eq:subcrit_lambda_assumption} and~\eqref{eq:kappa_threshold_assumption}, the r.h.s.\ of~\eqref{eq:tauZ_prob_upper_bound} converges to $0$ as $n \to \infty$. Thus, returning to the version of $Z_n(\cdot)$, which starts at time $t=0$ from the initial state dictated by $A_n(0)=n$, if $n$ is large enough, we have
\begin{equation*}
\pr \mleft(|Z_n(t_n^c)| < K',\, \tau_Z < T \mright) \le \frac{\varepsilon}{2}.
\end{equation*}
Together with~\eqref{eq:Z(tc)_bound_for_tauZ}, this completes the proof of Lemma~\ref{lem:opt_stop}.
\end{proof}

Our next lemma is an ingredient of the proof of Lemma~\ref{lem:Z_bound} for $t_n^c \le t \le T$.

\begin{lemma}[Lyapunov function control on the long run]\label{lem:lyap_control}
For any arbitrarily fixed $K'>0$, let us introduce the process
\begin{equation}\label{eq:xi_def}
\xi_n(t)\coloneqq \indic{|Z_n(t_n^c)| < K', \, \tau_Z > t} \cdot \cosh(Z_n(t))
\end{equation}
and the notation $\rho \coloneqq \left( 27 \cdot 32T^2 \right) \vee \cosh \mleft( K' \mright)$. Then $\ev(\xi_n(t)) \le \rho$ for any $t \in [t_n^c, T]$.
\end{lemma}
\begin{proof}
Observe that $\xi_n(t_n^c) < \rho$. 
By~\eqref{eq:chZ_drift_upper_bound}, we know that if $n$ is large enough, then
\begin{equation}\label{eq:xi_drift}
\mathcal{L} \xi_n(t) \le \frac{1}{32T^2 \cdot \lambda_n} \cdot \left( \rho - \xi_n(t) \right) \qquad \forall \, t \in [t_n^c, T].
\end{equation}

Taking the expected value of both sides of~\eqref{eq:xi_drift}, we obtain that
\begin{equation}\label{eq:xi_expectation}
\frac{\mathrm{d}}{\mathrm{d}t}\ev(\xi_n(t)) \le \frac{1}{32T^2 \cdot \lambda_n} \cdot \left( \rho - \ev(\xi_n(t)) \right) \qquad \forall \, t \in [t_n^c, T].
\end{equation}

We have $\ev(\xi_n(t_n^c)) < \rho$. Therefore, inequality~\eqref{eq:xi_expectation} implies $\ev(\xi_n(t)) \le \rho$ for any $t \in [t_n^c, T]$.
\end{proof}

\begin{proof}[Proof of Lemma~\ref{lem:Z_bound}]
We have already treated the case $t \leq t_n^c$ in Proposition~\ref{prop:arrival_at_OU}. In order to finish the proof in the $t_n^c \le t \le T$ case, we use the union bound:
\begin{equation}\label{eq:Z_n_t_uniqon_3}
\pr \mleft( |Z_n(t)| > K \mright) \stackrel{\eqref{eq:xi_def}}{\le} \pr \mleft( |Z_n(t_n^c)| \ge K' \mright) + \pr(\tau_Z \le t) + \pr \mleft( \xi_n(t) > \cosh(K) \, \mright), \qquad t \in [t_n^c, T].
\end{equation}
Thus the proof of
\begin{equation*}
\lim_{K \to \infty} \sup_{n \in \mathbb{N}} \sup_{ 0 \leq t \leq T} \pr \mleft( |Z_n(t)| > K \mright)=0
\end{equation*}
follows if we apply Proposition~\ref{prop:arrival_at_OU}, Lemma~\ref{lem:opt_stop} and Lemma~\ref{lem:lyap_control} (together with Markov's inequality) respectively to the three terms on the r.h.s.\ of~\eqref{eq:Z_n_t_uniqon_3} and minimize the r.h.s.\ over the value of~$K'$.
\end{proof}

\subsection{Proof of the Ornstein--Uhlenbeck scaling limit result}\label{subsec:ou_proof}

In this section we prove Theorem~\ref{thm:OU_scaling}. First, we define the space $D[a,b]$.

\begin{definition}[Space $D$]
Let $-\infty <a < b <+\infty$. The space $D[a,b]$ consists of functions $z\colon [a,b] \to \mathbb{R}$ that are right-continuous and have left limits.
\end{definition}

Billingsley introduced a metric on $D[a,b]$ (see \cite[Section 14]{B68}) that induces the Skorokhod topology (see~\cite[Theorem~14.1]{B68} or \cite[Theorem~(6.2) in Section~8]{D96}), and $D[a,b]$ is a complete, separable metric space under this metric (see \cite[Theorem~14.2]{B68} or \cite[Theorem~(6.1) in Section~8]{D96}).

Recall the processes $\widetilde{Z}_n^{t_n}$, $\widetilde{Z}^t$ from equations \eqref{eq:tilde_Z_def} and~\eqref{eq:OU_SDE}. In order to prove Theorem~\ref{thm:OU_scaling}, we use the following classical result on convergence of processes.

\begin{theorem}[Convergence in $D$, {\cite[Theorem 15.1]{B68}}]\label{thm:convergence_in_D}
Let $\{X_n\}$ be a sequence of random elements of $D[a,b]$ and $X \in D[a,b]$.
If the sequence $X_n$ is tight on $D[a,b]$ and the finite dimensional distributions converge weakly, i.e., for any points $a \le s_1< s_2< \ldots < s_r \le b$, we have
\begin{equation}\label{eq:joint_dist_conv}
\big( X_n(s_1), \, X_n(s_2), \, \ldots, \, X_n(s_r) \big) \Rightarrow
\big( X(s_1), \, X(s_2), \, \ldots, \, X(s_r) \big) \qquad \text{as } n \to \infty,
\end{equation}
then $X_n(\cdot) \Rightarrow X(\cdot)$ in $D[a,b]$.
\end{theorem}

In order to apply Theorem~\ref{thm:convergence_in_D} for the sequence $\widetilde{Z}_n^{t_n}$ and the process $\widetilde{Z}^t$ on $D[-S,S]$, we need the next two lemmas that we will prove later.

\begin{lemma}[Tightness of $\widetilde{Z}_n^{t_n}$]\label{lem:tightness_Ztilde}
The sequence $\widetilde{Z}_n^{t_n}$ is tight on $D[-S,S]$ for any $S < +\infty$.
\end{lemma}

We prove Lemma~\ref{lem:tightness_Ztilde} in Section~\ref{subsec:Z_tightness}.

\begin{lemma}[Convergence of joint characteristic functions]\label{lem:Z_joint_char_fn_conv}
Let us fix $-S \le s_1 < s_2 < \ldots < s_r \le S$ and coefficients $\nu_1, \, \nu_2, \, \ldots, \, \nu_r \in \mathbb{R}$ arbitrarily. Then we have
\begin{equation}\label{eq:conv_of_characteristic_fn}
\lim \limits_{n \to \infty} \ev \mleft[ \exp \mleft( i \cdot \sum \limits_{j=1}^r \nu_j \cdot \widetilde{Z}_n^{t_n}(s_j) \mright) \mright] = \ev \mleft[ \exp \mleft( i \cdot \sum \limits_{j=1}^r \nu_j \cdot \widetilde{Z}^t(s_j) \mright) \mright],
\end{equation}
where $i = \sqrt{-1}$ is the imaginary unit.
\end{lemma}

We prove Lemma~\ref{lem:Z_joint_char_fn_conv} in Section~\ref{subsec:Z_char_fn_conv}.

Now we can conclude Theorem~\ref{thm:OU_scaling}.

\begin{proof}[Proof of Theorem~\ref{thm:OU_scaling}]
By Theorem~\ref{thm:convergence_in_D} and Lemma~\ref{lem:tightness_Ztilde}, it is enough to show that equation~\eqref{eq:joint_dist_conv} holds for the sequence $\widetilde{Z}_n^{t_n}$ and the process $\widetilde{Z}^t$, i.e., the finite dimensional marginal distributions of $\widetilde{Z}_n^{t_n}$ converge to the finite dimensional marginal distributions of $\widetilde{Z}^t$. Since joint convergence in distribution is equivalent to the point-wise convergence of the joint characteristic functions, Lemma~\ref{lem:Z_joint_char_fn_conv} completes the proof of Theorem~\ref{thm:OU_scaling}.
\end{proof}

\subsubsection{Tightness of~\texorpdfstring{$\widetilde{Z}_n^{t_n}$}{Z}}\label{subsec:Z_tightness}

In this section our goal is to show the tightness of $\{\widetilde{Z}_n^{t_n}(\cdot)\}$, i.e., Lemma~\ref{lem:tightness_Ztilde}. For that purpose, we use Aldous' sufficient condition for tightness.

\begin{theorem}[Aldous criterion for tightness, {\cite[Theorem 1]{A78}}]\label{thm:aldous_crit}
Let $a<b \in \mathbb{R}$ and $\{X_n\}$ be a sequence of random elements of $D[a,b]$. Assume that $\{X_n(s) \}$ is tight on the line for each $s \in [a,b]$, moreover, that
\begin{equation}\label{eq:aldous_crit}
X_n \big((\tau_n + \delta_n) \wedge b \big) - X_n(\tau_n) \stackrel{\pr}{\to} 0
\quad \text{as} \quad
n \to \infty
\end{equation}
for each sequence $\{\tau_n, \, \delta_n\}$ that satisfies 
\begin{enumerate}
\item for each $n$, $\tau_n$ is a stopping time on the process $\{X_n(s), \, a \le s \le b\}$, with respect to the natural $\sigma$-fields, and~$\tau_n$ takes only finitely many values,
\item for each $n$, $\delta_n$ is a nonnegative constant, and $\delta_n \to 0$ as $n \to \infty$,
\end{enumerate}
then $\{X_n\}$ is tight in $D[a,b]$.
\end{theorem}

In order to prove the tightness of $\{\widetilde{Z}_n^{t_n}(\cdot)\}$ (i.e., Lemma~\ref{lem:tightness_Ztilde}), we will check the conditions of Theorem~\ref{thm:aldous_crit}. By Lemma~\ref{lem:Z_bound} the sequence $\{ \widetilde{Z}_n^{t_n}(s) \}$ of real-valued random variables is tight for each fixed $s \in \mathbb{R}$ and any sequence~$(t_n)$ that satisfies the criteria of Theorem~\ref{thm:OU_scaling}.
We will now state a strengthening of this statement, which will turn out to be useful when we check the other conditions of Theorem~\ref{thm:aldous_crit}

\begin{lemma}[Tightness of the supremum of $\widetilde{Z}^{t_n}(\cdot)$ on compact intervals]\label{lem:tight_sup_Z} For any $S \in \mathbb{R}_+$ and any sequence~$(t_n)$ that satisfies the criteria of Theorem~\ref{thm:OU_scaling}, we have
\begin{equation*}
\lim_{K \to \infty} \sup_n \pr \mleft( \max \limits_{-S \le s \le S} \left| \widetilde{Z}_n^{t_n}(s) \right| \geq K \mright)=0.
\end{equation*}
\end{lemma}

\begin{proof}
Let us fix $S \in \mathbb{R}_+$, and define $a_n \coloneqq t_n-S \lambda_n$, $b_n \coloneqq t_n+S\lambda_n$ and $I_n\coloneqq[a_n,b_n]$.
By our assumptions on $(t_n)$, we have $I_n \subseteq [t_n^c,T]$ if $n$ is large enough.

It follows from~\eqref{eq:tilde_Z_def} that we have $\max_{-S\leq s\leq S}\left|\widetilde Z_n^{t_n}(s)\right| =\max_{u\in I_n}|Z_n(u)|$.
Furthermore, Lemma~\ref{lem:Z_bound} gives
$ \lim_{K' \to \infty} \sup_n \pr \mleft( \left| Z_n(a_n) \right| \geq K' \mright)=0$, thus we only need to show that for any $K' \in \mathbb{R}_+$ we have
\begin{equation}\label{eq:tight_sup_Z_init}
\lim_{K \to \infty} \sup_n \pr \mleft( \max \limits_{u \in I_n} \left| Z_n(u) \right| \geq K \, \text{ and } \, \left| Z_n(a_n) \right| \leq K' \mright)=0.
\end{equation}

We now use the drift estimate of Lemma~\ref{lem:chZ_drift_upper_bound}.
More precisely, by Proposition~\ref{prop:exp_drift_precise_statement}, exactly as in the proof of
Lemma~\ref{lem:chZ_drift_upper_bound}, for all sufficiently large $n$ and all
$u\in I_n$, we have
\begin{equation}\label{eq:infgen_bound_chZ}
\indic{|Z_n(u)|\leq \kappa_n}\, \cdot \mathcal{L}\cosh(Z_n(u))\leq \indic{ |Z_n(u)|\leq \kappa_n}\cdot \frac{27}{\lambda_n}.
\end{equation}
Since $\kappa_n\to\infty$, the inequality $K'\leq K<\kappa_n$ holds if $n$ is large enough. Let us define the stopping time
$\tau_n\coloneqq\inf\{u\in I_n:\ |Z_n(u)|\geq K\}$. 
Using~\eqref{eq:infgen_bound_chZ}, we see that the stopped process
\begin{equation*}
Y_n(u)\coloneqq \cosh\mleft(Z_n(u\wedge\tau_n)\mright)-\frac{27}{\lambda_n} \big((u\wedge\tau_n)-a_n\big), \qquad a_n\leq u\leq b_n,
\end{equation*}
is a supermartingale. Let $E_n =\{ \left| Z_n(a_n) \right| \leq K' \}$.
An application of the optional stopping theorem gives
\begin{equation*}
\ev\mleft[\, \cosh\mleft(Z_n(\tau_n \wedge b_n)\mright)\indic{E_n}\, \mright]
\leq
\ev\mleft[\,\cosh\mleft(Z_n(a_n)\mright)\indic{E_n}\,\mright]+
\frac{27}{\lambda_n}\,(b_n-a_n) 
\leq
\cosh\mleft(K'\mright)+54S.
\end{equation*}
Markov's inequality then gives
$\pr \left(E_n\cap\{\tau_n< b_n\}\right) \leq \frac{\cosh(K')+54S}{\cosh(K)}$, which completes the proof of~\eqref{eq:tight_sup_Z_init}.
\end{proof}

\begin{proof}[Proof of Lemma~\ref{lem:tightness_Ztilde}] As we have already mentioned, we will check the conditions of Theorem~\ref{thm:aldous_crit}. We have already discussed that $\{ \widetilde{Z}_n^{t_n}(s) \}$ is tight for each fixed $s \in \mathbb{R}$. It remains to verify Aldous' increment condition~\eqref{eq:aldous_crit}. Let $\{\tau_n,\delta_n\}$ be a sequence as in Theorem~\ref{thm:aldous_crit}, applied to the process~$\widetilde Z_n^{t_n}$ on~${[-S,S]}$. We want to show $\widetilde Z_n^{t_n}((\tau_n+\delta_n)\wedge S) -\widetilde Z_n^{t_n}(\tau_n) \xrightarrow{\mathbb P}0$. 
Given any $K \in \mathbb{R}_+$, let us define the stopping time $\tau_n^K \coloneqq \tau_n \wedge \min\{ \, s \in [-S,S] \,: \, | \widetilde{Z}_n^{t_n}(s) | \geq K \, \}$.
By Lemma~\ref{lem:tight_sup_Z}, it is enough to show that
\begin{equation}\label{eq:stop_aldous}
\widetilde Z_n^{t_n}((\tau_n^K+\delta_n)\wedge S) -\widetilde Z_n^{t_n}(\tau_n^K) \xrightarrow{\mathbb P}0
\end{equation}
holds for any $K \in \mathbb{R}_+$. Let 
$\tau_n^*\coloneqq\min\{ \, s \,: \; |\widetilde Z_n^{t_n}(\tau_n^K \vee s ) -\widetilde Z_n^{t_n}(\tau_n^K)| \geq 1 \, \}$. Recall from~\eqref{eq:OU_infgen_def} that for any fixed $t \in \mathbb{R}_+$ and $z^* \in \mathbb{R}$ the O.-U.~process $\widehat{Z}_n^t(\cdot)$ satisfies
\begin{equation*}
\lim_{h \to 0_+} \frac 1h \ev \mleft[ \left( \widehat{Z}_n^t(s+h)-z^* \right)^2 \; \middle| \; \widehat{Z}_n^t(s) \mright]=
- \frac{1}{(g_n(t))^2} \widehat{Z}_n^t(s) \cdot 2 \cdot \left( \widehat{Z}_n^t(s)-z^*\right) +
 \frac{t}{(g_n(t))^3}.
\end{equation*}
Thus, using Proposition~\ref{prop:exp_drift_precise_ou_lyap_momgen}, the definitions of $\tau_n^K$ and $\tau_n^*$ as well as the properties of $(t_n)$ and $g_n(\cdot)$ (cf.~\eqref{eq:gnt_asymp_t}), we obtain that for any $\varepsilon>0$ we have
\begin{equation}\label{eq:infgen_square_Z}
\mathcal{L}\left( \widetilde Z_n^{t_n}((\tau_n^K \vee s) \wedge \tau_n^* ) -\widetilde Z_n^{t_n}(\tau_n^K) \right)^2 \leq \left[(K + 1) \cdot 2 +T +\varepsilon \right] \cdot \indic{ \tau_n^K < s < \tau_n^* } 
\end{equation}
if $n$ is large enough (where, with a slight abuse of notation we switched to the scaled time variable $s$ in the definition of $\mathcal{L}$). Integrating~\eqref{eq:infgen_square_Z} and taking expectations we obtain
\begin{equation*}
\ev\mleft[ \left( 
\widetilde Z_n^{t_n}((\tau_n^K +\delta_n) \wedge \tau_n^* ) -\widetilde Z_n^{t_n}(\tau_n^K) \right)^2\mright] \leq \delta_n \cdot \left[(K + 1) \cdot 2 +T +\varepsilon \right],
\end{equation*}
and then Markov's inequality and $\delta_n \to 0$ together imply $\lim_{n \to \infty}\mathbb{P}(\tau_n^* < \tau_n^K +\delta_n)=0$ and~\eqref{eq:stop_aldous}.
Thus all conditions of Theorem~\ref{thm:aldous_crit} are satisfied, and the sequence $\widetilde Z_n^{t_n}(\cdot)$ is tight in $D[-S,S]$. The proof of Lemma~\ref{lem:tightness_Ztilde} is complete.
\end{proof}

\subsubsection{Convergence of joint characteristic functions}\label{subsec:Z_char_fn_conv}

In this section our goal is to prove the convergence of the joint characteristic function of $\widetilde{Z}_n^{t_n}(\cdot)$ to the joint characteristic function of $\widetilde{Z}^t(\cdot)$, i.e., Lemma~\ref{lem:Z_joint_char_fn_conv}.

Recall that in Lemma~\ref{lem:Z_joint_char_fn_conv} we fixed $s_1 < s_2 < \ldots < s_r$ in the interval $[-S,S]$ and coefficients $\nu_1, \, \nu_2, \, \ldots, \nu_r$.

\begin{notation}
Let us introduce the following notation for the joint characteristic functions that appeared in Lemma~\ref{lem:Z_joint_char_fn_conv}:
\begin{align}
\varphi_n^{t_n} &\coloneqq \ev \mleft[ \exp \mleft( i \cdot \sum \limits_{j=1}^r \nu_j \cdot \widetilde{Z}_n^{t_n}(s_j) \mright) \mright],\label{eq:varphi_n_def}\\
\varphi^t &\coloneqq \ev \mleft[ \exp \mleft( i \cdot \sum \limits_{j=1}^r \nu_j \cdot \widetilde{Z}^t(s_j) \mright) \mright].\label{eq:varphi^t_def}
\end{align}
\end{notation}

Using the above notation, our goal is to show that $\lim \limits_{n \to \infty} \varphi_n^{t_n} = \varphi^t$ (cf.~\eqref{eq:conv_of_characteristic_fn}).

Now we introduce a Doob-martingale related to the joint characteristic function of $\widetilde{Z}^t$, and derive an explicit formula for it.

\begin{definition}[Martingale related to the joint characteristic function]
Let us denote by $(\widetilde{\mathcal{F}}^t(\cdot))$ the filtration generated by the stationary Ornstein--Uhlenbeck process $(\widetilde{Z}^t(\cdot))$, and introduce the Doob-martingale
\begin{equation}\label{eq:varphi^t(s)_def}
\varphi^t(s) \coloneqq \ev \mleft[ \exp \mleft( i \cdot \sum \limits_{j=1}^r \nu_j \cdot \widetilde{Z}^t(s_j) \mright) \; \middle| \; \widetilde{\mathcal{F}}^t(s) \mright], \qquad s \in \mathbb{R}.
\end{equation}
\end{definition}

\begin{lemma}[Explicit formula for $\varphi^t(s)$]\label{lem:varphi_explicit}
Let $s \in \mathbb{R}$ and $k \in \{0,\,1,\,\ldots,\,r\}$ be the index such that $s_k \le s < s_{k+1}$ (where $s_0 \coloneqq -\infty$, $s_{r+1} \coloneqq +\infty$). Then we have
\begin{equation}\label{eq:varphi_explicit}
\varphi^t(s) = \exp \mleft[ i \cdot \left( \sum \limits_{j=1}^k \nu_j \cdot \widetilde{Z}^t(s_j) + \sum \limits_{j=k+1}^r \nu_j \cdot \tilde{\mu}^t_j(s) \right) - \frac 12 \sum \limits_{j=k+1}^r \left(\tilde{\nu}^t_j \cdot \tilde{\sigma}^t_j(s_{j-1} \vee s) \right)^2 \mright],
\end{equation}
where
\begin{equation}\label{eq:mu_sigma_def}
\tilde{\mu}^t_j(s) \coloneqq \widetilde{Z}^t(s) \cdot \exp \mleft( -\frac{1}{t^2} \cdot (s_j - s) \mright), \qquad (\tilde{\sigma}^t_j(s))^2 \coloneqq \frac 12 \left( 1-\exp \mleft(-\frac{2}{t^2} \cdot (s_j-s) \mright) \right)
\end{equation}
and
\begin{equation*}
\tilde{\nu}^t_j \coloneqq \nu_j + \sum \limits_{m=j+1}^r \nu_m \cdot \exp \mleft( -\frac{1}{t^2} \cdot (s_m-s_j) \mright).
\end{equation*}
\end{lemma}

\begin{proof}
First, observe that if $k=r$, then~\eqref{eq:varphi_explicit} follows from the fact that the expression in the expected value in~\eqref{eq:varphi^t(s)_def} is measurable with respect to $\widetilde{\mathcal{F}}^t(s)$. Now we prove~\eqref{eq:varphi_explicit} by descending induction on $k$. In particular, we show that if \eqref{eq:varphi_explicit} holds for any $s \ge s_l$ for some $l \in \{1, \, \ldots , \, r\}$, then it also holds for $s \in [s_{l-1}, s_l)$.

By~\eqref{eq:OU_SDE}, it is known that for any $u \in \mathbb{R}$ and $s_j > u$, the conditional distribution of $\widetilde{Z}^t(s_j)$ given $\widetilde{\mathcal{F}}^t(u)$ is~${N \mleft( \tilde{\mu}^t_j(u), (\tilde{\sigma}^t_j(u))^2 \mright)}$. Therefore,
\begin{equation}\label{eq:normal_characteristic_fn}
\ev \mleft[ \exp \mleft( i \cdot \nu \cdot \widetilde{Z}^t(s_j) \mright) \, \middle| \, \widetilde{\mathcal{F}}^t(u) \mright] = \exp \mleft( i \cdot \nu \cdot \tilde{\mu}^t_j(u) - \frac 12 \cdot \left(\nu \cdot \tilde{\sigma}^t_j(u) \right)^2 \mright) \qquad \forall \, \nu \in \mathbb{R}.
\end{equation}
Using the tower rule, equation~\eqref{eq:normal_characteristic_fn} and the induction hypothesis, we obtain~\eqref{eq:varphi_explicit} for $s \in [s_{l-1},s_l)$. The proof of Lemma~\ref{lem:varphi_explicit} is complete.
\end{proof}

We will show a conditional version of~\eqref{eq:conv_of_characteristic_fn} in Lemma~\ref{lem:conditional_char_fn_conv}. Before that, we introduce some notation.

\begin{notation}
Let $n\in \mathbb{N}$ and $s \in \mathbb{R}$. Then we denote the set of possible values of $\widetilde{Z}_n^{t_n}(s)$ by $\mathcal{D}_n^s$, i.e.,
\begin{equation*}
\mathcal{D}_n^s \coloneqq \left\{ \, \sqrt{n \lambda_n} \cdot \left( \frac kn \cdot (t_n + \lambda_n s) - w_n(t_n+\lambda_n s) \right) \; \middle| \; k =0,\,1,\,\ldots,\, n \, \right\}.
\end{equation*}
Given some $K \in \mathbb{R}_+$, let us also introduce the notation $\mathcal{D}^s_{n, K} \coloneqq \mathcal{D}_n^s \cap [-K, K]$.
\end{notation}

\begin{definition}[Conditional joint characteristic functions]
Let $S^*>0$ such that $-S^*<s_1$. Then we define the following conditional expectations
\begin{align}
\phi_n^{t_n}(-S^*,z) &\coloneqq \ev \mleft[ \exp \mleft( i \cdot \sum \limits_{j=1}^r \nu_j \cdot \widetilde{Z}_n^{t_n}(s_j) \mright) \; \middle| \; \widetilde{Z}_n^{t_n}(-S^*) = z \mright] , \quad z \in \mathcal{D}^{-S^*}_n,\label{eq:phi_n_def} \\
\phi^t(-S^*,z) &\coloneqq \ev \mleft[ \exp \mleft( i \cdot \sum \limits_{j=1}^r \nu_j \cdot \widetilde{Z}^t(s_j) \mright) \; \middle| \; \widetilde{Z}^t(-S^*) = z \mright],\quad z \in \mathbb{R}.\label{eq:phi^t_def}
\end{align}
\end{definition}

Note that, since $\widetilde{Z}^t$ is a Markov process, $\varphi^t(-S^*) = \phi^t \mleft(-S^*,\, \widetilde{Z}^t(-S^*) \mright)$ (see~\eqref{eq:varphi^t(s)_def} and~\eqref{eq:phi^t_def}).

\begin{lemma}[Convergence of conditional joint characteristic functions]\label{lem:conditional_char_fn_conv}
For any $S^*>0$ such that ${-S^*<s_1}$ and for any $K \in \mathbb{R}_+$, the conditional expectations defined in~\eqref{eq:phi_n_def} and~\eqref{eq:phi^t_def} satisfy
\begin{equation}\label{eq:conv_of_characteristic_fn_conditional}
\lim \limits_{n\to \infty} \sup \limits_{z \in \mathcal{D}^{-S^*}_{n,K}} \left| \phi_n^{t_n}(-S^*,z) - \phi^t(-S^*,z) \right| = 0.
\end{equation}
\end{lemma}

\begin{proof}
We introduce a sequence of processes (cf.~\eqref{eq:varphi_explicit} and~\eqref{eq:mu_sigma_def}):
\begin{equation}\label{eq:varphi_n_explicit}
\varphi_n^{t_n}(s) \coloneqq \exp \mleft[ i \cdot \left( \sum \limits_{j=1}^k \nu_j \cdot \widetilde{Z}_n^{t_n}(s_j) + \sum \limits_{j=k+1}^r \nu_j \cdot \tilde{\mu}^{t_n}_{n,j}(s) \right) - \frac 12 \sum \limits_{j=k+1}^r \left(\tilde{\nu}^{t_n}_j \cdot \tilde{\sigma}^{t_n}_j(s_{j-1} \vee s) \right)^2 \mright],
\end{equation}
where
\begin{equation}\label{eq:mu_n_def}
\tilde{\mu}^{t_n}_{n,j}(s) \coloneqq \widetilde{Z}_n^{t_n}(s) \cdot \exp \mleft( -\frac{1}{t_n^2} \cdot (s_j - s) \mright).
\end{equation}

Let us define
\begin{equation}\label{eq:phi_tilde_def}
\tilde{\varphi}_n^{t_n}(s) \coloneqq \varphi_n^{t_n}(s) - \int \limits_{-S^*}^s \mathcal{L}\varphi_n^{t_n}(u) \, \mathrm{d}u \qquad s \in [-S^*,S^*].
\end{equation}

Recall~$\kappa_n$ from Definition~\ref{def:kappa}.
The process $\tilde{\varphi}_n^{t_n}(s)$ is a martingale.
Let us define the stopping time $\tau \coloneqq \min \mleft\{\, s \geq -S^* \,:\, \left| \widetilde{Z}_n^{t_n}(s) \right| \geq \kappa_n \, \mright\}$. Let us denote $\ev_z(\cdot) \coloneqq \ev \mleft( \, \cdot \, \big| \, \widetilde{Z}_n^{t_n}(-S^*) = z \mright)$.
By the optional stopping theorem for any $\varepsilon >0$ we obtain
\begin{equation}\label{eq:phi_tilde_ost}
\ev_z \mleft[ \tilde{\varphi}_n^{t_n}((s_r+\varepsilon)\wedge \tau) \mright] = \ev_z \mleft[ \tilde{\varphi}_n^{t_n}(-S^*) \mright], \qquad z \in \mathcal{D}^{-S^*}_{n,K}.
\end{equation}
The conditional expectation that appeared on the left-hand side of~\eqref{eq:phi_tilde_ost} satisfies:
\begin{multline}\label{eq:phi_tilde_ost_lhs}
\ev_z \mleft[ \tilde{\varphi}_n^{t_n}((s_r+\varepsilon)\wedge \tau) \mright] 
\stackrel{\eqref{eq:phi_tilde_def}}{=} 
\ev_z \mleft[\varphi_n^{t_n}((s_r+\varepsilon)\wedge \tau)\mright] - 
\ev_z\mleft(\int \limits_{-S^*}^{(s_r+\varepsilon)\wedge \tau} \mathcal{L}\varphi_n^{t_n}(u) \, \mathrm{d}u\mright) =\\
\ev_z \mleft[\varphi_n^{t_n}(s_r+\varepsilon)\mright]+
\ev_z \mleft[ (\varphi_n^{t_n}(\tau)-\varphi_n^{t_n}(s_r+\varepsilon))\indic{\tau<s_r+\varepsilon} \mright] -
\ev_z\mleft(\int \limits_{-S^*}^{(s_r+\varepsilon)\wedge \tau} \mathcal{L}\varphi_n^{t_n}(u) \, \mathrm{d}u\mright)
\stackrel{\eqref{eq:phi_n_def},\, \eqref{eq:varphi_n_explicit} }{=} \\
\phi_n^{t_n}(-S^*,z) + \ev_z \mleft[ (\varphi_n^{t_n}(\tau)-\varphi_n^{t_n}(s_r+\varepsilon))\indic{\tau<s_r+\varepsilon} \mright] -
\ev_z\mleft(\int \limits_{-S^*}^{(s_r+\varepsilon)\wedge \tau} \mathcal{L}\varphi_n^{t_n}(u) \, \mathrm{d}u\mright).
\end{multline}

On the other hand, if we consider the right-hand side of~\eqref{eq:phi_tilde_ost}, $\tilde{\varphi}_n^{t_n}(-S^*)$ is measurable with respect to~$\widetilde{Z}_n^{t_n}(-S^*)$, therefore, we can determine its limit as $n \to \infty$:
\begin{align}
\begin{split}\label{eq:phi_tilde_ost_rhs}
&\lim \limits_{n \to \infty} \ev_z \mleft[ \tilde{\varphi}_n^{t_n}(-S^*) \, \mright] \stackrel{\eqref{eq:varphi_n_explicit},\, \eqref{eq:phi_tilde_def}}{=}\\
&\lim \limits_{n \to \infty} \exp \mleft[ i \cdot \sum \limits_{j=1}^r \nu_j \cdot z \cdot \exp \mleft(-\frac{1}{t_n^2}(s_j+S^*) \mright) - \frac 12 \sum \limits_{j=1}^r \left(\tilde{\nu}^{t_n}_j \cdot \tilde{\sigma}^{t_n}_j(s_{j-1} \vee -S^*) \right)^2 \mright] \stackrel{t_n \to t}{=}\\
&\exp \mleft[ i \cdot \sum \limits_{j=1}^r \nu_j \cdot z \cdot \exp \mleft(-\frac{1}{t^2}(s_j+S^*) \mright) - \frac 12 \sum \limits_{j=1}^r \left(\tilde{\nu}^{t}_j \cdot \tilde{\sigma}^{t}_j(s_{j-1} \vee -S^*) \right)^2 \mright] \stackrel{\eqref{eq:phi^t_def}, \,\eqref{eq:varphi^t(s)_def}, \, \eqref{eq:varphi_explicit}}{=} \phi^t(-S^*,z).
\end{split}
\end{align}

We will show that the last two terms on the right-hand side of~\eqref{eq:phi_tilde_ost_lhs} go to zero as $n \to \infty$. The second term on the r.h.s.\ of~\eqref{eq:phi_tilde_ost_lhs} goes to zero because $\varphi_n^{t_n}(\cdot)$ is bounded (cf.~\eqref{eq:varphi_n_explicit}) and for any $K \in \mathbb{R}_+$ we have $ \lim_{n \to \infty} \max_{ z \in \mathcal{D}^{-S^*}_{n,K} } \mathbb{P}(\tau<s_r+\varepsilon \, | \, \widetilde{Z}_n^{t_n}(-S^*) = z \, ) = 0$ (the proof of this is similar to that of Lemma~\ref{lem:opt_stop} for $t \geq t_n^c$ or the proof of Lemma~\ref{lem:tight_sup_Z}). Now we show that the third term on the r.h.s.\ of~\eqref{eq:phi_tilde_ost_lhs} goes to zero.
As soon as we prove this, equations~\eqref{eq:phi_tilde_ost}, \eqref{eq:phi_tilde_ost_lhs} and~\eqref{eq:phi_tilde_ost_rhs} imply the desired~\eqref{eq:conv_of_characteristic_fn_conditional}. 

In order to prove that the third term on the r.h.s.\ of~\eqref{eq:phi_tilde_ost_lhs} goes to zero, we derive a formula for $\mathcal{L}\varphi_n^{t_n}(s)$. Let $f\colon\mathbb{R}^2 \to \mathbb{R}$ be a smooth test function, and let us denote by~$f_1$ and~$f_2$ the partial derivatives of $f$ with respect to its first and second variable and by~$f_{11}$ the second order partial derivative of~$f$ with respect to its first variable. We derive an analogous formula to~\eqref{eq:Z_n_t_infgen} for the bivariate test function~$f$ and for the process~$\widetilde{Z}_n^{t_n}(s)$ (instead of~$Z_n(t)$). Let us introduce the notation $\tilde{t}_n \coloneqq t_n+\lambda_n s$, i.e., $\widetilde{Z}_n^{t_n}(s) = Z_n(\tilde{t}_n)$ (see~\eqref{eq:tilde_Z_def}). By equations~\eqref{eq:tilde_Z_def} and~\eqref{eq:A_n_t_infgen}, we obtain
\begin{equation}\label{eq:Z_tilde_time_dependent_inf_gen}
\mathcal{L} \mleft(f \mleft( \widetilde{Z}_n^{t_n}(s), s \mright) \mright) =
\lambda_n \cdot H^{(2)}(n,A_n(\tilde{t}_n),\tilde{t}_n,f,s) + f_2 \mleft( \widetilde{Z}_n^{t_n}(s), s \mright),
\end{equation}
where the function $H^{(2)}(n,A,t,f,s)$ is defined analogously to the function $H(n,A,t,f)$ (see~\eqref{eq:H_def}), but for bivariate test function $f\colon\mathbb{R}^2 \to \mathbb{R}$ and with an additional time parameter~$s$:
\begin{equation*}
H^{(2)}(n,A,t,f,s) \coloneqq f_1(z_n^A(t),s) \sqrt{n \lambda_n} \left( \frac{A}{n} - \dot{w}_n(t) \right)+ \lambda_n A \ev_{A,n,t}\mleft[ f\mleft(z_n^A(t)- t \sqrt{\frac{\lambda_n}{n}} \mathscr{C},s\mright)-f\mleft(z_n^A(t),s\mright) \mright]. 
\end{equation*}

Let us consider the function (cf.~\eqref{eq:varphi_n_explicit} and~\eqref{eq:mu_n_def}):
\begin{equation}\label{eq:f_z_s_def}
f^t(z,s) \coloneqq \exp \mleft[ i \cdot z \cdot \left( \sum \limits_{j=k+1}^r \nu_j \cdot \exp \mleft( -\frac{1}{t^2} \cdot (s_j - s) \mright) \right) - \frac 12 \left(\tilde{\nu}^{t}_{k+1} \cdot \tilde{\sigma}^{t}_{k+1}(s) \right)^2 \mright], \qquad s \ge s_k.
\end{equation}
It satisfies the following equation (since~$\varphi^t(s)$ is a martingale, cf.~\eqref{eq:varphi^t(s)_def})
\begin{equation}\label{eq:f_ode}
f^t_2(z,s) - f^t_1(z,s) \cdot \frac{z}{t^2} + \frac 12 \cdot f^t_{11}(z,s) \cdot \frac{1}{t^2} = 0.
\end{equation}
Note also that
\begin{equation*}
\varphi_n^{t_n}(s) = \exp \mleft[ i \cdot \sum \limits_{j=1}^k \nu_j \cdot \widetilde{Z}_n^{t_n}(s_j) - \frac 12 \sum \limits_{j=k+2}^r \left(\tilde{\nu}^{t_n}_j \cdot \tilde{\sigma}^{t_n}_j(s_{j-1}) \right)^2 \mright] \cdot f^{t_n}\mleft(\widetilde{Z}_n^{t_n}(s), s \mright).
\end{equation*}
Therefore, applying the formula~\eqref{eq:Z_tilde_time_dependent_inf_gen} for the function~$f^{t_n}$ (defined in~\eqref{eq:f_z_s_def}), we can derive a formula for~$\mathcal{L}\varphi_n^{t_n}(s)$. 
By the definition of the stopping time $\tau$, we can apply \eqref{eq:exp_drift_precise_statement}. 
This and equations~\eqref{eq:gnt_asymp_t} and~\eqref{eq:f_ode} together imply that the third term on the r.h.s.\ of \eqref{eq:phi_tilde_ost_lhs} goes to zero, and the proof of Lemma~\ref{lem:conditional_char_fn_conv} is complete.
\end{proof}

\begin{lemma}[Joint characteristic function of $\widetilde{Z}^t$ is close to the conditional version]
For any $\varepsilon>0$ and $K>0$, there exists $S^* >0$ such that $-S^*<s_1$ and
\begin{equation}\label{eq:joint_char_close_to_conditional}
 \sup \limits_{|z| \le K} \left| \phi^t(-S^*,z) - \varphi^t \right| \le \frac{\varepsilon}{4}.
\end{equation}
\end{lemma}

\begin{proof}
Recall $\varphi^t(s)$ from~\eqref{eq:varphi^t(s)_def}. Observe that by the martingale property of $\varphi^t(s)$, for any $S^*>0$ such that~${-S^*<s_1}$, we have $\varphi^t = \ev[ \varphi^t(-S^*)]$. Hence the lemma follows from the explicit formula for $\varphi^t(s)$ proved in Lemma~\ref{lem:varphi_explicit} and~\eqref{eq:phi^t_def}.
\end{proof}

\begin{proof}[Proof of Lemma~\ref{lem:Z_joint_char_fn_conv}]
We show that $\lim_{n \to \infty} \varphi_n^{t_n} = \varphi^t$ (this convergence is equivalent to~\eqref{eq:conv_of_characteristic_fn}, see~\eqref{eq:varphi_n_def} and~\eqref{eq:varphi^t_def}). Let~${\varepsilon>0}$ be arbitrary. Then by~\eqref{eq:window_fluctuation_bound}, we can fix $K<+\infty$ such that for every $n \in \mathbb{N}$ and for any~${s <s_1}$, we have
\begin{equation}\label{eq:char_fn_conv_Z_not_large}
\pr \mleft( \, \left|\widetilde{Z}_n^{t_n}(s) \right| \ge K \, \mright) \le \frac{\varepsilon}{4}.
\end{equation}

Now let $S^* < +\infty$ be the constant such that~\eqref{eq:joint_char_close_to_conditional} holds for the fixed $\varepsilon$ and $K$.
Moreover, for any $K$ by~\eqref{eq:conv_of_characteristic_fn_conditional}, there exists $n_0$ such that for any $n \ge n_0$
\begin{equation}\label{eq:char_fn_conv_varphi_n_varphi^t_close}
\sup \limits_{z \in \mathcal{D}^{-S^*}_{n,K}} \left| \phi_n^{t_n}(-S^*,z) - \phi^t(-S^*,z) \right| \le \frac{\varepsilon}{4}.
\end{equation}
Also observe that
\begin{equation}\label{eq:varphi_n_phi_n_relation}
\varphi_n^{t_n} \stackrel{\eqref{eq:varphi_n_def}, \, \eqref{eq:phi_n_def}}{=} \ev \mleft[ \phi_n^{t_n} \mleft( -S^*, \, \widetilde{Z}_n^{t_n}(-S^*) \mright) \mright].
\end{equation}
Therefore,
\begin{align*}
\left| \varphi_n^{t_n} - \varphi^t \right| &\stackrel{\eqref{eq:varphi_n_phi_n_relation}}{=}
\left| \ev \mleft[ \phi_n^{t_n} \mleft( -S^*, \, \widetilde{Z}_n^{t_n}(-S^*) \mright) \mright] - \varphi^t \right| \stackrel{\eqref{eq:char_fn_conv_Z_not_large}}{\le}\\
&\frac{\varepsilon}{4}+\left| \ev \mleft[ \phi_n^{t_n} \mleft( -S^*, \, \widetilde{Z}_n^{t_n}(-S^*) \mright) \cdot \indic{ \left| \widetilde{Z}_n^{t_n}(-S^*)\right| \le K} \mright] - \varphi^t \right| \stackrel{\eqref{eq:char_fn_conv_varphi_n_varphi^t_close}}{\le}\\
&\frac{\varepsilon}{2} + \left| \ev \mleft[ \phi^t \mleft( -S^*, \, \widetilde{Z}_n^{t_n}(-S^*) \mright) \cdot \indic{ \left| \widetilde{Z}_n^{t_n}(-S^*)\right| \le K} \mright] - \varphi^t \right| \stackrel{\eqref{eq:char_fn_conv_Z_not_large}}{\le}\\
&\frac{3\varepsilon}{4} + \left| \ev \mleft[ \phi^t \mleft( -S^*, \, \widetilde{Z}_n^{t_n}(-S^*) \mright) \cdot \indic{ \left| \widetilde{Z}_n^{t_n}(-S^*)\right| \le K} \mright] - \ev \mleft[ \varphi^t \cdot \indic{ \left| \widetilde{Z}_n^{t_n}(-S^*)\right| \le K}\mright] \right| \stackrel{\eqref{eq:joint_char_close_to_conditional}}{\le} \varepsilon.
\end{align*}
The proof of Lemma~\ref{lem:Z_joint_char_fn_conv} is complete.
\end{proof}

\section{Proof of the infinitesimal generator formula}\label{subsec:inf_gen_proof}
In this section we prove Proposition~\ref{prop:exp_drift_precise_ou_lyap_momgen}.

\subsection{Relation between infinitesimal generators}

Recall the notation $z^A_n(t)$, $E(n,A,t,f)$ and $\mathcal{R}(n,\beta,T)$ from~\eqref{eq:z_n_t_A_def}, \eqref{eq:E_n_A_t_f_def} and~\eqref{eq:R_n_def}, respectively.

\begin{notation}
Let $n \in \mathbb{N}$, $\beta \in \mathbb{R}$ and $(A,t,f) \in \mathcal{R}(n,\beta,T)$. We introduce the notation
\begin{equation}\label{eq:Y_n_A_t_f_def}
\mathcal{Y}(n,A,t,f)\coloneqq f'(z_n^A(t)) \frac{\sqrt{n \lambda^5_n}}{1-w_n(t)}+ \frac{n \lambda^2_n }{t} E(n,A,t,f).
\end{equation}
\end{notation}

Recall the process $\widehat{Z}_n^t$ from~\eqref{eq:ou_SDE_n_t_g} and its infinitesimal generator $\widetilde{L}(n, z, t, f)$ from~\eqref{eq:OU_infgen_def}. We show that $\mathcal{Y}(n,A,t,f)$ is close to the infinitesimal generator.

\begin{lemma}[Approximation of the infinitesimal generator]\label{lem:Y_L_close}
For any $\beta \in \mathbb{R}$, we have
\begin{equation}\label{eq:Y_L_close}
\lim_{n \to \infty} \sup_{(A,t,f) \in \mathcal{R}(n,\beta,T) }
 \left| \mathcal{Y}(n,A,t,f) - \widetilde{L}(n,z_n^A(t),t,f) \right| =0.
\end{equation}
\end{lemma}

We prove Lemma~\ref{lem:Y_L_close} in Section~\ref{subsubsec:inf_gen_approx}, let us first deduce Proposition~\ref{prop:exp_drift_precise_ou_lyap_momgen} from it.

\begin{proof}[Proof of Proposition~\ref{prop:exp_drift_precise_ou_lyap_momgen}]
In order to prove Proposition~\ref{prop:exp_drift_precise_ou_lyap_momgen}, we rewrite
\begin{align}\label{eq:H_rewrite}
\begin{split}
\lambda_n H(n,A,t,f)\stackrel{\eqref{eq:w_ODE},\, \eqref{eq:H_def} }{=}& f'(z_n^A(t)) \sqrt{n \lambda^3_n} \left( \frac{A}{n} + \lambda_n \frac{w_n(t)}{1-w_n(t)} -\frac{w_n(t)}{t} \right)+ \lambda^2_n A E(n,A,t,f) \stackrel{\eqref{eq:z_n_t_A_def} }{=} \\
&f'(z_n^A(t))\sqrt{n \lambda^5_n} \frac{w_n(t)}{1-w_n(t)}+\lambda^2_n A E(n,A,t,f)+f'(z_n^A(t))\frac{\lambda_n}{t}z^A_n(t) \stackrel{\eqref{eq:z_n_t_A_def} }{=} \\
&w_n(t)\mathcal{Y}(n,A,t,f)+ f'(z_n^A(t))\frac{\lambda_n}{t}z^A_n(t)+\frac{\sqrt{n \lambda^3_n}}{t} z^A_n(t)E(n,A,t,f).
\end{split}
\end{align}
In order to prove~\eqref{eq:f_test_fn_drift_precise_statement}, we show that the last two terms on the r.h.s.\ of~\eqref{eq:H_rewrite} are insignificant, while $w_n(t)\mathcal{Y}(n,A,t,f)$ is close to $\widetilde{L}(n,z^A_n(t),t,f)$. We begin by showing that~\eqref{eq:Y_L_close} implies
\begin{equation}\label{eq:wY_L_close}
\lim_{n \to \infty} \sup_{(A,t,f) \in \mathcal{R}(n,\beta,T) }
\left| w_n(t) \mathcal{Y}(n,A,t,f) - \widetilde{L}(n,z_n^A(t),t,f) \right| =0.
\end{equation}

First, note that it follows from~\eqref{eq:gnt_bounds}, \eqref{eq:OU_infgen_def}, \eqref{eq:D_z_def} and~\eqref{eq:R_n_def} that there exists some finite positive constant~$C_{T,\beta}$, which only depends on $T$ and $\beta$ such that
\begin{equation}\label{eq:L_absolute_value_bound}
\forall \; n \in \mathbb{N}: \sup_{(A,t,f) \in \mathcal{R}(n,\beta,T) } |\widetilde{L}(n,z_n^A(t),t,f)| \leq C_{T,\beta} \kappa_n.
\end{equation}

Denoting $\mathcal{Y}=\mathcal{Y}(n,A,t,f)$ and $\widetilde{L}=\widetilde{L}(n,z_n^A(t),t,f)$, we have
\begin{equation}\label{eq:Y_L_w_add_subtract}
w_n(t)\mathcal{Y}-\widetilde{L}=w_n(t)(\mathcal{Y}-\widetilde{L})+(w_n(t)-1)\widetilde{L}.
\end{equation}
On the one hand, $0 \leq w_n(t) \leq 1$ by~\eqref{eq:wnvnbound}, so $|w_n(t)(\mathcal{Y}-\widetilde{L})| \to 0$ as $n \to \infty$ by~\eqref{eq:Y_L_close}. On the other hand, $(1-w_n(t))=\mathcal{O}(\lambda_n)$ by~\eqref{eq:gnt_bounds}, so $|(w_n(t)-1)\widetilde{L}| \ll 1$ by~\eqref{eq:L_absolute_value_bound} and our assumption $\kappa_n \lambda_n \ll 1$ (cf.~\eqref{eq:kappa_threshold_assumption}). From these observations and~\eqref{eq:Y_L_w_add_subtract}, we obtain~\eqref{eq:wY_L_close}.

Now we show that the last two terms on the r.h.s.\ of~\eqref{eq:H_rewrite} are insignificant. From~\eqref{eq:Y_L_close} and~\eqref{eq:L_absolute_value_bound}, we obtain
\begin{equation}\label{B_absolute_value_bound}
\forall \; n \in \mathbb{N}: \sup_{(A,t,f) \in \mathcal{R}(n,\beta,T) } |\mathcal{Y}(n,A,t,f)| \leq C'_{T,\beta} \kappa_n.
\end{equation}

Rearranging~\eqref{eq:Y_n_A_t_f_def}, we use~\eqref{B_absolute_value_bound}, \eqref{eq:D_z_def} and $1/{(1-w_n(t))}=\mathcal{O} \mleft(1/{\lambda_n} \mright)$ (cf.~\eqref{eq:gnt_bounds}) to obtain
\begin{equation}\label{eq:E_absolute_value_bound}
\forall \; n \in \mathbb{N}: \sup_{(A,t,f) \in \mathcal{R}(n,\beta,T) } |E(n,A,t,f)| \leq \frac{C''_{T,\beta}}{n \lambda^2_n} \left( \kappa_n+\frac{\sqrt{n \lambda^5_n}}{\lambda_n} \right).
\end{equation}

Now we can show that the third term on the r.h.s.\ of~\eqref{eq:H_rewrite} is insignificant $\forall \; n \in \mathbb{N}:$
\begin{align}\label{eq:third_term_small}
\begin{split}
\sup_{(A,t,f) \in \mathcal{R}(n,\beta,T) } \left|\frac{\sqrt{n \lambda^3_n}}{t} z^A_n(t)E(n,A,t,f)\right| \stackrel{\eqref{eq:R_n_def},\, \eqref{eq:E_absolute_value_bound} }{\leq}&
C'''_{T,\beta} \sqrt{n \lambda^3_n} \kappa_n \frac{1}{n \lambda^2_n} \left( \kappa_n + \frac{\sqrt{n \lambda^5_n}}{\lambda_n} \right) =\\
& C'''_{T,\beta} \left( \frac{\kappa_n^2}{\sqrt{n \lambda_n}}+ \kappa_n \lambda_n \right) \stackrel{(*)}{\ll} 1,
\end{split}
\end{align}
where $(*)$ holds, since $\kappa_n \lambda_n \ll 1$ holds by~\eqref{eq:kappa_threshold_assumption}, moreover,
$\kappa_n \ll \lambda_n^{-1}$ and $\kappa_n \ll \sqrt{n \lambda^3_n}$ both hold by~\eqref{eq:kappa_threshold_assumption}, and multiplying these two bounds we obtain $\kappa_n^2 \ll \sqrt{n \lambda_n}$.

Finally, the second term on the r.h.s.\ of~\eqref{eq:H_rewrite} is also insignificant:
\begin{equation}\label{eq:second_term_small}
\forall \; n \in \mathbb{N}: \sup_{(A,t,f) \in \mathcal{R}(n,\beta,T) } \left|f'(z_n^A(t))\frac{\lambda_n}{t}z^A_n(t) \right| \stackrel{\eqref{eq:D_z_def},\, \eqref{eq:R_n_def} }{\leq} |\beta| \lambda_n \kappa_n \stackrel{\eqref{eq:kappa_threshold_assumption}}{\ll} 1.
\end{equation}
Putting together \eqref{eq:H_rewrite} with the bounds \eqref{eq:wY_L_close}, \eqref{eq:third_term_small} and~\eqref{eq:second_term_small}, we obtain~\eqref{eq:f_test_fn_drift_precise_statement}. This concludes the proof of Proposition~\ref{prop:exp_drift_precise_ou_lyap_momgen}, given Lemma~\ref{lem:Y_L_close}.
\end{proof}

\subsection{Approximation of the infinitesimal generator}\label{subsubsec:inf_gen_approx}

In this section we prove Lemma~\ref{lem:Y_L_close}.

\begin{notation}
Given some $f \in \mathcal{D}(z, \beta)$ (cf.~\eqref{eq:D_z_def}), let us denote the remainder term of the second order Taylor expansion of $f$ about $z$ by
\begin{equation}\label{eq:error_term_taylor}
R_2(f,z,x) \coloneqq f(z-x)-f(z)+f'(z)x- \ffrac 12 f''(z)x^2.
\end{equation}
\end{notation}

\begin{claim}[Upper bound on the difference]
We have the following upper bound for the difference that occurred in \eqref{eq:Y_L_close}:
\begin{align}\label{eq:triangle_separating_derivatives}
\begin{split}
\left| \mathcal{Y}(n,A,t,f) - \widetilde{L}(n,z_n^A(t),t,f) \right|\le&
|f'(z_n^A(t))| \cdot \left| \sqrt{n \lambda^5_n} \left( \frac{1}{1-w_n(t)}- \ev_{A,n,t}\mleft[ \mathscr{C} \mright] \right) + \frac{ z_n^A(t) }{g^2_n(t)} \right| +\\
& \frac{|f''(z_n^A(t)) | }{2} t \cdot \left| \lambda_n^3 \ev_{A,n,t}\mleft[ \mathscr{C}^2 \mright] -
\frac{1 }{g^3_n(t)} \right| +\\
&\frac{n \lambda^2_n}{t} \left| \ev_{A,n,t}\mleft[ R_2 \left(f,z_n^A(t),t \sqrt{\frac{\lambda_n}{n}}\mathscr{C} \right) \mright] \right|.
\end{split}
\end{align}
\end{claim}

\begin{proof}
Observe that we can rewrite the l.h.s.\ of~\eqref{eq:triangle_separating_derivatives}. Denoting $\mathcal{Y}=\mathcal{Y}(n,A,t,f)$ and $\widetilde{L}=\widetilde{L}(n,z_n^A(t),t,f)$:
\begin{equation}\label{eq:Y_L_difference}
\left| \mathcal{Y} - \widetilde{L} \right| \stackrel{\eqref{eq:OU_infgen_def},\, \eqref{eq:Y_n_A_t_f_def} }{=} \left| f'(z_n^A(t)) \frac{\sqrt{n \lambda^5_n}}{1-w_n(t)}+ \frac{n \lambda^2_n }{t} E(n,A,t,f) + \frac{ z_n^A(t) f'(z_n^A(t)) }{g^2_n(t)} - \frac{t f''(z_n^A(t)) }{2 g^3_n(t)} \right|.
\end{equation}
Using the notation $R_2(f,z,x)$ from~\eqref{eq:error_term_taylor}, we can write
\begin{align}\label{eq:taylor_E_n_A_t_f}
\begin{split}
E(n,A,t,f) \stackrel{\eqref{eq:E_n_A_t_f_def}}{=}
&-f'(z_n^A(t))t \sqrt{ \frac{\lambda_n}{n} }
\ev_{A,n,t}\mleft[ \mathscr{C} \mright] +\frac{1}{2} f''(z_n^A(t))t^2 \frac{\lambda_n}{n}
\ev_{A,n,t}\mleft[ \mathscr{C}^2 \mright] +\\
&\ev_{A,n,t}\mleft[ R_2 \mleft(f,z_n^A(t),t \sqrt{\frac{\lambda_n}{n}}\mathscr{C} \mright) \mright].
\end{split}
\end{align}
Therefore, \eqref{eq:triangle_separating_derivatives} follows from~\eqref{eq:Y_L_difference} and~\eqref{eq:taylor_E_n_A_t_f}.
\end{proof}

In order to prove Lemma~\ref{lem:Y_L_close}, we need to show that the three terms on the r.h.s.\ of
\eqref{eq:triangle_separating_derivatives} are small:

\begin{lemma}[Terms of the upper bound are insignificant]\label{lem:control_of_f_expectations}
For any $\beta \in \mathbb{R}$, we have
\begin{align}
\label{eq:OU_drift_comes_up}
& \lim_{n \to \infty} \, \sup_{t_n^c \leq t \leq T} \, \sup_{ A: \, |z_n^A(t)| \leq \kappa_n } \left| \sqrt{n \lambda^5_n}
\left( \frac{1}{1-w_n(t)}- \ev_{A,n,t}\mleft[ \mathscr{C} \mright] \right)
+ \frac{ z_n^A(t) }{g^2_n(t)} \right|=0, \\
\label{eq:OU_volatility_comes_up}
& \lim_{n \to \infty} \, \sup_{t_n^c \leq t \leq T} \, \sup_{ A: \, |z_n^A(t)| \leq \kappa_n }
\left| \lambda_n^3 \ev_{A,n,t}\mleft[ \mathscr{C}^2 \mright] - \frac{1 }{g^3_n(t)} \right|=0, \\
\label{eq:control_of_f_expect_taylor_remainder}
 & \lim_{n \to \infty} \sup_{(A,t,f) \in \mathcal{R}(n,\beta,T) }
n \lambda^2_n\ev_{A,n,t}\mleft[ \, \left| R_2 \mleft(f,z_n^A(t),t \sqrt{\frac{\lambda_n}{n}}\mathscr{C} \mright) \right| \, \mright]=0.
\end{align}
\end{lemma}
Before we prove Lemma~\ref{lem:control_of_f_expectations}, let us deduce Lemma~\ref{lem:Y_L_close} from it.
\begin{proof}[Proof of Lemma~\ref{lem:Y_L_close}] The inequalities
$|f'(z_n^A(t))| \leq |\beta|$, $|f''(z_n^A(t)) | \leq |\beta|^2$ and $1 \leq t \leq T$ follow from the
assumption $(A,t,f) \in \mathcal{R}(n,\beta,T)$ (cf.~\eqref{eq:R_n_def}), thus the desired \eqref{eq:Y_L_close}
follows from~\eqref{eq:triangle_separating_derivatives} together with the statements of Lemma~\ref{lem:control_of_f_expectations}.
\end{proof}

\subsection{Terms of the upper bound are insignificant}

In this section we prove Lemma~\ref{lem:control_of_f_expectations}. Recalling \eqref{eq:W_t_A_n_def}, let us denote $ W_n^A(t)\coloneqq tA/n$.

\begin{claim}[Bound on $W_n^A$]
We have
\begin{equation}\label{eq:w_geq_1}
\liminf_{n \to \infty} \inf_{t_n^c \leq t \leq T} \inf_{ A: \, |z_n^A(t)| \leq \kappa_n } \frac{1-W_n^A(t)}{\lambda_n} \geq 1.
\end{equation}
\end{claim}

\begin{proof}
Our assumption $|z_n^A(t)| \leq \kappa_n $ together with~\eqref{eq:z_n_t_A_def} imply $|W_n^A(t)-w_n(t)|/\lambda_n \leq \kappa_n /\sqrt{n \lambda^3_n}$. Next, we observe that
\begin{equation*}
\frac{\kappa_n}{\sqrt{n \lambda^3_n}} \ll \frac{\kappa_n^2}{\sqrt{n \lambda^3_n}} = \left(\frac{\kappa_n^{4/3}}{n^{1/3} \lambda_n}\right)^{3/2} \stackrel{\eqref{eq:kappa_threshold_assumption}}{\ll} 1.
\end{equation*} 
This, together with the lower bound \eqref{eq:gnt_bounds} on $(1-w_n(t))/\lambda_n$, implies~\eqref{eq:w_geq_1}.
\end{proof}

\begin{claim}[Asymptotic relations for $W^A_n$]
Let $(t_n)_{n=1}^\infty$ and $(A_n)_{n=1}^\infty$ be sequences such that $t_n^c \le t_n \leq T$ and $|z_n^{A_n}(t_n)| \leq \kappa_n$. Then the following asymptotic relations hold as $n \to \infty$:
\begin{equation}\label{eq:tA_n_asymp}
\mathrm{(a)}\, \left| \frac{t_n A_n }{n} -(1-\lambda_n g_n(t_n) ) \right| = o(\lambda_n), \quad 
\mathrm{(b)}\, \left|\frac{t_n A_n}{n} - 1 \right| \to 0, \quad \mathrm{(c)}\, A_n \asymp n.
\end{equation}

Moreover, if we assume that $\frac{t_n - t_n^c}{\lambda_n} \to \infty$ as $n \to \infty$, then
\begin{equation}\label{eq:tA_n_asymp_stronger}
 \left| \frac{t_n A_n }{n} -(1-\lambda_n t_n ) \right| = o(\lambda_n)
\end{equation}
\end{claim}

\begin{proof} We first prove statement (a):
\begin{equation*}
\left| \frac{t_n A_n }{n} -(1-\lambda_n g_n(t_n) ) \right|
\stackrel{\eqref{eq:def_g_n_t},\,\eqref{eq:z_n_t_A_def}}{=} 
\left| \frac{ z_n^{A_n}(t_n) }{ \sqrt{n \lambda_n} } \right| 
\stackrel{ |z_n^{A_n}(t_n)| \leq \kappa_n}{\leq} \frac{\kappa_n}{\sqrt{n \lambda_n}} \stackrel{(*)}{=} o(\lambda_n),
\end{equation*}
where $(*)$ holds since $\kappa_n/\sqrt{n \lambda_n} \ll \lambda_n$ is equivalent to $\kappa_n^{2/3} n^{-1/3} \ll \lambda_n$, which follows from~\eqref{eq:kappa_threshold_assumption}.

Statement (b) now follows from statement (a), $\lambda_n \ll 1$ (cf.~\eqref{eq:subcrit_lambda_assumption}) and~\eqref{eq:gnt_bounds}. Statement (c) follows from statement (b) and the assumptions on $t_n$.

In order to prove equation~\eqref{eq:tA_n_asymp_stronger}, by the assumption $(t_n - t_n^c)/\lambda_n \to \infty$ as $n \to \infty$, we can use \eqref{eq:gnt_asymp_t}. Hence $\lambda_n(t_n-g_n(t_n)) = o(\lambda_n)$, therefore, \eqref{eq:tA_n_asymp}(a) implies~\eqref{eq:tA_n_asymp_stronger}.
\end{proof}

\begin{corollary}[Conditions for Erd\H{o}s--R\'enyi component size formulas]\label{cor:ER_conditions}
Let $(t_n)_{n=1}^\infty$ and $(A_n)_{n=1}^\infty$ be sequences such that $t_n^c \le t_n \leq T$ and $|z_n^{A_n}(t_n)| \leq \kappa_n$. Then we have $1/3 \le W_n^{A_n}(t_n) \le 1-A_n^{-1/3}$ if $n$ is large enough.
\end{corollary}

\begin{proof}
If $n$ is large enough, then $1/3 \le W_n^{A_n}(t_n)$ follows from~\eqref{eq:tA_n_asymp}(b) and $W_n^{A_n}(t_n) \le 1-A_n^{-1/3}$ follows from~\eqref{eq:tA_n_asymp}(a):
\begin{equation*}
W_n^{A_n}(t_n) \stackrel{\eqref{eq:tA_n_asymp}\mathrm{(a)}}{\le} 1-\lambda_n g_n(t_n) + o(\lambda_n) \stackrel{\eqref{eq:subcrit_lambda_assumption},\, \eqref{eq:tA_n_asymp}\mathrm{(c)}}{\le} 1-A_n^{-1/3}.\qedhere
\end{equation*}
\end{proof}

The motivation of Corollary~\ref{cor:ER_conditions} is to ensure that we can use the formulas of Claim~\ref{cl:component_size} under the assumptions $t_n^c \le t_n \leq T$ and $|z_n^{A_n}(t)| \leq \kappa_n$.

Now we prove the three limits stated in Lemma~\ref{lem:control_of_f_expectations}.

\begin{proof}[Proof of~\eqref{eq:OU_drift_comes_up}] We will use
\eqref{eq:first_moment_of_component_size} to estimate $\ev_{A,n,t}\mleft[ \mathscr{C} \mright]$. We begin with
\begin{align}\label{eq:w_W_compare}
\begin{split}
\sqrt{n \lambda^5_n} \left( \frac{1}{1-W^A_n(t)}- \frac{1}{1-w_n(t)} \right) \stackrel{\eqref{eq:z_n_t_A_def},\, \eqref{eq:W_t_A_n_def}}{=} &\lambda_n^2 \frac{ z^A_n(t) }{(1-W^A_n(t))(1-w_n(t))} \stackrel{\eqref{eq:def_g_n_t}}{=}\\
&\frac{z^A_n(t)}{g_n^2(t)} \frac{1-w_n(t)}{1-W^A_n(t)} \stackrel{\eqref{eq:z_n_t_A_def},\, \eqref{eq:W_t_A_n_def}}{=} \frac{z^A_n(t)}{g_n^2(t)} + \frac{ (z^A_n(t))^2 /\sqrt{n \lambda_n } }{g_n^2(t)(1-W^A_n(t)) }.
\end{split}
\end{align}

We can thus bound
\begin{align}\label{eq:W_w_diff_goes_to_zero}
\begin{split}
&\lim_{n \to \infty} \, \sup_{t_n^c \leq t \leq T} \, \sup_{ A: \, |z_n^A(t)| \leq \kappa_n }
\left| \sqrt{n \lambda^5_n} \left( \frac{1}{1-w_n(t)}- \frac{1}{1-W^A_n(t)} \right) + \frac{ z_n^A(t) }{g^2_n(t)} \right| \stackrel{\eqref{eq:R_n_def},\, \eqref{eq:w_W_compare}}{\leq} \\
&\lim_{n \to \infty} \, \sup_{t_n^c \leq t \leq T} \, \sup_{ A: \, |z_n^A(t)| \leq \kappa_n } \left| \frac{ \kappa_n^2 /\sqrt{n \lambda_n } }{g_n^2(t)(1-W^A_n(t)) } \right| \stackrel{\eqref{eq:gnt_bounds},\, \eqref{eq:w_geq_1}}{\leq} \lim_{n \to \infty}
\frac{ \kappa_n^2 /\sqrt{n \lambda_n } }{ \lambda_n } \stackrel{\eqref{eq:kappa_threshold_assumption}}{=}0.
\end{split}
\end{align}

We can also bound
\begin{align}\label{eq:W_EC_diff_goes_to_zero}
\begin{split}
&\lim_{n \to \infty} \, \sup_{t_n^c \leq t \leq T} \, \sup_{ A: \, |z_n^A(t)| \leq \kappa_n } \left| \sqrt{n \lambda^5_n} \left( \frac{1}{1-W^A_n(t)} - \ev_{A,n,t}\mleft[ \mathscr{C} \mright] \right) \right| \stackrel{\eqref{eq:first_moment_of_component_size}}{\leq } \\
&\lim_{n \to \infty} \, \sup_{t_n^c \leq t \leq T} \, \sup_{ A: \, |z_n^A(t)| \leq \kappa_n } \sqrt{n \lambda^5_n} \frac{\widetilde{C}}{(1-W_n^A(t))^4 A } \stackrel{\eqref{eq:w_geq_1}, \, \eqref{eq:tA_n_asymp}\mathrm{(c)}}{\leq } \lim_{n \to \infty} \frac{C'}{\sqrt{n \lambda_n^3 }} \stackrel{\eqref{eq:subcrit_lambda_assumption} }{=}0.
\end{split}
\end{align}
Putting together \eqref{eq:W_w_diff_goes_to_zero} and~\eqref{eq:W_EC_diff_goes_to_zero}, we obtain~\eqref{eq:OU_drift_comes_up}.
\end{proof}

\begin{proof}[Proof of~\eqref{eq:OU_volatility_comes_up}] We will use
\eqref{eq:second_moment_of_component_size} to estimate $\ev_{A,n,t}\mleft[ \mathscr{C}^2 \mright]$.
 We begin with
\begin{align}
\begin{split}\label{eq:W_w_second_moment_compare}
&\lim_{n \to \infty} \, \sup_{t_n^c \leq t \leq T} \, \sup_{ A: \, |z_n^A(t)| \leq \kappa_n } \left| \lambda_n^3 \frac{1}{(1-W^A_n(t))^3} - \frac{1 }{g^3_n(t)}\right|
\stackrel{\eqref{eq:def_g_n_t},\, \eqref{eq:z_n_t_A_def},\, \eqref{eq:W_t_A_n_def} }{=}\\
&\lim_{n \to \infty} \, \sup_{t_n^c \leq t \leq T} \, \sup_{ A: \, |z_n^A(t)| \leq \kappa_n } \left| \frac{1 }{g^3_n(t)} \left[ \left( 1+\frac{z_n^A(t)/\sqrt{n \lambda_n}}{1-W^A_n(t)} \right)^3-1 \right]\right|\stackrel{(*)}{=}0,
\end{split}
\end{align}
where $(*)$ follows from $\lim_{n \to \infty} \, \sup_{t_n^c \leq t \leq T} \, \sup_{ A: \, |z_n^A(t)| \leq \kappa_n }\left| \frac{z_n^A(t)/\sqrt{n \lambda_n}}{1-W^A_n(t)}\right|=0 $ (which can be proved
similarly to~\eqref{eq:W_w_diff_goes_to_zero}) and~\eqref{eq:gnt_bounds}. We can also bound
\begin{align}
\begin{split}\label{eq:W_C_second_moment_close}
&\lim_{n \to \infty} \, \sup_{t_n^c \leq t \leq T} \, \sup_{ A: \, |z_n^A(t)| \leq \kappa_n } \left| \lambda_n^3 \ev_{A,n,t}\mleft[ \mathscr{C}^2 \mright]- \lambda_n^3 \frac{1}{(1-W^A_n(t))^3}\right| 
\stackrel{\eqref{eq:second_moment_of_component_size} }{\leq } \\
&\lim_{n \to \infty} \, \sup_{t_n^c \leq t \leq T} \, \sup_{ A: \, |z_n^A(t)| \leq \kappa_n } \lambda_n^3 \frac{\widetilde{C}}{(1-W_n^A(t))^6 A}
\stackrel{\eqref{eq:w_geq_1},\, \eqref{eq:tA_n_asymp}\mathrm{(c)}}{\leq } \lim_{n \to \infty} \frac{C'}{n \lambda_n^3 }
\stackrel{\eqref{eq:subcrit_lambda_assumption} }{=}0.
\end{split}
\end{align}
Putting together \eqref{eq:W_w_second_moment_compare} and~\eqref{eq:W_C_second_moment_close}, we obtain~\eqref{eq:OU_volatility_comes_up}.
\end{proof}

\begin{proof}[Proof of~\eqref{eq:control_of_f_expect_taylor_remainder}]
For any $\beta,\, z,\, x \in \mathbb{R}$ and any $f \in \mathcal{D}(z, \beta)$, we have
\begin{equation*}
|R_2(f,z,x)| \stackrel{(*)}{\leq} \frac{|\beta|^3}{6} |x|^3 \mathrm{e}^{|\beta x|} \leq
\frac{|\beta|^3}{6} |x|^3 \mathrm{e} + \mathrm{e}^{2 |\beta x|} \indic{ |\beta x| \geq 1 },
\end{equation*}
where $(*)$ follows from Taylor's theorem together with~\eqref{eq:D_z_def} and~\eqref{eq:error_term_taylor}.
Thus in order to prove~\eqref{eq:control_of_f_expect_taylor_remainder}, we just have to prove the following two statements:
\begin{align}
\label{eq:third_moment_goes_to_zero}
&\lim_{n \to \infty} \, \sup_{t_n^c \leq t \leq T} \, \sup_{ A: \, |z_n^A(t)| \leq \kappa_n }
n \lambda^2_n \ev_{A,n,t}\mleft[ \mleft( \sqrt{\frac{\lambda_n}{n}}\mathscr{C} \mright)^3 \, \mright] =0, \\
\label{eq:truncated_exponential_goes_to_zero}
&\lim_{n \to \infty} \, \sup_{t_n^c \leq t \leq T} \, \sup_{ A: \, |z_n^A(t)| \leq \kappa_n }
n \lambda^2_n \ev_{A,n,t}\mleft[ \exp \mleft(2 |\beta | t \sqrt{\frac{\lambda_n}{n}}\mathscr{C} \mright) \cdot \indic{ |\beta| t \sqrt{\frac{\lambda_n}{n}}\mathscr{C} \geq 1 } \mright] =0.
\end{align}

We first show \eqref{eq:third_moment_goes_to_zero}:
\begin{equation*}
\lim_{n \to \infty} \, \sup_{t_n^c \leq t \leq T} \, \sup_{ A: \, |z_n^A(t)| \leq \kappa_n } \sqrt{\frac{\lambda_n^7}{n}}
\ev_{A,n,t}\mleft[ \mathscr{C}^3 \mright]
\stackrel{\eqref{eq:general_moment_of_component_size_crude},\,\eqref{eq:w_geq_1} }{\leq} \widetilde{C}_3 \lim_{n \to \infty} \sqrt{\frac{\lambda_n^7}{n}}
\lambda_n^{-5}\stackrel{\eqref{eq:subcrit_lambda_assumption}}{=}0.
\end{equation*}
It remains to show \eqref{eq:truncated_exponential_goes_to_zero}. We use \eqref{eq:magic_branching_proc_estimate_2}. It follows from~\eqref{eq:subcrit_lambda_assumption} that for large enough $n$ and $t_n^c \le t \le T$, we have
\begin{equation}\label{eq:trunc_exp_bound_aux}
2|\beta|t \sqrt{\frac{\lambda_n}{n}} < \frac{1}{4} \lambda_n^2
\end{equation}
Therefore, for large enough $n$ and $t_n^c \le t \le T$, we have
\begin{align}
\begin{split}\label{eq:trunc_exp_bound}
\frac{\exp \mleft(2|\beta|t \sqrt{\frac{\lambda_n}{n}} \mathscr{C}\mright) \cdot \indic{ |\beta|t \sqrt{\frac{\lambda_n}{n}} \mathscr{C} \geq 1 }}{\exp \mleft( \frac{1}{4}\lambda_n^2 \mathscr{C} \mright)} =&
\left( \frac{\mathrm{e}^2 \cdot \indic{ |\beta|t \sqrt{\frac{\lambda_n}{n}} \mathscr{C} \geq 1 }}{\exp \mleft[ \frac{1}{4}\lambda_n^2 \left( |\beta|t \sqrt{\frac{\lambda_n}{n}} \right)^{-1}\mright]} \right)^{|\beta|t \sqrt{\frac{\lambda_n}{n}}\mathscr{C}} \stackrel{(*)}{\leq}\\
&\exp \mleft[2 - \frac{1}{4}\lambda_n^2 \left( |\beta|t \sqrt{\frac{\lambda_n}{n}} \right)^{-1} \mright]
\leq
\exp \mleft(- C_{T,\beta} \sqrt{ \lambda_n^3 n } \mright)
\end{split}
\end{align}
for some positive constant $C_{T,\beta}$ that only depends on $\beta$ and $T$, where at $(*)$ we used that by~\eqref{eq:trunc_exp_bound_aux} the fraction on the l.h.s.\ is smaller than $1$. Therefore, for large enough $n$, any $t_n^c \leq t \leq T$ and any $A$ such that $|z_n^A(t)| \leq \kappa_n$, we have $\exp \mleft( \frac{1}{4}\lambda_n^2 \mathscr{C} \mright) \leq \exp \mleft( \frac{1}{4} (1-W_n^A(t))^2 \mathscr{C}\mright) $ by~\eqref{eq:w_geq_1}, thus
\begin{equation*}
\sup_{t_n^c \leq t \leq T} \, \sup_{ A: \, |z_n^A(t)| \leq \kappa_n } \ev_{A,n,t}\mleft[ \exp \mleft(2 |\beta | t \sqrt{\frac{\lambda_n}{n}}\mathscr{C} \mright) \cdot \indic{ |\beta| t \sqrt{\frac{\lambda_n}{n}}\mathscr{C} \geq 1 } \mright]
 \stackrel{\eqref{eq:magic_branching_proc_estimate_2},\, \eqref{eq:trunc_exp_bound} }{\leq} 3 \exp \mleft( - C_{T,\beta} \sqrt{ \lambda_n^3 n } \mright).
\end{equation*}

It, together with $\ln n \ll \sqrt{ \lambda_n^3 n }$ (which follows from~\eqref{eq:subcrit_lambda_assumption}), gives~\eqref{eq:truncated_exponential_goes_to_zero}.
\end{proof}

The proof of Lemma~\ref{lem:control_of_f_expectations} is complete.

\begin{remark}[Stronger assumptions on~$\lambda_n$]
We note that most of our proofs works for $n^{-1/3} \ll \lambda_n \ll 1$. We use our stronger assumptions~\eqref{eq:subcrit_lambda_assumption} only to have a sequence $(\kappa_n)_{n=1}^{\infty}$ such that $\ln( {1}/{\lambda_n}) \ll \kappa_n$ (see~\eqref{eq:kappa_threshold_assumption}) and at the last step of the proof of the formula~\eqref{eq:control_of_f_expect_taylor_remainder}. In the proof of~\eqref{eq:control_of_f_expect_taylor_remainder}, we only need $(\ln n)^{2/3} n^{-1/3} \ll \lambda_n$. To the existence of $\kappa_n$ that satisfies~\eqref{eq:kappa_threshold_assumption}, it is enough to assume that $\ln ( 1/\lambda_n) \ll \left(n \lambda_n^3 \right)^{1/4}$. For that purpose $(\ln n)^{4/3} n^{-1/3} \ll \lambda_n$ is sufficient, while $(\ln n)^{2/3} n^{-1/3} \ll \lambda_n$ is not (see also Remark~\ref{rem:kappa_construction}).
\end{remark}

\section{Proof of the L\'evy process scaling limit result}\label{sec:levy_proof}

In this section our goal is to prove Theorem~\ref{thm:tempered_stable_scaling}.

Similarly to Theorem~\ref{thm:OU_scaling}, we prove Theorem~\ref{thm:tempered_stable_scaling} by \cite[Theorem 15.1]{B68}. But instead of $U_n^{t_n}$ and $U^t$, we apply this theorem to the following auxiliary processes.

\begin{definition}[Auxiliary increasing processes]
Let us define the following processes:
\begin{align}
Y_n^t(s) &\coloneqq -\left[ A_n \mleft(t+\frac{s}{n\lambda_n^2} \mright) - A_n(t) \right] \cdot \lambda_n^2,\label{eq:defnY_n(s)}\\
Y^t(s) &\coloneqq \frac{1}{t^2} \cdot \int \limits_0^s \int \limits_0^\infty y N(\mathrm{d} y,\mathrm{d} u),\label{eq:defnY(s)}
\end{align}
where $N(\cdot,\cdot)$ is a Poisson point process on $\mathbb{R}^2_+$ with intensity $\frac{1}{\sqrt{2\pi}}y^{-3/2}\mathrm{e}^{-y/2}\, \mathrm{d} y \, \mathrm{d} u$ (cf.~\eqref{eq:defnU(s)}).
\end{definition}

Observe that $Y^t(s)$ is an increasing L\'evy process, therefore, it is a subordinator process.

\begin{proposition}[Convergence to subordinator process]\label{prop:Yn_convergence}
Let $t$ and $t_n$ as in Theorem~\ref{thm:tempered_stable_scaling}, and consider the processes $Y_n^{t_n}$ and $Y^t$ defined in \eqref{eq:defnY_n(s)} and~\eqref{eq:defnY(s)}, respectively. Then $Y_n^{t_n}(\cdot) \Rightarrow Y^t(\cdot)$ in $D[0,S]$ for each~${S<\infty}$.
\end{proposition}

We prove Proposition~\ref{prop:Yn_convergence} in Section~\ref{subsec:subordinator}. Before that we show that it indeed implies Theorem~\ref{thm:tempered_stable_scaling}.

\begin{proof}[Proof of Theorem~\ref{thm:tempered_stable_scaling}]

Observe that
\begin{equation}\label{eq:U_Y_connection}
 U_n^{t_n}(s) \stackrel{\eqref{eq:defnU_n(s)},\, \eqref{eq:Wn_def},\, \eqref{eq:defnY_n(s)}}{=} -t_n \cdot Y_n^{t_n}(s) + \frac{A_n\mleft(t_n+\frac{s}{n \lambda_n^2} \mright)}{n}\cdot s, \qquad U^t(s) \stackrel{\eqref{eq:defnU(s)},\, \eqref{eq:defnY(s)}}{=} \frac st - t \cdot Y^t(s).
\end{equation}

Note also that $t_n \to t$ and $\frac 1n \cdot A_n\mleft(t_n+\frac{s}{n \lambda_n^2}\mright) \to \frac 1t$ uniformly for $s \in [0,S]$ (see Theorem~\ref{thm:window_critical}) as $n \to \infty$. Therefore, Proposition~\ref{prop:Yn_convergence} and~\eqref{eq:U_Y_connection} imply Theorem~\ref{thm:tempered_stable_scaling}.
\end{proof}

\subsection{Convergence to the subordinator L\'evy process}\label{subsec:subordinator}

In this section we prove Proposition~\ref{prop:Yn_convergence}.

To prove Proposition~\ref{prop:Yn_convergence}, we use Theorem~\ref{thm:convergence_in_D}, i.e., we show that $Y_n^{t_n}$ is tight on $D[0,S]$ and that the finite dimensional distributions of $Y_n^{t_n}$ converge weakly to the finite dimensional distributions of $Y^t$.

First, we prove the tightness of $Y_n^{t_n}$. For that purpose, we use Aldous' criterion, i.e., we show that the sequence $Y_n^{t_n}$ satisfies the conditions of Theorem~\ref{thm:aldous_crit}. Before that we prove an upper bound on the expected size of the decrement of alive vertices in a given interval. Recall the filtration $(\mathcal{F}_n(\cdot))$ from Definition~\ref{def:alive_vert}, the notion of the stopping time $ \tau_Z = \min \mleft\{\, t \ge t_n^c \; : \; |Z_n(t)| \geq \kappa_n \, \mright\} $ from~\eqref{eq:tauZ_def} and the notation~$\widetilde{C}_1$ from~\eqref{eq:general_moment_of_component_size_crude}.

\begin{lemma}[Bound on the decrement of alive vertices]\label{lem:bnd_decrement} There exists $C \in \mathbb{R}_+$ such that for any $n \in \mathbb{N}$ and for any stopping time $\tau_n$ w.r.t.\ $(\mathcal{F}_n(\cdot))$ that takes its values in $[t_n^c,T]$, for any $D \in \mathbb{R_+}$ such that $\tau_n+D \le T$ almost surely, and for any $K > 0$, we have
\begin{equation}\label{eq:alive_difference_bound}
\pr \big(\, |A_n(\tau_n+D ) - A_n(\tau_n)| \geq K\, \big) \le \frac{D \cdot n \cdot C}{K} + \pr(\tau_Z < T).
\end{equation}
\end{lemma}

\begin{proof} Recall the notion of 
 $\ev_{A,n,t}$ and $\mathscr{C}$ from Definition~\ref{not:ER_notation_A_n_t} and that the process $A_n(t) - \int_0^t \mathcal{L}A_n(s) \, \mathrm{d}s$ is a martingale (cf.~\eqref{eq:alive_corrector_martingale}). We have
\begin{multline*}
-\ev \mleft[ A_n((\tau_n+D) \wedge \tau_Z) - A_n( \tau_n \wedge \tau_Z ) \mright] =
- \ev \mleft[ \, \int \limits_{\tau_n \wedge \tau_Z }^{(\tau_n+D) \wedge \tau_Z} \mathcal{L}A_n(s) \, \mathrm{d}s \mright] \stackrel{\eqref{eq:neighbor_drift}}{=} \\
\ev \mleft[ \, \int \limits_{\tau_n \wedge \tau_Z}^{(\tau_n+D) \wedge \tau_Z } \lambda_n A_n(s) \cdot \ev_{A_n(s),n,s}[\mathscr{C}] \, \mathrm{d}s \mright] \stackrel{\eqref{eq:general_moment_of_component_size_crude}, \, \eqref{eq:w_geq_1}}{\le}
D \cdot \lambda_n \cdot n \cdot \frac{2\widetilde{C}_1}{\lambda_n} = 2 D \cdot n \cdot \widetilde{C}_1.
\end{multline*}
An application of Markov's inequality concludes the proof of~\eqref{eq:alive_difference_bound}.
\end{proof}

\begin{lemma}[Tightness of $Y_n^{t_n}$]\label{lem:tightness_Yn}
The sequence $Y_n^{t_n}$ is tight on $D[0,S]$.
\end{lemma}

\begin{proof} 
We show that~\eqref{eq:alive_difference_bound} implies that the conditions of Theorem~\ref{thm:aldous_crit} hold for the sequence $Y_n^{t_n}$. Let~$\tau_n$ and~$\delta_n$ as in Theorem~\ref{thm:aldous_crit}. Then for any $\varepsilon>0$ we obtain~\eqref{eq:aldous_crit}:
\begin{align*}
\begin{split}
\pr \mleft( \,|Y_n^{t_n}(\tau_n+\delta_n) - Y_n^{t_n}(\tau_n)| \geq \varepsilon \,\mright) \stackrel{\eqref{eq:defnY_n(s)}}{=}&
\pr \mleft( \, \left| A_n \mleft(t_n + \frac{\tau_n+\delta_n}{n \lambda_n^2} \mright) -A_n \mleft(t_n + \frac{\tau_n}{n \lambda_n^2} \mright) \right| \geq \frac{\varepsilon}{\lambda^2_n} \,\mright) \stackrel{\eqref{eq:alive_difference_bound}}{\le}\\
&\frac{\lambda_n^2}{\varepsilon} \cdot \frac{\delta_n}{n \lambda_n^2} \cdot n \cdot C + \pr(\tau_Z <T) \stackrel{n \to \infty}{\longrightarrow} 0,
\end{split}
\end{align*}
where we used $\delta_n \to 0$ and Lemma~\ref{lem:opt_stop} (i.e., that we have $\tau_Z \geq T$ w.h.p.) to conclude.

Similarly, for any $K > 0$, we have
\begin{equation*}
\pr \mleft( |Y_n^{t_n}(s)| \geq K \mright) \stackrel{\eqref{eq:defnY_n(s)}}{=} \mathbb{P} \mleft( \,\left| A_n \mleft(t_n + \frac{s}{n \lambda_n^2} \mright) - A_n(t_n) \right| \geq \frac{K}{\lambda_n^2} \,\mright) \stackrel{\eqref{eq:alive_difference_bound}}{\le} \frac{\lambda_n^2}{K} \cdot \frac{s}{n \lambda_n^2} \cdot n \cdot C + \mathbb{P}(\tau_Z <T).
\end{equation*}
Therefore, for any $\varepsilon>0$, we can choose $K$ such that ${\pr \mleft( |Y_n^{t_n}(s)| \geq K \mright) \le \varepsilon}$ for all $n$, i.e.,~$Y_n^{t_n}(s)$ is tight for any $s \in [0,S]$. By Theorem~\ref{thm:aldous_crit}, the tightness of the sequence $(Y_n^{t_n}(\cdot))$ on~$D[0,S]$ follows.
\end{proof}

Our next goal is to prove the convergence of finite dimensional distributions of $Y_n^{t_n}$. Recall Notation~\ref{not:ER_graph}.

\begin{definition}[Auxiliary L\'evy processes]\label{def:aux_levy_proc}
Let $\Aupper\coloneqq A_n(t_n)$, $\Alower \coloneqq A_n(t_n)-\lceil \lambda_n^{-2}\cdot \ln (1/\lambda_n)\rceil$, and let $I_n^S$ denote the interval $\left[t_n, \, t_n + \frac{S}{n \lambda_n^2}\right]$. We will later show that w.h.p.\ $A_n(t)$ stays in $[\Alower, \Aupper]$ as long as $t \in I_n^S$.

We define two integer-valued decreasing L\'evy processes $\overline{A}_n^{t_n}(\cdot)$ and $\underline{A}_n^{t_n} (\cdot)$ on the interval $I_n^S$ (we will later show that w.h.p.\ the inequalities $\overline{A}_n^{t_n}(t) \leq A_n(t) \leq \underline{A}_n^{t_n} (t)$ hold as long as $t \in I_n^S$, c.f.\ \eqref{eq:bounds_on_Y_t_n}).
\begin{itemize}
\item $\overline{A}_n^{t_n}(t_n)= A_n(t_n)$ and $ \underline{A}_n^{t_n}(t_n) = A_n(t_n)$,
\item $\underline{A}_n^{t_n} (\cdot)$ decreases with constant rate $\lambda_n\Alower$, $\overline{A}_n^{t_n}(\cdot)$ decreases with constant rate $\lambda_n \Aupper$,
\item the sizes of the decrements of $\underline{A}_n^{t_n}(\cdot)$ are i.i.d., the distribution of a decrement is the same as the distribution of the size of the connected component of a fixed vertex in $\mathcal{G}(\Alower,\, 1-\exp(-{t_n}/{n}))$,
\item the sizes of the decrements of $\overline{A}_n^{t_n}(\cdot)$ are i.i.d., the distribution of a decrement is the same as the distribution of the size of the connected component of a fixed vertex in $\mathcal{G}\mleft(\Aupper,\, 1-\exp \mleft(-\frac{t_n+\frac{S}{n\lambda_n^2}}{n} \mright) \mright)$.
\end{itemize}
Note that $\overline{A}_n^{t_n}(\cdot)$ and $\underline{A}_n^{t_n} (\cdot)$ are actually only L\'evy processes once we condition on the value of $A_n(t_n)$.

We also define two increasing (conditionally) L\'evy processes for $s \in [0,S]$
\begin{align}
\underline{Y}_n^{t_n}(s) &\coloneqq -\left[ \underline{A}_n^{t_n} \mleft(t_n+\frac{s}{n\lambda_n^2} \mright) - \underline{A}_n^{t_n}(t_n) \right] \cdot \lambda_n^2, \label{eq:Y_underline_def}\\
\overline{Y}_n^{t_n}(s) &\coloneqq -\left[ \overline{A}_n^{t_n} \mleft(t_n+\frac{s}{n\lambda_n^2} \mright) - \overline{A}_n^{t_n}(t_n) \right] \cdot \lambda_n^2.\label{eq:Y_overline_def}
\end{align}
\end{definition}

\needspace{3\baselineskip}
\begin{lemma}[Properties of the auxiliary L\'evy processes]\label{lem:Y_under_over_prop}
\mbox{}
\begin{enumerate}
\item With high probability, the processes $\underline{Y}_n^{t_n}$ and $\overline{Y}_n^{t_n}$ (defined in \eqref{eq:Y_underline_def} and~\eqref{eq:Y_overline_def}) squeeze the process $Y_n^{t_n}$ (defined in~\eqref{eq:defnY_n(s)}), i.e., for any $s \in [0,S]$, we have
\begin{equation}\label{eq:bounds_on_Y_t_n}
\lim_{n \to \infty} \pr\mleft( \, \underline{Y}_n^{t_n}(s) \le Y_n^{t_n}(s) \le \overline{Y}_n^{t_n}(s) \, \mright)=1. 
\end{equation}

\item As $n \to \infty$, the finite dimensional distributions of both $\underline{Y}_n^{t_n}$ and $\overline{Y}_n^{t_n}$ converge weakly to the finite dimensional distributions of $Y^t$ (defined in \eqref{eq:defnY(s)}). 
\end{enumerate}
\end{lemma}

We prove Lemma~\ref{lem:Y_under_over_prop} in Section~\ref{subsec:aux_Levy_processes}. Before that, we deduce Proposition~\ref{prop:Yn_convergence} from it.

\begin{proof}[Proof of Proposition~\ref{prop:Yn_convergence}]
Lemma~\ref{lem:Y_under_over_prop} implies that the finite dimensional distributions of $Y_n^{t_n}$ converge weakly to the finite dimensional distributions of $Y^t$. Therefore, by Theorem~\ref{thm:convergence_in_D} and Lemma~\ref{lem:tightness_Yn}, we obtain that~${Y_n^{t_n}(\cdot) \Rightarrow Y^t(\cdot)}$.
\end{proof}

\subsection{Properties of the auxiliary L\'evy processes}\label{subsec:aux_Levy_processes}

In this section our goal is to prove Lemma~\ref{lem:Y_under_over_prop}.

Recall that we defined $z_n^A(t)$ in \eqref{eq:z_n_t_A_def} for given $n \in \mathbb{N}$, $t \in \mathbb{R}_+$ and $A \in \left\{ 1, \ldots, n \right\}$. Recall also that we have fixed $T \in (1, +\infty)$ and the notation $\ev_{A,n,t}$ from Definition~\ref{not:ER_notation_A_n_t}.

\begin{lemma}[Limit of a Laplace transform]\label{lem:levy_laplace_infgen_ingredient}
For any $\gamma \geq 0$, we have
\begin{equation}\label{eq:levy_laplace_infgen_ingredient}
\lim \limits_{s \to \infty} \, \lim_{n \to \infty} \, \sup_{t_n^c + s \cdot \lambda_n \leq t \leq T} \, \sup_{A \,: \, |z_n^A(t)| \leq \kappa_n }
\left| \frac{ \ev_{A,n,t}\mleft( \exp(-\gamma \lambda_n^2 \mathscr{C})-1 \mright)}{\lambda_n} -\left( t-\sqrt{t^2+2\gamma} \right) \right|=0.
\end{equation}
\end{lemma}

Before we prove Lemma~\ref{lem:levy_laplace_infgen_ingredient}, let us use it to prove Lemma~\ref{lem:Y_under_over_prop}.

\begin{proof}[Proof of Lemma~\ref{lem:Y_under_over_prop}]

Recall that $I_n^S$ denotes the interval $\left[t_n,\, t_n + \frac{S}{n\lambda_n^2}\right]$.
Observe that $A_n (\cdot) \le \Aupper$ on $I_n^S$. Moreover, by the definition of $\Alower$ (see Definition~\ref{def:aux_levy_proc}), Lemma~\ref{lem:bnd_decrement}, also using that $\lambda_n \to 0$, we obtain that with high probability $\Alower \le A_n (\cdot)$ holds on $I_n^S$.

By Claim~\ref{cl:time_dep_inf_gen} and Definition~\ref{def:aux_levy_proc} it follows that $\underline{A}_n^{t_n}(\cdot)$ decreases with smaller rate than $A_n(\cdot)$, and $A_n(\cdot)$ decreases with smaller rate than $\overline{A}_n^{t_n}(\cdot)$ on the interval $I_n^S$. Moreover, the distribution of the decrements of~$\underline{A}_n^{t_n}(\cdot)$ is stochastically dominated by the distribution of the decrements of $A_n(\cdot)$ at any point of the interval~$I_n^S$, which, in turn, is stochastically dominated by the distribution of the decrements of $\overline{A}_n^{t_n}(\cdot)$. Therefore, we can couple~$\underline{A}_n^{t_n}(\cdot)$,~$A_n(\cdot)$ and~$\overline{A}_n^{t_n}(\cdot)$ on the interval $I_n^S$ in a way that for any $s \in [0,S]$, we have
\begin{equation}\label{eq:alive_vert_coupled}
\lim_{n \to \infty}\pr\mleft( \underline{A}_n^{t_n} \mleft(t_n+\frac{s}{n\lambda_n^2} \mright) \ge A_n \mleft(t_n+\frac{s}{n\lambda_n^2} \mright) \ge
\overline{A}_n^{t_n} \mleft(t_n+\frac{s}{n\lambda_n^2} \mright) \mright)=1.
\end{equation}
By~\eqref{eq:Y_underline_def} and~\eqref{eq:Y_overline_def}, \eqref{eq:alive_vert_coupled} implies the inequalities~\eqref{eq:bounds_on_Y_t_n}. 

Now we prove the convergence of finite dimensional distributions. Since $\underline{Y}_n^{t_n}$ and $\overline{Y}_n^{t_n}$ have independent increments, it is enough to prove that for any fixed $s \in [0,S]$, we have~${\underline{Y}_n^{t_n}(s) \Rightarrow Y^t(s)}$ and $\overline{Y}_n^{t_n}(s) \Rightarrow Y^t(s)$. Since $\underline{Y}_n^{t_n}$, $\overline{Y}_n^{t_n}$ and $Y^t$ are nonnegative random variables, it is enough to prove the convergence of the Laplace transforms, i.e., we only need to show that for any $\gamma > 0$ we have
\begin{equation}\label{eq:Laplace_transforms_converge}
\ev \mleft( \mathrm{e}^{-\gamma \underline{Y}_n^{t_n}(s)} \mright) \to
\ev \mleft( \mathrm{e}^{-\gamma Y^t(s)} \mright)
\quand
\ev \mleft( \mathrm{e}^{-\gamma \overline{Y}_n^{t_n}(s)} \mright) \to
\ev \mleft( \mathrm{e}^{-\gamma Y^t(s)} \mright)
\quad \text{as} \quad
n \to \infty.
\end{equation}
We will use Lemma~\ref{lem:levy_laplace_infgen_ingredient} to prove~\eqref{eq:Laplace_transforms_converge}. In order to do so, we note that it follows from Lemma~\ref{lem:Z_bound} and the definitions of~$\kappa_n$ (cf.~\eqref{eq:kappa_threshold_assumption}), $z_n^A(t)$ (cf.~\eqref{eq:z_n_t_A_def}), the formula~\eqref{eq:w_ODE} for $\dot{w}_n(t)$ together with $g_n(t) \approx t$ (cf.~\eqref{eq:gnt_asymp_t}) and the definitions of $I_n^S$, and $\Alower$, $\Aupper$ that w.h.p.\ $|z_n^A(t)| \leq \kappa_n$ holds for all $t \in I_n^S $ and all $A \in [\Alower, \Aupper]$.

By Definition~\ref{def:aux_levy_proc}, $\underline{Y}_n^{t_n}(s)$ is the sum of $\text{POI}\mleft( \frac{s\Alower}{n\lambda_n} \mright)$ independent terms that have the same distribution as~${\lambda_n^2 \cdot \mathscr{C}}$, where $\mathscr{C}$ denotes the size of the connected component of a fixed vertex in $\mathcal{G}(\Alower,\, 1-\exp(-{t_n}/{n}))$. Therefore,
\begin{equation}\label{eq:Y_underline_Laplace}
\ev \mleft( \mathrm{e}^{-\gamma \underline{Y}_n^{t_n}(s)} \, \big| \, A_n(t_n) \mright) = \exp \mleft[ \frac{s\Alower}{n\lambda_n} \cdot \left( \ev_{\Alower,n,t_n} \mleft( \mathrm{e}^{-\gamma \lambda_n^2 \cdot \mathscr{C}} \mright) - 1 \right) \mright].
\end{equation}

Similarly, we have
\begin{equation}\label{eq:Y_overline_Laplace}
\ev \mleft( \mathrm{e}^{-\gamma \overline{Y}_n^{t_n}(s)}\, \big| \, A_n(t_n) \mright) = \exp \mleft[ \frac{s \Aupper}{n\lambda_n} \cdot \left( \ev_{\Aupper,n,t_n+\frac{S}{n\lambda_n^2}} \left( \mathrm{e}^{-\gamma \lambda_n^2 \cdot \mathscr{C}} \right) - 1 \right) \mright].
\end{equation}

Using Campbell's formula for Poisson point processes (see e.g.\ \cite[Section~3.2]{K92}) in the equation marked by $(*)$ below, we obtain
\begin{equation}\label{eq:Y_Laplace}
\ev \mleft( \mathrm{e}^{-\gamma Y^t(s)} \mright)
\stackrel{(*)}{=}
\exp \mleft[\, \int \limits_0^s \int \limits_0^\infty \left( \mathrm{e}^{-\gamma \cdot \frac{y}{t^2}}-1\right) \cdot \frac{1}{\sqrt{2\pi}}\cdot y^{-3/2} \cdot \mathrm{e}^{-y/2} \, \mathrm{d}y\, \mathrm{d}u \mright]
= \exp \mleft[s \cdot \left(1- \sqrt{1+\frac{2\gamma}{t^2}} \right) \mright].
\end{equation}

Observing that $t_n \to t$, $\frac{S}{n\lambda_n^2} \to 0$ and ${\Alower}/{n} \to 1/t$ and ${\Aupper}/{n} \to 1/t$ in probability as $n \to \infty$ (see Theorem~\ref{thm:window_critical} and~\eqref{eq:subcrit_lambda_assumption}), we obtain that \eqref{eq:Laplace_transforms_converge} follows by~\eqref{eq:levy_laplace_infgen_ingredient}, \eqref{eq:Y_underline_Laplace}, \eqref{eq:Y_overline_Laplace} and~\eqref{eq:Y_Laplace}.
\end{proof}

\subsection{Limit of the Laplace transform}\label{subsec:lemma_levy_laplace}

In this section we prove Lemma~\ref{lem:levy_laplace_infgen_ingredient}. We will deduce it from the next lemma.

\begin{lemma}[Alternative form of Laplace transform limit]\label{lem:levy_laplace_infgen_alt}
For any $\beta \geq 0$ and for any choice of the sequences $(t_n)_{n=1}^{\infty}$, $(\alpha_n)_{n=1}^{\infty}$ and $(A_n)_{n=1}^{\infty}$ satisfying
\begin{equation}\label{eq:t_n_alpha_n_A_n_conditions}
\frac{t_n - t_n^c}{\lambda_n} \to \infty \text{ as } n \to \infty, \qquad \forall \, n \in \mathbb{N} \; : \; t_n \leq T, \quad -\beta \leq \alpha_n \leq 0, \quad \left|z_n^{A_n}(t_n)\right| \leq \kappa_n,
\end{equation}
we have
\begin{equation}\label{eq:levy_laplace_infgen_ingredient2}
\lim_{n \to \infty}
\left( \frac{1}{\lambda_n} \ev_{A_n,n,t_n}\mleft[ \exp \mleft( t_n^2 \left(\alpha_n -\frac{\alpha_n^2}{2} \right) \lambda_n^2 \mathscr{C} \mright) -1 \mright] - \alpha_n t_n \right) =0.
\end{equation}
\end{lemma}

Before we prove Lemma~\ref{lem:levy_laplace_infgen_alt}, let us deduce Lemma~\ref{lem:levy_laplace_infgen_ingredient} from it.

\begin{proof}[Proof of Lemma~\ref{lem:levy_laplace_infgen_ingredient}]
Let us fix $\gamma \geq 0$. For each $n \in \mathbb{N}$, choose $t_n$ and $A_n$ in a way that the supremum that appears on the left-hand-side of~\eqref{eq:levy_laplace_infgen_ingredient} is achieved. Now in order to apply Lemma~\ref{lem:levy_laplace_infgen_alt}, we choose $\beta=\sqrt{1+2 \gamma }$ and $\alpha_n=1-\sqrt{1+2 \gamma /t_n^2}$ so that $t_n^2(\alpha_n -\frac{\alpha_n^2}{2} )=-\gamma$. With this choice we have $\alpha_n t_n=t_n-\sqrt{t_n^2+2\gamma}$, therefore, Lemma~\ref{lem:levy_laplace_infgen_alt} indeed implies Lemma~\ref{lem:levy_laplace_infgen_ingredient}.
\end{proof}

The rest of Section~\ref{subsec:lemma_levy_laplace} is devoted to the proof of Lemma~\ref{lem:levy_laplace_infgen_alt}. Let us fix $\beta \geq 0$ and sequences $(t_n)$, $(\alpha_n)$ and $(A_n)$ satisfying \eqref{eq:t_n_alpha_n_A_n_conditions} for the rest of Section~\ref{subsec:lemma_levy_laplace}. Recall the asymptotic relations from~\eqref{eq:tA_n_asymp} and~\eqref{eq:tA_n_asymp_stronger}. The key idea of the proof of Lemma~\ref{lem:levy_laplace_infgen_alt} is to use the moment-generating formula~\eqref{eq:magic_A_n_t}. In order to shorten our formulas, we define
\begin{equation}\label{eq:j_n_def}
j_n \coloneqq \lceil n \lambda_n \alpha_n \rceil.
\end{equation}

Note that $j_n \leq 0$ by~\eqref{eq:t_n_alpha_n_A_n_conditions}. Next we define the error terms $\Phi_n(\cdot)$ and $\Psi_n(\cdot)$.
\begin{notation}
We also define $\Phi_n(k)$ and $\Psi_n(k)$ for any $k \in\{ 1, \ldots, A_n\}$ by
\begin{align}
\Phi_n(k)\coloneqq & \prod_{i=0}^{k-1} \mathrm{e}^{-t_n j_n/n}\left( 1+ \frac{j_n}{A_n-i} \right) -
\left[ \mathrm{e}^{-t_n j_n/n}\left( 1+ \frac{j_n}{A_n} \right) \right]^{k}, \label{eq:Phi_n_k_def}\\
\Psi_n(k)\coloneqq & \left[ \mathrm{e}^{-t_n j_n/n}\left( 1+ \frac{j_n}{A_n} \right) \right]^{k} - \exp \mleft[ t_n^2 \left(\alpha_n -\frac{\alpha_n^2}{2} \right) \lambda_n^2 k \mright]. \label{eq:Psi_n_k_def}
\end{align}
\end{notation}

\begin{lemma}[$\Phi_n$ and $\Psi_n$ vanish]\label{lem:Phi_Psi_goes_to_zero} We have
\begin{align}
\lim_{n \to \infty} \frac{1}{\lambda_n} \ev_{A_n,n,t_n}\mleft[ \Phi_n( \mathscr{C} ) \mright] = & 0, \label{eq:Phi_n_goes_to_zero}\\
\lim_{n \to \infty} \frac{1}{\lambda_n} \ev_{A_n,n,t_n}\mleft[ \Psi_n( \mathscr{C} ) \mright] = & 0. \label{eq:Psi_n_goes_to_zero}
\end{align}
\end{lemma}

Before we prove Lemma~\ref{lem:Phi_Psi_goes_to_zero}, let us deduce Lemma~\ref{lem:levy_laplace_infgen_alt} from it.
\begin{proof}[Proof of Lemma~\ref{lem:levy_laplace_infgen_alt}]
Recall from~\eqref{eq:magic_A_n_t} that for each $n \in \mathbb{N}$, we have
\begin{equation}\label{eq:magic_n}
\frac{1}{\lambda_n} \ev_{A_n,n,t_n}\mleft[ \prod_{i=0}^{\mathscr{C}-1} \mathrm{e}^{-t_n j_n/n}\left( 1+ \frac{j_n}{A_n-i} \right)-1 \mright]= \frac{j_n}{\lambda_n A_n}.
\end{equation}
The left-hand side of~\eqref{eq:magic_n} can be asymptotically replaced with the similar term of~\eqref{eq:levy_laplace_infgen_ingredient2} by Lemma~\ref{lem:Phi_Psi_goes_to_zero}, while the right-hand side of~\eqref{eq:magic_n} can be asymptotically replaced by $\alpha_n t_n$, because
\begin{equation*}
\left| \frac{j_n}{ \lambda_n A_n} -\alpha_n t_n \right| \stackrel{\eqref{eq:tA_n_asymp}\mathrm{(c)},\, \eqref{eq:j_n_def} }{=} 
\left| \alpha_n \cdot \left( \frac{n }{ A_n} - t_n \right) +\mathcal{O}\mleft( \frac{1}{ n \lambda_n } \mright) \right| \stackrel{\eqref{eq:subcrit_lambda_assumption}, \, \eqref{eq:t_n_alpha_n_A_n_conditions},\, \eqref{eq:tA_n_asymp}\mathrm{(b)}}{=} o(1).
\end{equation*} 
These observations together imply~\eqref{eq:levy_laplace_infgen_ingredient2}.
\end{proof}

It remains to prove Lemma~\ref{lem:Phi_Psi_goes_to_zero}. The following claim and lemma will be useful in the proof.

\begin{claim}[Bounded terms]
We have
\begin{equation}\label{eq:bounded_by_one}
\mathrm{e}^{-t_n j_n/n}\left( 1+ \frac{j_n}{A_n} \right) \leq 1.
\end{equation}
\end{claim}

\begin{proof}
Note that we have $1/A_n-t_n/n \ge 0$ for large enough $n$ by~\eqref{eq:tA_n_asymp_stronger}. It implies \eqref{eq:bounded_by_one}:
\begin{equation*}
\mathrm{e}^{-t_n j_n/n}\left( 1+ \frac{j_n}{A_n} \right) \leq \mathrm{e}^{-t_n j_n/n} \mathrm{e}^{ j_n /A_n }= \mathrm{e}^{j_n( 1/A_n -t_n/n )} \leq 1.\qedhere
\end{equation*}
\end{proof}

\begin{lemma}[Upper bound on $\Phi_n$]\label{lem:removing_i}
For large enough $n$ and for any $0 \leq k \leq A_n/2$, we have
\begin{equation}\label{eq:removing_i}
\left| \Phi_n(k) \right| \leq 4 \frac{\lambda_n}{n} \beta T^2 k^2.
\end{equation}
\end{lemma}

\begin{proof}
We will use the following inequality
\begin{equation}\label{eq:prod_sum_ineq}
\text{if } \quad 0\leq a_i, b_i \leq 1, \quad \text{ then } \quad \left| \prod_{i=0}^{k-1} a_i - \prod_{i=0}^{k-1} b_i \right| \leq \sum_{i=0}^{k-1} |a_i-b_i|.
\end{equation}

By~\eqref{eq:Phi_n_k_def}, \eqref{eq:bounded_by_one} and~\eqref{eq:prod_sum_ineq} the l.h.s.\ of~\eqref{eq:removing_i} can be upper bounded by
\begin{equation*}
\sum_{i=0}^{k-1} \left| \mathrm{e}^{-t_n j_n/n}\left( \frac{j_n}{A_n} - \frac{j_n}{A_n-i} \right) \right| \stackrel{(*)}{\leq} 2|j_n| \sum_{i=0}^{k-1} \frac{ i}{|A_n-i|A_n }
\stackrel{(**)}{\leq} 4 \frac{\lambda_n}{n} \beta T^2 k^2,
\end{equation*}
where $(*)$ holds for large enough $n$ by~\eqref{eq:subcrit_lambda_assumption} and~\eqref{eq:j_n_def}, and $(**)$ holds for large enough $n$ by our assumption
$0 \leq k \leq A_n/2 $ as well as \eqref{eq:t_n_alpha_n_A_n_conditions}, \eqref{eq:tA_n_asymp}(b) and~\eqref{eq:j_n_def}. This proves Lemma~\ref{lem:removing_i}.
\end{proof}

Now we prove the two limits of Lemma~\ref{lem:Phi_Psi_goes_to_zero}.

\begin{proof}[Proof of~\eqref{eq:Phi_n_goes_to_zero}]
First, note that
\begin{equation}\label{eq:truncated_Phi_n_expect_goes_to_zero}
\frac{\ev_{A_n,n,t_n}\mleft[ |\Phi_n(\mathscr{C})| \, \Big| \, \mathscr{C} \leq A_n/2 \mright]}{\lambda_n}
\stackrel{\eqref{eq:removing_i} }{\leq} \frac{4\beta T^2}{n} \ev_{A_n,n,t_n}\mleft[ \mathscr{C}^2 \mright]
\stackrel{\eqref{eq:general_moment_of_component_size_crude},\, \eqref{eq:w_geq_1},\, \eqref{eq:t_n_alpha_n_A_n_conditions} }{\leq} \frac{4\beta T^2 \widetilde{C}_2}{n \lambda_n^3} \stackrel{\eqref{eq:subcrit_lambda_assumption} }{\ll} 1.
\end{equation}
Next, note that $|\Phi_n(k)|\leq 1$ follows from~\eqref{eq:Phi_n_k_def} and~\eqref{eq:bounded_by_one}, thus
\begin{align}
\begin{split}\label{eq:leftover_Y_n_goes_to_zero}
\frac{\ev_{A_n,n,t_n}\mleft[ |\Phi_n(\mathscr{C})| \, \Big| \, \mathscr{C} > A_n/2 \mright]}{\lambda_n} \leq&
\frac{\pr_{A_n,n,t_n}[\mathscr{C} > A_n/2]}{\lambda_n} \stackrel{(*)}{\leq} \frac{\ev_{A_n,n,t_n}[\mathscr{C}] }{ (A_n/2) \lambda_n} \stackrel{\eqref{eq:general_moment_of_component_size_crude},\, \eqref{eq:w_geq_1},\, \eqref{eq:t_n_alpha_n_A_n_conditions} }{\leq} \\
& \frac{ \widetilde{C}_1 \lambda_n^{-1} }{ (A_n/2) \lambda_n }
 \stackrel{\eqref{eq:tA_n_asymp}\mathrm{(c)}}{\asymp} \frac{ 1 }{ n \lambda_n^2 }
 \stackrel{\eqref{eq:subcrit_lambda_assumption} }{\ll} 1,
\end{split}
\end{align}
where at $(*)$ we used Markov's inequality. Putting together \eqref{eq:truncated_Phi_n_expect_goes_to_zero} and~\eqref{eq:leftover_Y_n_goes_to_zero}, we obtain~\eqref{eq:Phi_n_goes_to_zero}.
\end{proof}

Before we prove~\eqref{eq:Psi_n_goes_to_zero}, let us state a technical lemma.
\begin{lemma}[Technical lemma]\label{lem:technical_fs}
We have
\begin{align}
\label{eq:first_order_cancellation_la_sq} & \lim_{n \to \infty } \frac{1}{\lambda^2_n} \left| -\frac{t_n j_n }{n} +\frac{j_n}{A_n} -t_n^2 \alpha_n \lambda_n^2 \right| =0, \\
\label{eq:second_order_cancellation_la_sq} & \lim_{n \to \infty } \frac{1}{\lambda^2_n} \left| \frac{j_n^2}{A_n^2} -t_n^2 \alpha_n^2 \lambda_n^2 \right| =0.
\end{align}
\end{lemma}
Before we prove Lemma~\ref{lem:technical_fs}, let us deduce \eqref{eq:Psi_n_goes_to_zero} from it. 

\begin{proof}[Proof of~\eqref{eq:Psi_n_goes_to_zero}]
Recalling the definition of $\Psi_n(k)$ from~\eqref{eq:Psi_n_k_def}, we first note that
\begin{align}
\begin{split}\label{eq:Psi_n_k_upper_bound}
|\Psi_n(k)| \stackrel{(*)}{\leq}
& k \cdot \left| - \frac{t_n j_n}{n} +
\ln \mleft( 1+\frac{j_n}{A_n} \mright) -
t_n^2 \cdot \left(\alpha_n -\frac{\alpha_n^2}{2} \right) \cdot \lambda_n^2 \right|
\stackrel{ (**) }{=} \\
&k \cdot \left| - \frac{t_n j_n}{n} + \frac{j_n}{A_n} -\frac{1}{2} \frac{j_n^2}{A_n^2} +
o\mleft( \frac{j_n^2}{A_n^2} \mright) -
t_n^2 \cdot \left(\alpha_n -\frac{\alpha_n^2}{2} \right) \cdot \lambda_n^2 \right| \stackrel{(***)}{=} k \cdot o(\lambda_n^2),
\end{split}
\end{align}
where in $(*)$ we used that $|\mathrm{e}^{x}-\mathrm{e}^{y}| \leq |x-y|$ if $x,y \leq 0$ together with~\eqref{eq:bounded_by_one}, in $(**)$ we used that for any~${|x|\leq \frac12}$, we have $\ln(1+x)=x-\frac{1}{2}x^2 +o(x^2)$, and in $(***)$ we used both statements of Lemma~\ref{lem:technical_fs} and~${j_n^2 / A_n^2=\mathcal{O}(\lambda_n^2)}$ (which follows from~\eqref{eq:t_n_alpha_n_A_n_conditions}, \eqref{eq:tA_n_asymp}(c) and~\eqref{eq:j_n_def}).
Using this, we obtain~\eqref{eq:Psi_n_goes_to_zero}:
\begin{equation*}
\frac{1}{\lambda_n} \ev_{A_n,n,t_n}\mleft[\, |\Psi_n( \mathscr{C} ) | \,\mright] \stackrel{\eqref{eq:Psi_n_k_upper_bound} }{\leq} \ev_{A_n,n,t_n}[\mathscr{C}] o(\lambda_n)
\stackrel{\eqref{eq:general_moment_of_component_size_crude},\, \eqref{eq:w_geq_1},\, \eqref{eq:t_n_alpha_n_A_n_conditions} }{=} o(1).\qedhere
\end{equation*}
\end{proof}

The proof of Lemma~\ref{lem:Phi_Psi_goes_to_zero} is complete (given the statement of Lemma~\ref{lem:technical_fs}). Therefore, it remains to show Lemma~\ref{lem:technical_fs}.

\begin{proof}[Proof of~\eqref{eq:first_order_cancellation_la_sq}]
We will use that $A_n \asymp n$ (cf.~\eqref{eq:tA_n_asymp}(c)) throughout.
\begin{align}
\begin{split}\label{eq:first_order_canc_proof}
\left| -\frac{t_n j_n }{n} +\frac{j_n}{A_n} -t_n^2 \alpha_n \lambda_n^2 \right| =
& \left|
\frac{j_n}{A_n} \left( (1-\lambda_n t_n) - \frac{t_n A_n }{n} \right) + \lambda_n t_n \cdot \left( \frac{j_n}{A_n} -t_n \alpha_n \lambda_n \right) \right| \stackrel{\eqref{eq:tA_n_asymp_stronger}, \, \eqref{eq:j_n_def}}{=} \\
&\frac{|j_n|}{A_n} o(\lambda_n) + \lambda_n t_n \left| \frac{n \lambda_n \alpha_n}{A_n} +\mathcal{O}\mleft( \frac{1}{n} \mright) -t_n \alpha_n \lambda_n \right| \stackrel{\eqref{eq:t_n_alpha_n_A_n_conditions},\, \eqref{eq:j_n_def} }{=} \\
&o(\lambda_n^2) + \lambda^2_n t_n \left| \alpha_n \cdot \left( \frac{n }{A_n}- t_n \right) +\mathcal{O}\mleft( \frac{1}{n \lambda_n } \mright) \right| \stackrel{(*)}{=} o(\lambda_n^2),
\end{split}
\end{align}
where in $(*)$ we used that $|n/A_n-t_n|=o(1)$ (cf.~\eqref{eq:tA_n_asymp}(b)) and that $n \lambda_n \gg 1$ (cf.~\eqref{eq:subcrit_lambda_assumption}) together with the assumptions \eqref{eq:t_n_alpha_n_A_n_conditions} on $t_n$ and~$\alpha_n$.
\end{proof}

\begin{proof}[Proof of~\eqref{eq:second_order_cancellation_la_sq}]
Using the definition \eqref{eq:j_n_def} of $j_n$, we have
\begin{equation*}
\left| \frac{j_n^2}{A_n^2} -t_n^2 \alpha_n^2 \lambda_n^2 \right| \stackrel{\eqref{eq:tA_n_asymp}\mathrm{(c)}}{=} 
\left| \frac{ (n \lambda_n \alpha_n)^2 }{A_n^2} + \mathcal{O}\mleft( \frac{n \lambda_n}{n^2} \mright) -t_n^2 \alpha_n^2 \lambda_n^2 \right| =
\lambda_n^2 \left| \alpha_n^2 \cdot \left( \frac{n^2}{A_n^2} -t_n^2 \right) + \mathcal{O}\mleft( \frac{ 1}{n \lambda_n } \mright) \right| \stackrel{ (**) }{=} o(\lambda_n^2),
\end{equation*}
where in $(**)$ we used similar considerations as in equation $(*)$ of~\eqref{eq:first_order_canc_proof}.
\end{proof}
\noindent The proof of Lemma~\ref{lem:technical_fs} and, therefore, the proof of Lemma~\ref{lem:levy_laplace_infgen_alt} is complete.

\appendix

\section{Detailed study of the window curve}\label{ap:window_curve_proofs}

In the Appendix we prove Claim~\ref{cl:tc_asymp} and the statements of Section~\ref{sec:ODE}, in particular Claim~\ref{cl:tc_less_than_t*} and Lemmas~\ref{lem:w_bounds},~\ref{lem:gn_properties}.

Recall the functions $w_n(t)$ defined in \eqref{eq:wn_def}, which solves the ODE \eqref{eq:w_ODE}, $v_n(t)$ defined as the solution of the ODE \eqref{eq:v_ODE}, its inverse $\mathrm{t}_n(v)$ (defined in \eqref{eq:t_def_as_v_inverse}) and $g_n(t)$ from~\eqref{eq:def_g_n_t}. In this section we omit this dependence on $n$ from our notation, hopefully without causing any confusion and use $w(t)=w_n(t)$, $v(t) = v_n(t)$, $\mathrm{t}(v) = \mathrm{t}_n(v)$ and $g(t) = g_n(t)$. We also use the notation $\lambda=\lambda_n$, $t^c = t_n^c$ (recall \eqref{eq:tnc_def}) and $t^* =t_n^*$ (recall \eqref{eq:t_star_def}).

First, we prove Lemma~\ref{lem:w_bounds}.

\begin{proof}[Proof of Lemma~\ref{lem:w_bounds}]
The function $w(t)$ is nonnegative by that fact that if $w(t) = 0$ for some $t>0$, then $\dot{w}(t) = 0$ (see~\eqref{eq:w_ODE}).

Observe that it follows from~\eqref{eq:v_ODE} that $\dot{v}(t) \leq 1$, therefore, $v(t) \leq t$. Now to show $w(t) \leq v(t)$, it is enough to show that if $v(t)=w(t)$ for some $t$, then $\dot{w}(t) \leq \dot{v}(t)$. This indeed follows by comparing \eqref{eq:w_ODE} and~\eqref{eq:v_ODE} and observing that $\frac{v(t)}{t} \leq 1$. This proves $w(t) \leq v(t)$ for all $t \geq 0$.
Finally, $v(t) \le \frac{1}{1+\lambda}$ follows from Claim~\ref{cl:vn_properties}.
\end{proof}

We will use the next lemma to prove Lemma~\ref{lem:gn_properties}.

\begin{proposition}[Self-organized $\lambda_n$-critical state of the window curve]\label{prop:soc_fixed_point}
Let $T > 1$ and $ \lambda \leq \frac{1}{4T^2} \wedge \frac{1}{100}$. Then for any $s \ge 6$, we have
\begin{equation}\label{eq:soc_fixed_point}
1- \lambda t - 4 \lambda \cdot \exp \mleft( -\frac{s-6}{20} \mright) \leq w(t) \leq 1- \lambda t + 2 T^2 \lambda^2 \quad
\text{for any} \quad t^c + s \cdot \lambda\leq t \leq T. 
\end{equation}
\end{proposition}

We prove Proposition~\ref{prop:soc_fixed_point} in Appendix~\ref{subap:soc_fixed_point}.

To prove Lemma~\ref{lem:gn_properties}, we also study the monotonicity of $g$. For that purpose, we introduce a notation.

\begin{definition}[Reaching $1-\lambda t$]
Let us define
\begin{equation}\label{eq:theta_def}
\theta \coloneqq \min \{\, t \; : \; w(t)=1-\lambda t \, \}.
\end{equation}
\end{definition}

\begin{lemma}[Monotonicity of $g$]\label{lem:g_monoton}
For any $\lambda < 1/T$, the function $g$ is monotone decreasing on the interval~${[0,\theta]}$ and monotone increasing on~${[\theta, T]}$.
\end{lemma}

We prove Lemma~\ref{lem:g_monoton} in Appendix~\ref{subap:soc_fixed_point}.

Now we can prove Lemma~\ref{lem:gn_properties}. Since in the statements of Lemma~\ref{lem:gn_properties} we take $n \to \infty$, here we indicate explicitly the dependence on $n$ in our notation.

\begin{proof}[Proof of Lemma~\ref{lem:gn_properties}]
Equation \eqref{eq:gnt_asymp_t} follows from~\eqref{eq:def_g_n_t} and~\eqref{eq:soc_fixed_point}.

Now we prove~\eqref{eq:gnt_bounds}. By the definition of $g_n(t)$, it follows from $w_n(t) \le \frac{1}{1+\lambda_n}$ (see~\eqref{eq:wnvnbound}) that $g_n(t) \ge \frac{1}{1+\lambda_n}$. It implies the lower bound in \eqref{eq:gnt_bounds}. On the other hand, by~\eqref{eq:tnc_def}, we have $g_n(t_n^c)=4$ and $\lim \limits g_n(T) = T$ as~${n \to \infty}$ by~\eqref{eq:gnt_asymp_t}. Therefore, Lemma~\ref{lem:g_monoton} implies the upper bound in~\eqref{eq:gnt_bounds}.
\end{proof}

\subsection{Self-organized \texorpdfstring{$\lambda_n$}{lambda-n}-critical state of the window curve}\label{subap:soc_fixed_point}

In this section our goal is to prove Proposition~\ref{prop:soc_fixed_point} and as a side product, we also obtain Claim~\ref{cl:tc_asymp}, Claim~\ref{cl:tc_less_than_t*} and Lemma~\ref{lem:g_monoton}. Before proving Proposition~\ref{prop:soc_fixed_point}, we define an auxiliary function $z(\cdot)$ (which will serve as a more delicate approximation of $w(\cdot)$ than $v(\cdot)$).

\begin{definition}[Solution of a linear ODE]
Given $v(\cdot)$, we define the function $z(t)$ that arises as the solution of the linear ODE
\begin{equation}\label{eq:z_ODE}
\dot{z}(t) = -\lambda \cdot \frac{z(t)}{1-v(t)} +\frac{z(t)}{t} , \quad z(0)=0, \quad \dot{z}(0)=1.
\end{equation}
\end{definition}

In the next lemma, we show that $z(\cdot)$ is a lower bound on $w(\cdot)$ and (given the function $v(\cdot)$) the ODE of~$z(\cdot)$ can be solved explicitly.

\begin{lemma}[Properties of $z$]
For any $t \ge 0$, we have
\begin{equation}\label{eq:z_less_than_w}
z(t) \leq w(t).
\end{equation}

On the other hand, we have the explicit formula
\begin{equation}\label{eq:z_solution_v}
z(t)= t \cdot \left[ 1 - (1+\lambda)v(t) \right]^{\frac{\lambda}{1+\lambda}}.
\end{equation}
\end{lemma}

\begin{proof}
From \eqref{eq:w_ODE} and~\eqref{eq:wnvnbound}, it follows that
\begin{equation*}
\dot{w}(t) \geq w(t) \cdot \left( \frac{-\lambda}{1-v(t)} +\frac{1}{t} \right).
\end{equation*}
Comparing this to $z(\cdot)$, which satisfies
\begin{equation*}
\dot{z}(t) \stackrel{\eqref{eq:z_ODE} }{=} z(t) \cdot \left( \frac{-\lambda}{1-v(t)} +\frac{1}{t} \right),
\end{equation*}
we see that we have
$z(t) \leq w(t)$ for any $t \geq 0$. This completes the proof of~\eqref{eq:z_less_than_w}.

On the other hand, it is easy to check that function $z(\cdot)$, defined by formula~\eqref{eq:z_solution_v}, solves~\eqref{eq:z_ODE}.
\end{proof}

\begin{lemma}[Difference of $v$ and $w$]
For any $0 \leq t \leq t^*$, we have 
\begin{equation}\label{eq:v_w_diff_bound}
0 \stackrel{\eqref{eq:wnvnbound}}{\leq} v(t)-w(t) \leq 2 \lambda.
\end{equation}
\end{lemma}

\begin{proof}
We have
\begin{align}
\begin{split}\label{eq:D_bound_by_z_and_v}
\dot{v}(t)-\dot{w}(t) \stackrel{\eqref{eq:w_ODE},\, \eqref{eq:v_ODE}}{=}& -\lambda \cdot \left( \frac{v(t)}{1-v(t)} - \frac{w(t)}{1-w(t)} \right)+
\left( 1- \frac{w(t)}{t} \right) \stackrel{\eqref{eq:wnvnbound}}{\leq} \\
 &1- \frac{w(t)}{t} \stackrel{\eqref{eq:z_less_than_w} }{\leq} 1- \frac{z(t)}{t}
\stackrel{\eqref{eq:z_solution_v} }{=} 1- \left[ 1 - (1+\lambda)v(t) \right]^{\frac{\lambda}{1+\lambda}}.
\end{split}
\end{align}

Then we have for any $t \leq t^*$
\begin{align*}
v(t) -w(t) \stackrel{\eqref{eq:D_bound_by_z_and_v} }{\leq} &
\int \limits_0^{t^*} 1- \left[ 1 - (1+\lambda)v(t) \right]^{\frac{\lambda}{1+\lambda}} \, \mathrm{d} t
\stackrel{\eqref{eq:v_dot_lower_bound} }{\leq} 2 \int \limits_0^{(1+2\lambda)^{-1}} 1- \left[ 1 - (1+\lambda)v \right]^{\frac{\lambda}{1+\lambda}} \, \mathrm{d} v \leq \\
&2 \int \limits_0^{(1+\lambda)^{-1}} 1- \left[ 1 - (1+\lambda)v \right]^{\frac{\lambda}{1+\lambda}} \, \mathrm{d}v=
2 \left( \frac{1}{1+\lambda}- \frac{1}{1+2\lambda} \right) \leq 2 \lambda.
\end{align*}
This proves~\eqref{eq:v_w_diff_bound}.
\end{proof}

\begin{claim}[Explicit formula for $\mathrm{t}(v)$]
Let $\mathrm{t}(\cdot)$ be defined by~\eqref{eq:t_def_as_v_inverse}. We have the explicit formula
\begin{equation} \label{eq:explicit_t_v}
\mathrm{t}(v)= \frac{v}{1+\lambda} - \frac{\lambda}{(1+\lambda)^2} \ln \mleft( 1- (1+\lambda)v \mright).
\end{equation}
\end{claim}

\begin{proof}
The formula~\eqref{eq:explicit_t_v} arises as the solution of~\eqref{eq:v_ODE} using separation of variables.
\end{proof}

\begin{claim}[Bounds on $t^*$]\label{cl:crude_arrival}
If $\lambda \leq \frac{1}{10}$, then we have
\begin{align}
&1-\lambda \ln(\lambda) - 4\lambda \leq t^* \leq 1-\lambda \ln(\lambda) + \lambda, \label{eq:bounds_on_arrival_time} \\
&w \mleft( t^* \mright) > 1-4 \lambda,\label{eq:crude_arrival_of_w}\\
&t^*-6\lambda \le t^c \le t^*.\label{eq:tc_t*_close}
\end{align}
\end{claim}

\begin{proof}
We have
\begin{equation*}
t^* \stackrel{\eqref{eq:t_star_def},\, \eqref{eq:explicit_t_v} }{=} \frac{1}{1+2\lambda} \cdot \frac{1}{1+\lambda} -\frac{\lambda}{(1+\lambda)^2}\ln(\lambda) + \frac{\lambda}{(1+\lambda)^2} \ln(1+2\lambda), 
\end{equation*}
from which, using $\lambda \leq {1}/{10}$, one can deduce that~\eqref{eq:bounds_on_arrival_time} holds.

Inequality~\eqref{eq:crude_arrival_of_w} follows by~\eqref{eq:v_w_diff_bound}:
\begin{equation*}
w \mleft( t^* \mright) \stackrel{\eqref{eq:v_w_diff_bound}}{\ge} v(t^*) - 2\lambda \stackrel{\eqref{eq:t_def_as_v_inverse},\, \eqref{eq:t_star_def} }{=} \frac{1}{1+2\lambda} -2\lambda > 1-4 \lambda.
\end{equation*}

Finally, observe that
\begin{equation}\label{eq:tc_lower_bound}
 w(t^*-6\lambda) \stackrel{\eqref{eq:wnvnbound}}{\le} v(t^*-6\lambda) \stackrel{\eqref{eq:v_dot_lower_bound}}{\le} v(t^*) - 3\lambda \stackrel{\eqref{eq:t_star_def}}{\le} 1-4\lambda.
\end{equation}
Therefore, by the definition of $t^c$ (see~\eqref{eq:tnc_def}), inequalities \eqref{eq:crude_arrival_of_w} and~\eqref{eq:tc_lower_bound} prove~\eqref{eq:tc_t*_close}.
\end{proof}

Note that Claim~\ref{cl:crude_arrival} immediately implies Claim~\ref{cl:tc_asymp} and Claim~\ref{cl:tc_less_than_t*}.

\begin{notation}
Let us define
\begin{equation}\label{eq:def_x}
x(t) \coloneqq \frac{w(t)-(1-\lambda t)}{\lambda}. 
\end{equation}
\end{notation}

Note that the function $x(t)$ depends on $n$ as well, but similarly to other notation in this section, we omit it
from our notation.

Recall the definition of~$\theta$ from~\eqref{eq:theta_def}. Observe that~$\theta$ is the first time when~$x(\cdot)$ becomes nonnegative and also note that $x(t) = t-g(t)$ (cf.~\eqref{eq:def_g_n_t}). Hence~\eqref{eq:gnt_asymp_t} is equivalent to $\lim \limits_{s \to \infty} \limsup \limits_{n \to \infty} \sup \limits_{t_n^c + s \cdot \lambda_n \leq t \leq T} |x(t)|=0$. 

\begin{claim}[Properties of $x$]
The function $x(t)$ defined in \eqref{eq:def_x} solves the ODE
\begin{equation}\label{eq:x_ode}
\dot{x}(t) = \frac{1}{1- \frac{x(t)}{t}} + \left( \frac{1}{\lambda} +x(t) \right) \cdot \frac{-x(t)}{t \cdot (t -x(t))}.
\end{equation}

If $\lambda \le {1}/{100}$, then
\begin{equation}\label{eq:x_t_star}
x(t^*) \geq -3,
\end{equation}
and
\begin{equation}\label{eq:x_nonneg_after_tau}
x(t) \geq 0 \quad \text{for any} \quad t \geq \theta. 
\end{equation}
\end{claim}

\begin{proof}
The fact that $x(t)$ solves~\eqref{eq:x_ode} follows from~\eqref{eq:w_ODE} and~\eqref{eq:def_x}.
The inequality~\eqref{eq:x_t_star} follows from~\eqref{eq:bounds_on_arrival_time}, \eqref{eq:crude_arrival_of_w} and~\eqref{eq:def_x}.
Finally, \eqref{eq:x_nonneg_after_tau} holds because the right-hand side of~\eqref{eq:x_ode} is equal to $1$ if $x(t)=0$, so this positive drift does not let $x(t)$ become negative for any $t \geq \theta$.
\end{proof}

Lemma~\ref{lem:g_monoton} also follows from observation~\eqref{eq:x_nonneg_after_tau}:

\begin{proof}[Proof of Lemma~\ref{lem:g_monoton}]
By~\eqref{eq:def_g_n_t} and~\eqref{eq:w_ODE}, we have $\dot{g}(t) = w(t) \cdot \left( \frac{1}{1-w(t)} - \frac{1}{\lambda t} \right)$. Therefore, for any $\lambda < \frac 1T$ and $t \in [0,T]$, $\dot{g}(t) > 0$ if and only if $w(t) > 1-\lambda t$. By~\eqref{eq:theta_def}, $w(t) < 1-\lambda t$ for any $t < \theta$ and by~\eqref{eq:x_nonneg_after_tau}, $w(t) \ge 1-\lambda t$ for any $t \in [\theta, T]$. It implies the statement.
\end{proof}

\begin{lemma}[Lower bound on $x$]
Recall the definition of $t^*$ from~\eqref{eq:t_star_def} and $\theta$ from~\eqref{eq:theta_def}. If $\lambda \leq {1}/{100}$, then for any $s \in \mathbb{R}_+$ and $t \in [t^*+s \cdot \lambda,\, \theta]$, we have 
\begin{equation}\label{eq:x_lower_bound}
 x(t) \ge -4\exp \mleft( -\frac{ s }{20} \mright).
\end{equation}
\end{lemma}

\begin{proof}
Note that by~\eqref{eq:bounds_on_arrival_time}, we have $1 \leq t^*$. For any $t^* \leq t \leq \theta$, we have $x(t)<0$. Recall from~\eqref{eq:x_t_star} that we also have $x(t^*) \geq -3$. Therefore, by~\eqref{eq:x_ode}, for any $t^* \leq t \leq \theta$, we have $\dot{x}(t) > 0$ and thus $ -3 \leq x(t)$. Therefore, we have
\begin{equation}\label{eq:ineq_1}
\frac{1}{1- \frac{x(t)}{t}} \geq \frac{1}{1- x(t)} \geq \frac{1}{4}, \qquad
\frac{1}{\lambda} +x(t) \geq \frac{1}{2 \lambda}.
\end{equation}

Moreover, for any $t^* \leq t \leq \theta \wedge 2 $, we have
\begin{equation}\label{eq:ineq_2}
\frac{-x(t)}{t \cdot (t -x(t))} \geq \frac{-x(t)}{2 \cdot (2+3)}.
\end{equation}
Thus, for any $t^* \leq t \leq \theta \wedge 2 $, we obtain
\begin{equation}\label{eq:x_lower_ode}
\dot{x}(t) \stackrel{\eqref{eq:x_ode},\, \eqref{eq:ineq_1},\, \eqref{eq:ineq_2} }{\geq} \frac{1}{4}- \frac{x(t)}{20 \lambda}, \qquad x(t^*) \geq -3.
\end{equation}
Therefore, for any $t^* \leq t \leq \theta \wedge 2$
\begin{equation*}
x(t) \stackrel{\eqref{eq:x_lower_ode}}{\geq} 5 \lambda +(-3 -5 \lambda)e^{-\frac{t-t^*}{20 \lambda}} \geq - 4 e^{-\frac{t-t^*}{20 \lambda}}.
\end{equation*}
Therefore, if $s \in \mathbb{R}_+$ and $t^*+s \cdot \lambda \leq \theta \wedge 2$, then
\begin{equation}\label{eq:x_lower_bound_after_t*}
 x(t^*+s \lambda) \ge -4 \exp \mleft(-\frac{s}{20} \mright).
\end{equation}
By the fact that $\dot{x}(t)>0$ for any $t^* \le t \le \theta$, inequality~\eqref{eq:x_lower_bound_after_t*} implies~\eqref{eq:x_lower_bound}.
\end{proof}

\begin{lemma}[$x$ is small after $\theta$]
Recall the definition of $\theta$ from~\eqref{eq:theta_def}. Let $T>1$. If $\lambda \leq {1}/{(4T^2)}$, then for any $\theta \leq t \leq T$, we have
\begin{equation}\label{eq:x_small}
x(t) \leq 2 \lambda T^2.
\end{equation}
\end{lemma}

\begin{proof}
We know that $x(\theta)=0$ and we also know from~\eqref{eq:x_nonneg_after_tau} that $x(t) \geq 0$ for any $t \geq \theta$. Thus for any $\theta \leq t < T$, we have
\begin{equation*}
\dot{x}(t) \stackrel{\eqref{eq:x_ode} }{ < } \frac{1}{1- \frac{x(t)}{t}} - \frac{1}{\lambda} \cdot \frac{x(t)}{T^2}
\end{equation*}
Now if $\lambda \leq {1}/(4T^2)$ and if we assume $x(t) = 2 \lambda T^2$, then $\dot{x}(t)<0$, so this negative drift does not let $x(t)$ become bigger than $2 \lambda T^2$ for $\theta \leq t < T$, i.e., \eqref{eq:x_small} holds.
\end{proof}

\begin{proof}[Proof of Proposition~\ref{prop:soc_fixed_point}]
First, note that $t^c \stackrel{\eqref{eq:tc_t*_close}}{\ge} t^*-6\lambda$. Therefore, for any $s \ge 6$, if $t \in [t^c + s \cdot \lambda, \, \theta]$, we have
\begin{equation}\label{eq:x_before_theta}
 -4 \exp \mleft(- \frac{s-6}{20} \mright) 
 \stackrel{\eqref{eq:x_lower_bound}}{\le} x(t) \stackrel{\eqref{eq:theta_def}}{\le} 0
\end{equation}

On the other hand, for any $t \in [\theta, \, T]$, we have
\begin{equation}\label{eq:x_after_theta}
0 \stackrel{\eqref{eq:theta_def}}{\le} x(t) \stackrel{\eqref{eq:x_small}}{\le} 2\lambda T^2.
\end{equation}
Now Proposition~\ref{prop:soc_fixed_point} follows from the definition~\eqref{eq:def_x} of $x(t)$, inequalities~\eqref{eq:x_before_theta} and~\eqref{eq:x_after_theta}.
\end{proof}

\bigskip
\noindent
\textbf{\large Acknowledgements.}
Bal\'azs R\'ath and M\'arton Sz\H{o}ke were partially supported by the grant NKFI-FK-142124 of NKFI (National Research, Development and Innovation Office), and by the ERC Synergy Grant No.\ 810115 - DYNASNET. Dominic Yeo was partially supported by EPSRC doctoral training grant EP/K503113 and ISF grant 1325/14.

\end{document}